\documentclass[12pt, oneside]{article}   	
\usepackage[letterpaper, margin=1in]{geometry}
\usepackage{graphicx}				
\usepackage{amssymb}
\usepackage{mathtools,amsmath}
\usepackage{amsthm}
\usepackage{amssymb}
\usepackage{mathrsfs}
\usepackage{epstopdf}
\usepackage[pdftex, pdfusetitle, plainpages=false,  bookmarks, bookmarksnumbered,colorlinks, linkcolor=blue, citecolor=red,filecolor=black, urlcolor=black]{hyperref}
\usepackage{setspace}
\usepackage[capitalise]{cleveref}
\usepackage[all]{xy}
\usepackage{graphicx}
\usepackage{subcaption}
\usepackage{bm}
\usepackage{enumitem}
\usepackage{MnSymbol}
\usepackage{color}
\usepackage{titling}
\usepackage{kbordermatrix}
\usepackage[dvipsnames]{xcolor}
\usepackage{float} 
\usepackage{caption}
\usepackage{subcaption}
\usepackage{cancel}
\usepackage{array}
\usepackage{multirow}
\usepackage[all]{xy}
\usepackage[pdftex]{hyperref}
\usepackage{manfnt}
\usepackage{eqparbox}

\theoremstyle{plain}
   \newtheorem{theorem}{Theorem}[section]
   \newtheorem{proposition}[theorem]{Proposition}
   \newtheorem{lemma}[theorem]{Lemma}
   \newtheorem{notation}[theorem]{Notation}
   \newtheorem{corollary}[theorem]{Corollary}

\theoremstyle{definition}
   \newtheorem{definition}{Definition}[section]
   \newtheorem{question}{Question}
   \newtheorem{example}{Example}[section] 
\theoremstyle{remark}
 \newtheorem{remark}{Remark}[section]
 
\newcommand{\R}{\mathbb{R}}
\newcommand{\C}{\mathbb{C}}
\newcommand{\Z}{\mathbb{Z}}
\newcommand{\A}{\mathscr{A}}
\newcommand{\LL}{\mathcal{L}}
\newcommand{\T}{\mathcal{T}}
\newcommand{\M}{\mathcal{M}}

\newcommand{\pp}{\mathbb{P}}
\newcommand{\conv}{conv}
\newcommand{\ch}{\mathrm{ch}}

\def\-{\raisebox{.75pt}{-}}
\usepackage{tikz}
\usepackage{tikzscale}
\usetikzlibrary{trees,arrows,positioning,shapes.geometric}
\usetikzlibrary{calc}
\usetikzlibrary{decorations.pathmorphing}

\newcommand{\xx}{1}
\newcommand{\yy}{1}

\newcommand\textoverset[3][]{\mathrel{\overset{{\eqmakebox[#1]{#2}}}{#3}}}

\title{Families of polytopes with rational linear precision in higher dimensions}
\author{Isobel Davies, Eliana Duarte, Irem Portakal, Miruna-\c Stefana Sorea}
\date{\today}

\begin{document}

\maketitle

\makeatletter
\def\blfootnote{\xdef\@thefnmark{}\@footnotetext}
\makeatother

\blfootnote{\emph{MSC2020 Subject Classification:} Primary: 52B20; Secondary: 14M25, 62R01, 65D17.}


\begin{abstract}
In this article we introduce a new family of lattice polytopes with rational linear precision.
For this purpose, we define a new class of discrete statistical models that we call multinomial staged tree models. We prove that these models have rational
maximum likelihood estimators (MLE) and give a criterion for these models to be log-linear.
Our main result is then obtained by applying Garcia-Puente and Sottile's theorem that establishes a correspondence between
polytopes with rational linear precision  and log-linear models with rational MLE. Throughout this article we also study the interplay between the primitive collections of the
normal fan of  a polytope with rational linear precision and  the shape of the Horn matrix of its corresponding statistical model. Finally, we investigate lattice polytopes arising from toric multinomial staged tree models, in terms of the combinatorics of their tree representations.
\end{abstract}
\section{Introduction}
In Geometric Modelling, pieces of parametrised curves and surfaces are used as building blocks to describe geometric shapes in 2D and 3D. Some of the most widely used parametric units for this purpose are  B\'{e}zier curves, triangular B\'{e}zier surfaces and tensor product surfaces. These pieces of curves and surfaces are constructed using a
set of polynomial blending functions defined on the convex hull of a set of points $\A$,
together with a set of control points.
Taking as inspiration the theory of toric varieties and the form of the blending functions for the previous examples, Krasauskas  introduced the  more general notion of a 
\emph{toric patch} whose domain is a lattice polytope $P \subseteq \R^d$ \cite{krasauskas2002}. The blending functions, $\{\beta_{w,m}:P\to \mathbb{R}\}_{m\in \A}$, of a toric patch are constructed from the set of
lattice points $\A:=P\cap \Z^d$ and a vector of positive weights $w$ associated to each point in $\A$.

A significant difference between an arbitrary toric patch and one of the triangular
or tensor product patches is that its blending functions do not necessarily satisfy
the property of \emph{linear precision}. A collection of blending functions $\{\beta_{m}:P\to \mathbb{R}\}_{m\in \A}$ has linear precision if for any affine function $\Lambda:\mathbb{R}^d \to \mathbb{R}$,
\[\Lambda(u)=\sum_{m\in \A}\Lambda(m)\beta_{m}(u), \text{ for all } u \in P. \]
Thus, linear precision is the ability of the  blending functions to
replicate affine functions and it is desirable from the practical standpoint \cite{Garcia2010}. 
To decide if the collection of  blending functions associated to $(P,w)$ has linear precision it is necessary and sufficient to
check that the identity $p=\sum_{m\in \A}\beta_{w,m}(p)m$ holds for all $p\in P$ \cite[Proposition 11]{Garcia2010},
in this case we say
the pair $(P,w)$ has \emph{strict linear precision}.
If there exist rational blending functions  $\{\hat{\beta}_{w,m}:P\to \mathbb{R}\}_{m\in \A}$ that are nonnegative on $P$,
form a partition of unity, parametrise the same variety $X_{\A,w}$ as the blending functions $\{\beta_{w,m}:P\to \mathbb{R}\}_{m\in \A}$, and also have linear precision, we say the pair $(P,w)$ has \emph{rational linear precision}. 
It is an open problem, motivated by Geometric Modelling, to characterise all pairs $(P,w)$ that have 
rational linear precision  in dimension $d\geq 3$ \cite{krasauskas2002,Clarke2020}. 
The classification
of all such pairs in dimension $d=2$ is given in \cite{Bothmer2010}. 

Garcia-Puente and Sottile studied the property of rational linear precision
for toric patches by associating a scaled projective toric variety $X_{\A,w}$ to
the pair $(P,w)$ \cite{Garcia2010}. The variety $X_{\A,w}$  is the image
of the map $[w\chi]_{\A}: (\C^*)^d\to \pp^{n-1}$ defined by
$\mathbf{t}\mapsto[w_1 \mathbf{t}^{m_1}:w_2 \mathbf{t}^{m_2}:\ldots: w_s \mathbf{t}^{m_n}]$ where $\A=\{m_1,\ldots,m_n\}$. One of their main results states that a pair $(P,w)$ has rational linear precision
if and only if the variety $X_{\A,w}$, seen as a discrete statistical model, has rational maximum likelihood estimator (MLE). This result establishes a communication channel between Geometric Modelling and Algebraic Statistics. Thus, it is natural to use ideas from Algebraic Statistics 
to study the property of rational linear precision.

Models with rational MLE are algebraic varieties that admit a  parametrisation known as Horn uniformisation \cite{Duarte2020,huh2014}.
This parametrisation depends on a \emph{Horn matrix} $H$ and a coefficient
for each column of $H$.
In their recent study of moment maps of toric varieties \cite{Clarke2020}, Clarke and Cox  go one step further in strengthening the relationship between pairs $(P,w)$ with rational linear precision and  models $X_{\A,w}$ with rational MLE by using Horn matrices
to characterise all pairs $(P,w)$ that have strict linear precision. They propose the use of Horn matrices to study
polytopes with rational linear precision and state several questions and conjectures about the relationship between the Horn matrix of $X_{\A,w}$ and the primitive collections of the normal fan of $P$.

In this article we study the property of rational linear precision of pairs $(P,w)$
from the point of view of Algebraic Statistics.
Our main contribution is Theorem~\ref{thm:main}, which introduces a new family of polytopes (with associated weights) that has rational linear precision. We construct this family from a subclass of discrete statistical models introduced in Section~\ref{multinomialstagedtreemodels} that we call \emph{multinomial staged trees}. Looking at
specific members of this family in 3D, we settle some of the questions raised in \cite{Clarke2020} related to Horn matrices and primitive collections.

This paper is structured as follows: In Sections~\ref{notation}-\ref{subsec:rational and mle} we provide background material on rational linear precision, discrete statistical models with rational MLE and Horn matrices. In Section~\ref{subsec:primitive} we state Questions~\ref{q1} and \ref{q2}  which guided our investigations related to  Horn matrices and primitive collections. These questions are followed by a quick outline referring to the places in this article where they are addressed.
 In  Section~\ref{sec:2Dand3Dpolytopes}, we characterise the shape of the Horn matrix for pairs $(P,w)$ in $2D$. We also present a family 
of pairs $(P,w)$ in $3D$ that has rational linear precision and explain  several aspects of this family that relate to Questions~\ref{q1} and \ref{q2}.
In Section~\ref{subsec:defmultitrees} we define multinomial staged tree models, 
we prove they have rational MLE in Section~\ref{subsec:rationalMLE}  and we
characterise the subclass of these models that are toric varieties in Section~\ref{sec:toricmultitrees}. These results lead to our main theorem, Theorem~\ref{thm:main}. Finally, in Section~\ref{sec:toricmultinomialtrees}, we show that the examples from Section~\ref{sec:2Dand3Dpolytopes} are all multinomial staged trees and prove our conjectures about the relationship between the combinatorics of the trees and  primitive collections.  

\section{Preliminaries}
\label{sec:preliminaries}
We assume the reader is familiar with introductory material on computational algebraic geometry and toric geometry at the level of \cite{cox2015} and \cite{cox2011}. 
\subsection{Notation and conventions}\label{notation}
We consider pairs $(P,w)$ where  $P$ is a $d$-dimensional lattice polytope in $\R^d$, $\Z^d$ is the fixed lattice, 
$\A=P\cap \Z^d=\{m_1,\ldots,m_n\}$ and $w$ is a vector of  positive weights indexed by $\A$.
Fix $n_1,\ldots, n_r$ to be the inward facing primitive normal vectors of $P$ corresponding to the
facets $F_1,\ldots, F_r$ of $P$ and let $a_1,\ldots,a_r$ be the corresponding
integer translates in the facet presentation of $P$ given by $ P =\{ p\in \R^d : \langle p,n_i
\rangle \geq -a_i ,\forall
i\in \{1,\ldots,r\} \} $. The lattice distance to the face $F_i$ evaluated at  $p\in \R^d$  is
\[
  h_{i}(p) = \langle p, n_i \rangle + a_i, \;\; i=1,\ldots, r,
\]
we record each of these values in
the vector $h(p)=(h_1(p),\ldots,h_r(p))$
The value $h_{i}(m_j)$ is the lattice distance from the $j$-th lattice point
to the $i$-th facet. The matrix with $ij$ entry equal to $h_{i}(m_j)$ 
is the \emph{lattice distance matrix} of $\A$. 
We will often consider products of linear forms or variables
whose exponents are given by vectors. For vectors
$v=(v_1,\ldots, v_N ), w =(w_1,\ldots, w_N)$ we use $v^w$  to denote the product $\prod_{i=1}^N v_i ^{w_i}$ and use the convention that
$0^0=1$. Common choices for $v, w$ in the upcoming sections are the vectors $\mathbf{t}=(t_1,\ldots, t_d)$, $h(p)$ and $h(m), m\in \A$.
If $P$ is a polytope and $a\geq 1$ is an integer, $aP$ denotes its dilation. 

\subsection{Rational linear precision}
In this section we follow closely the exposition in \cite{Clarke2020}. A more elementary introduction to this
topic is available in \cite[Chapter 3]{cox2020}.

\begin{definition} \label{def:patches}
Let $P\subseteq \R^d$ be a full dimensional polytope and let $w=(w_1,\ldots,w_n)$ be a vector
of positive weights. 
\begin{enumerate}
\item For $1\leq j \leq n$ and $p\in P$, 
$\beta_{j}(p):=h(p)^{h(m_j)}=\prod_{i=1}^r h_{i}(p)^{h_i(m_j)}.$
    \item The functions $\beta_{w,j}:=w_j \beta_{j}/\beta_{w}$ are the \emph{toric blending functions}
    of $(P,w)$, where $\beta_{w}(p):= \sum_{j=1}^n w_j \beta_{j}(p).$
    \item Given \emph{control points} $\{Q_{j}\}_{1\leq j \leq n} \in \R^\ell$,   the \emph{toric patch} $F:P\to \mathbb{R}^{\ell}$ is defined by
    \begin{equation}\label{eq:tautologicalpatch}
    p \mapsto \frac{1}{\beta_{w}(p)}\sum_{j=1}^n w_j \beta_{j}(p)Q_j.    
    \end{equation}
    
\end{enumerate}
\end{definition}
\noindent In part $(3)$ of the previous definition, it is natural to choose the set of control points to be  $\A$.
\begin{definition} \label{def:patches2}Let $(P,w)$ be as in Definition~\ref{def:patches}.
\begin{enumerate}
\item The \emph{tautological patch} $K_{w}\!:P\!\to\! P$ is the toric patch (\ref{eq:tautologicalpatch}) where $\{Q_j\!=\!m_j\}_{1\leq j\leq n}$.

    \item The pair $(P,w)$ has \emph{strict linear precision} if $K_{w}$ is the identity on $P$, that is \[ p = \frac{1}{\beta_{w}(p)}\sum_{j=1}^n w_j \beta_{j}(p)m_j, \text{  for all } p\in P.\]
    \item The pair $(P,w)$ has \emph{rational linear precison} if there are rational functions
    $\hat{\beta_1},\ldots, \hat{\beta_n}$ on $\C^d$
    satisfying: 
    \begin{enumerate}
        \item $\sum_{j=1}^n\hat{\beta}_j=1$ as rational functions on $\C^d$.
        \item The map $\hat{\beta}:\C^d \dasharrow X_{\A,w}\subset \pp^{n-1}$, $\mathbf{t}\to (\hat{\beta}_1(\mathbf{t}),
        \ldots, \hat{\beta}_n(\mathbf{t}))$ is a rational parametrisation of $X_{\A,w}$.
        \item For every $p\in P \subset \C^d$, $\hat{\beta}_j(p)$ is defined and is a nonnegative real number.
        \item $\sum_{j=1}^{n}\hat{\beta}_j(p)m_j=p$ for
        all $p\in P$.
    \end{enumerate}
\end{enumerate}
\end{definition}

\begin{remark}
We are interested in the property of 
linear precision. By \cite[Proposition 2.6]{Garcia2010}, the blending functions $\{\beta_{w,j}: 1\leq j\leq n \}$
have linear precision if and only if  the pair $(P,w)$ has strict linear precision.  Rational
 linear precision requires the existence of
 rational functions $\{\hat{\beta_j}:  P\to \R : 1\leq j\leq n\}$ that have strict linear precision,  and that are related to the blending functions of $(P,w)$ via $3(b)$ in Definition~\ref{def:patches2}.
\end{remark}

\begin{remark} \label{rmk:newtonpolynomial}
An alternative way to specify a pair $(P,w)$ is by using a homogeneous polynomial $F_{\A,w}$ whose dehomogenisation $f_{\A,w}= \sum_{j=1}^n w_j \mathbf{t}^{m_i}$ encodes the weights in the coefficients and the lattice points in $\A$ as exponents. We  use this notation in Section~\ref{sec:2Dand3Dpolytopes} to describe toric patches in  2D and 3D. 
\end{remark}

\begin{remark} \label{rmk:rationalpower}
If $(P,w)$ has 
rational linear precision then $(aP,\tilde{w})$, $a\geq 1$, also has this property where $\tilde{w}$ is the vector of coefficients of 
$(f_{\A,w})^a$. See \cite[Lemma 2.2]{Bothmer2010}.
\end{remark}
\begin{example} \label{ex:trapex}
Consider the trapezoid $P=\conv((0,0),(3,0),(1,2),(0,2))$, with ordered set
of lattice points  $\A$ and vector of weights $w$, given as follows:
\begin{align*}
\A&=\{(0,2),(1,2),(0,1),(1,1),(2,1),(0,0),(1,0),(2,0),(3,0)\} \\
w&= (1,1,2,4,2,1,3,3,1).
\end{align*}
The polynomial $f_{\A,w}(s,t)=(1+s)(1+s+t)^2$ encodes $(P,w)$. 
The lattice distance functions for the facets of  $P$ are:
\begin{align*}
    h_1(s,t) &= s, & h_2(s,t)&=t, & h_3(s,t)&=3-t-s,  &h_4(s,t)&=2-t.
\end{align*}
The toric blending functions for $(P,w)$ are
\[\beta_{w,(i,j)}(s,t)={2\choose j}{3-j \choose i}\frac{ s^i t^j (3-s-t)^{3-i-j}(2-t)^{2-j}}{6-4t+t^2}, \text{ where } (i,j)\in \A.\]
The pair $(P,w)$ does not have strict linear precision, but it has rational linear precision. By Proposition~\ref{prop:trapezoid2D} the parametrisation of the patch which has linear
precision is given by:
\[\hat{\beta}_{w,(i,j)}(s,t)={2\choose j}{3-j \choose i}\frac{ s^i t^j (3-s-t)^{3-i-j}(2-t)^{2-j}}{4(3-t)^{3-j}}, \text{ where } (i,j)\in \A.\]
\end{example}

\begin{example}
Let $\Delta_d=\{x\in \mathbb{R}^d: x_1+\dotsm+x_d \leq 1,\; x_i\geq 0\}$ be the standard simplex in $\mathbb{R}^d$ and $k\Delta_d$ be its dilation by the integer $k\geq 1$. To
a point $m=(a_1,\ldots,a_d)\in \A=k\Delta_d \cap \mathbb{Z}^d$ we associate the weight
\[w_{m}={k\choose m}={k\choose k-|m|,a_1,\ldots, a_d}, \text{ where } |m|=a_1+\cdots+a_d.\]
The pair $(k\Delta_d,w)$ has strict linear precision, see \cite[Example 4.7]{Garcia2010}. By \cite[Example 4.6]{Clarke2020}, the product of two pairs, $(P,w)$ and $(Q,\tilde{w})$, with strict linear precision also has
strict linear precision. Hence, \emph{the B\'ezier simploids} \cite{Derose93}, which are polytopes of the form $k_1\Delta_{d_1}\times \cdots \times k_r \Delta_{d_r}$ for positive integers $k_1,\ldots, k_{r},n_1,\ldots n_r$, have strict linear precision. Conjecture 4.8 in \cite{Clarke2020} states that these are the only polytopes with strict linear precision.
\end{example}

\subsection{Discrete statistical models with rational MLE}
A probability distribution of a discrete random variable $X$ with outcome space
 $\{1,\ldots,n\}$ is a vector $(p_1,\ldots,p_{n})\in \mathbb{R}^{n}$ where $p_i=P(X=i)$, $i\in \{1,\ldots,n\}$, $p_i\geq 0$ and $\sum_{i=1}^n p_i=1$. 
The open probability simplex \[\Delta_{n-1}^{\circ}=\{(p_1,\ldots,p_n)\in \mathbb{R}^{n}\,|\, p_i> 0 ,p_1+\cdots+p_n=1\}\]
consists of all  strictly positive probability distributions for a discrete random variable with $n$ outcomes. A discrete statistical model $\mathcal{M}$
is a subset  of $\Delta_{n-1}^{\circ}$.

Given a set $\mathcal{D}=\{X_1,\ldots, X_N\}$ of independent and identically distributed observations   of $X$, we let $u=(u_1,\ldots,u_n)$ be the vector where $u_i$ 
is the number of times the outcome $i$ appears in $\mathcal{D}$. 
The likelihood function $L(p,u): \mathcal{M}\to \mathbb{R}_{\geq 0}$
defined by $(p_1,\ldots, p_n)\mapsto \prod p_{i}^{u_i}$ records the probability of observing the set $\mathcal{D}$. The \emph{maximum likelihood estimator} (MLE) of the model $\mathcal{M}$ is the function $\Phi: \mathbb{R}^{n}
\to \mathcal{M}$ that sends each vector $(u_1,\ldots, u_n)$ to the maximiser of $L(p,u)$, i.e.
\[\Phi(u):= \mathrm{arg} \max L(p,u).\]
For arbitrary  $\mathcal{M}$, the problem of estimating  $\mathrm{arg} \max L(p,u)$ is a difficult one. 
However, for special families, such as discrete exponential families,
there are theorems that guarantee the existence and uniqueness of $\mathrm{arg} \max L(p,u)$ when $u$ has nonzero entries. We are interested in the
case where $\Phi$ is a rational function of $u$.
\begin{definition}
Let $\mathcal{M}$ be a discrete statistical model with MLE $\Phi: \R^{n}\to \mathcal{M}, u\mapsto \hat{p}$. 
The model $\mathcal{M}$ has \emph{rational MLE}
if the coordinate functions of $\Phi$ are rational functions in $u$.
\end{definition}

\begin{example} \label{ex:independence}
Consider the model $\M$ of two independent binary random variables $X,Y$, with outcome set $\{0,1\}$ and $p_{ij}=P(X=i,Y=j)$. This
model is the set of all points $(p_{00},p_{01},p_{10},
p_{11})$ in $\Delta^{\circ}_{3}$ that satisfy the equation $p_{00}p_{11}-p_{10}p_{01}=0$. The model has rational MLE $\Phi: \R^{4}\to
\M$ where \[(u_{00},u_{01},u_{10},u_{11})\mapsto \left( 
\frac{u_{0+}u_{+0}}{u_{++}^2},\frac{u_{0+}u_{+1}}{u_{++}^2},\frac{u_{1+}u_{+0}}{u_{++}^2}, \frac{u_{1+}u_{+1}}{u_{++}^2}\right)\]
and $u_{i+}= u_{i0}+u_{i1}, u_{+j}= u_{0j}+u_{1j},u_{++}=\sum_{i,j\in\{0,1\}}u_{ij}$.
\end{example}
\begin{definition}
A \emph{Horn matrix} is an integer matrix whose column sums are equal to zero. Given a Horn matrix $H$, with columns $h_1,\ldots, h_n$,  and a vector $\lambda\in \mathbb{R}^n$, the Horn parameterisation $\varphi_{(H,\lambda)}:\mathbb{R}^{n}\to \mathbb{R}^{n}$ is the rational map
 given by
\[u\mapsto (\lambda_1 (Hu)^{h_1},\lambda_{2}(Hu)^{h_2},\ldots, \lambda_n (Hu)^{h_n}).\]

\end{definition}
\begin{example}
The MLE $\Phi$ in Example~\ref{ex:independence} is given by a Horn parametrisation  $\varphi_{(H,\lambda)}$, where 
\[ u=\begin{pmatrix}
u_{00} \\ u_{01}\\ u_{10} \\ u_{11}
\end{pmatrix},\;\;
H =\begin{pmatrix}
1 & 1 & 0 & 0 \\
0& 0& 1& 1 \\
1& 0& 1& 0 \\
0& 1& 0& 1 \\
-2&-2&-2&-2
\end{pmatrix}, \;\;
Hu= 
\begin{pmatrix}
u_{0+}\\
u_{1+}\\
u_{+0}\\
u_{+1}\\
-2u_{++}\\
\end{pmatrix},\;\; \text{ and }
\lambda= (4,4,4,4).
\]
\end{example}

\begin{definition}
We say that $(H,\lambda)$ is a Horn pair if: (1) the sum of the coordinates of $\varphi_{(H,\lambda)}$ as rational functions in $u$ is equal to $1$ and (2) the map $\varphi_{(H,\lambda)}$ is defined for all positive vectors and it sends these to positive vectors in $\mathbb{R}^{r}$. 
\end{definition}

\begin{theorem}{\cite[Theorem 1]{Duarte2020}}\label{thm:horn}
A discrete statistical model $\M$ has rational MLE $\Phi$ if and only if there exists a Horn pair $(H,\lambda)$ such that 
    $\M$ is the image of the Horn parametrisation $\varphi_{(H,\lambda)}$ restricted to the
    open orthant $\R_{>0}^{n}$ and $\Phi = \varphi_{(H,\lambda)}$ on $\R_{>0}^{n}$. 
\end{theorem}
It is possible that two Horn parametrisations $\varphi_{(H,\lambda)}$ and $\varphi_{(\tilde{H},\tilde{\lambda})}$
are equal even if $H\neq \tilde{H}$ and $\lambda\neq \tilde{\lambda}$. A Horn matrix $H$ is \emph{minimal} if it has no zero rows and no two rows are linearly dependent. By \cite[Proposition 6.11]{Clarke2020}
there exists a unique, up to permutation of the rows, minimal Horn matrix that
defines $\varphi_{(H,\lambda)}$. Any other pair $(H,\lambda)$  that defines the same Horn parametrisation  may be transformed into one where $H$ is a minimal Horn matrix; this is done by adding collinear rows, deleting zero rows and adjusting the vector $\lambda$ accordingly, see \cite[Lemma 3]{Duarte2020}. 
We end this section by noting that  \cite[Proposition 23]{Duarte2020} states that if $(H,\lambda)$ is a minimal Horn pair, then  every row of $H$ has either all entries  greater than or equal zero or all entries
less than or equal to zero. We call the submatrix of  $H$ that consists of all rows with nonnegative entries, the positive part of $H$, and its complement the negative part of $H$.

\subsection{The links between Algebraic Statistics and Geometric Modelling} \label{subsec:rational and mle}
The links referred to in the title of this section are Theorem~\ref{thm:rational} and Theorem~\ref{thm:strictlinhorn}. 

Given a pair $(P,w)$, the  \emph{scaled projective toric variety} $X_{\A,w}$ is the image
of the map $[w\chi]_{\A}: (\C^*)^d\to \pp^{n-1}$ defined by
$\mathbf{t}\mapsto[w_1 \mathbf{t}^{m_1}:w_2 \mathbf{t}^{m_2}:\ldots: w_n \mathbf{t}^{m_n}]$. To consider
the maximum likelihood estimation problem in the realm of complex
algebraic geometry we consider the variety
$W=V(x_1\ldots x_n(x_1+\dotsm+x_n))\subset \pp^{n-1}$ and the
map \[X_{\A,w}\setminus W \to (\C^*)^n,\;\;[x_1: \ldots: x_n]\mapsto \frac{1}{x_1+\cdots
+x_n}(x_1,\ldots,x_n).\] The image of this map is closed and
denoted by $Y_{\A,w}$. We call  $Y_{\A,w}$ a \emph{scaled very affine toric variety}. The set $\M_{\A,w} = Y_{\A,w}\cap \R^{n}_{>0}$ is a subset of
the open simplex $\Delta_{n-1}^{\circ}$ and as such it is a statistical model.
This class of models, of the form $\M_{\A,w},$ are known as \emph{log-linear models}.

\begin{remark}
The variety $Y_{\A,w}$ admits two parameterisations, one by monomials and one by toric blending functions \cite[Proposition 5.2]{Clarke2020}. These
are
\begin{align}
\overline{w_{\chi_\A}}: \C^d \dasharrow Y_{\A,w}, & \;\;\;\;\mathbf{t} \mapsto \left( \frac{w_1 \mathbf{t}^{m_1}}{\sum_{j=1}^sw_j\mathbf{t}^{m_j}},\ldots, \frac{w_n \mathbf{t}^{m_n}}{\sum_{j=1}^sw_j\mathbf{t}^{m_j}}\right),\label{monomialparametrisation}\\
\overline{w_{\beta_{\A}}}: \C^d \dasharrow Y_{\A,w},& \;\;\;\;\mathbf{t} \mapsto \left( \frac{w_1\beta_1(\mathbf{t})}{\beta_{w}(\mathbf{t})},\ldots,\frac{w_n\beta_n(\mathbf{t})}{\beta_{w}(\mathbf{t})} \right).\nonumber
\end{align}
\end{remark}

We now consider the maximum likelihood estimation problem for log-linear models. Given a vector of counts $u$, we let $\overline{u}:=u/|u|\in \Delta_{n-1}^{\circ}$ be the empirical distribution, where $|u|= \sum u_j$. We define the tautological map $\tau_{\A}$ following the convention in \cite{Clarke2020},
\begin{align}
\tau_{\A}:\Delta_{n-1}^{\circ}\to P^{\circ},\qquad(\overline{u}_1,\ldots, \overline{u}_n)\mapsto \sum_{j=1}^{n}\overline{u}_jm_j.\label{tautologicalmap}    \end{align}
The maximum likelihood estimate of $\overline{u}$ for
the model $\M_{\A,w}$ exists and it is unique whenever all entries of $\overline{u}$ are positive. 
\begin{theorem}
\cite[Corollary 7.3.9]{Sullivant2019}
The maximum likelihood estimate in
$\M_{\A,w}$ for the empirical distribution  $\overline{u}\in \Delta_{n-1}^{\circ}$
is the unique point $\hat{p}\in \M_{\A,w}$ that satisfies  $\tau_{\A}(\hat{p})=\tau_{\A}(\overline{u})$.
\end{theorem}
In the Algebraic Statistics literature, models with rational MLE are also known as   models with \emph{maximum likelihood degree}  equal to $1$.
Even though the previous theorem guarantees the existence and uniqueness of the MLE, it is not true that every log-linear model has rational MLE. 
We refer the reader to \cite{Amendola2019} for several examples of log-linear models that do not have rational MLE, or equivalently for examples of models with maximum likelihood degree greater than $1$. We end this section by recalling two theorems that connect models
with rational MLE and pairs with rational linear precision.

\begin{theorem} \cite[Proposition 5.1]{Garcia2010}\label{thm:rational}
The pair $(P,w)$ has rational linear precision if and only if
the model $\M_{\A,w}$ has rational MLE.
\end{theorem}

\begin{theorem}\cite{Clarke2020}\label{thm:strictlinhorn}
Set $a_P := \sum_{i=1}^r a_i$ and $ \; n_{P}:= \sum_{i=1}^rn_i$. The following
are equivalent:
\begin{enumerate}
    \item The pair $(P,w)$ has strict linear precision.
    \item $n_P =0$ and $\beta_{w}(p)=\sum_{j=1}^{n}w_j \beta_{j}(p)
    =\sum_{j=1}^n w_j \prod_{i=1}^r h_{i}(p)^{h_{i}(m_j)}$
    is a nonzero constant $c$.
    \item $\M_{\A,w}$ has rational MLE with minimal Horn pair
    $(H,\lambda)$ given by
    \begin{equation*}
        H=\begin{pmatrix}
        h_1(m_1) & h_1(m_2) & \ldots & h_1(m_n) \\
        h_2(m_1) & h_2(m_2) & \ldots & h_2(m_n)   \\
        \vdots & \vdots &  & \vdots \\
        h_r(m_1) & h_r(m_2) & \ldots & h_{r}(m_n) \\
        -a_P & -a_P & \ldots &-a_P
        \end{pmatrix},\;\;\;\;
        \lambda_j = \frac{w_j}{c}(-a_P)^{a_{P}}
    \end{equation*}
\end{enumerate}
\end{theorem}

\subsection{Primitive collections and Horn pairs}\label{subsec:primitive}
The notion of primitive collections was first introduced by Batyrev in \cite{Batyrev1991} for a smooth and projective toric variety $X_{\Sigma_P}$ of the polytope $P$. It provides an elegant description of the nef cone for $X_{\Sigma_P}$. This result has been generalised to the simplicial case and the definition of primitive collections for the  non-simplicial case has been introduced in \cite{CoxRenesse2009}.  
\begin{definition}\label{def: primitive collection}
Let $\Sigma_P$ be a normal fan. For $\sigma\in \Sigma_P$, $\sigma(1)$ denotes the $1$-faces of $\sigma$. A subset $C \subseteq \Sigma{_P}(1)$ of $1$-faces of $\Sigma_P$ is called a \emph{primitive collection} if 
\begin{enumerate}
    \item $C \nsubseteq \sigma(1)$ for all $\sigma \in \Sigma_P$.
    \item For every proper subset $C' \subsetneq C$, there exists $\sigma \in \Sigma_P$ such that $C' \subseteq \sigma(1)$.
\end{enumerate}
In particular, if $\Sigma_P$ is simplicial, $C$ is a primitive collection if $C$ does not generate a cone of $\Sigma_P$ but every proper subset does.
\end{definition}
For strict linear precision, Theorem~\ref{thm:strictlinhorn} gives the minimal Horn pair based only on the lattice distance functions of the facets of the polytope.
The authors in \cite{Clarke2020} raise the question whether it is possible to obtain a similar description of minimal Horn pairs of polytopes with rational linear precision.

\begin{question}
\label{q1}
Is the positive part of the minimal Horn matrix of a pair $(P,w)$ with rational linear precision always equal to the lattice distance matrix of $\A$?
\end{question}
For pairs $(P,w)$ in 2D with rational linear precision,
and the family of prismatoids in Section~\ref{3Dmodels},
the answer to Question~\ref{q1} is affirmative, see Theorem~\ref{thm: 2D Horn matrix}, Proposition~\ref{thm:bigfamily}, and Appendix \ref{ap:table}.\\

In \cite{Clarke2020} there are two examples, one of a trapezoid \cite[Section 8.1]{Clarke2020} and one of a decomposable graphical model \cite[Section 8.3]{Clarke2020}, where the positive part of the Horn matrix is the lattice distance matrix of $\A$ and the negative rows are obtained via the primitive collections of the normal fan of $P$. These examples motivate the next definition
and Question~\ref{q2}:

\begin{definition}\label{def: matrix with ld and primitive}
To a pair $(P,w)$ we associate the matrix $M_{\A,\Sigma_P}$ which consists of the lattice distance matrix of $\A$, with $ij$-th
entry $h_{i}(m_j)$, together with 
negative rows given by summing the rows of the lattice distance functions $-h_i$, for which the facet normals $n_i$ belong to the same primitive collection of $\Sigma_P$. 
\end{definition}
\begin{question} \label{q2}
For a pair $(P,w)$ with rational linear precision is there a Horn pair $(H,\lambda)$ for which $H=M_{\A,\Sigma_P}$?
\end{question}
\noindent  For pairs $(P,w)$ in 2D with rational linear precision, the answer to Question~\ref{q2} is affirmative, see Theorem~\ref{thm: 2D Horn matrix}. 
For the family of prismatoids in Section~\ref{3Dmodels},
Question~\ref{q2} is affirmative only for certain subclasses, see Theorem~\ref{thm: 3D primitivecollection compatible}.
For an arbitrary pair $(P,w)$ with rational linear precision, the matrix $M_{\A,\Sigma_P}$ is not necessarily a Horn matrix, see  Section~\ref{subsec:nonsimple}. Even in the case that $M_{\A,\Sigma_P}$ is a Horn matrix, it does not necessarily give rise to a Horn pair for $(P,w)$, see Section~\ref{section:fewerfacets}. In Section~\ref{sec:2Dand3Dpolytopes} we see a number of special cases for which the answer to Question~\ref{q2} is affirmative. In Section~\ref{sec:toricmultinomialtrees}, we give a condition on $(P,w)$ which guarantees  the existence of a Horn pair $(H,\lambda)$ with $H=M_{\A,\Sigma_P}$. We also provide an explanation for the negative rows of the Horn matrix in the language of multinomial staged tree models - introduced in Section~\ref{multinomialstagedtreemodels}.   
\section{Examples of Horn pairs in 2D and 3D}\label{sec:2Dand3Dpolytopes}
In this section we present families of 2D and 3D pairs $(P,w)$ with rational linear precision and explore the connection
between the geometry of the polytope and the shape of its corresponding Horn pair. Throughout this section we use $(s,t)$, respectively $(s,t,v)$ to denote $\mathbf{t}$ in the 2D, respectively 3D case.

\subsection{Toric surface patches and Horn pairs in 2D}\label{toricsurfacepatches}
By \cite{Bothmer2010}, the only 2D toric patches with rational linear precision are the Bézier triangles, tensor product patches and trapezoidal patches, seen in Figure~\ref{fig:2dpolygons}. 
This family of polygons, that we denote by $\mathcal{F}$, consists of all the Newton polytopes of
the polynomials
\begin{align*}
f_{a,b,d}(\mathbf{t}):=(1+s)^a((1+s)^d+t)^b \,\text{ for }\, a,b,d\geq 0.
\end{align*}
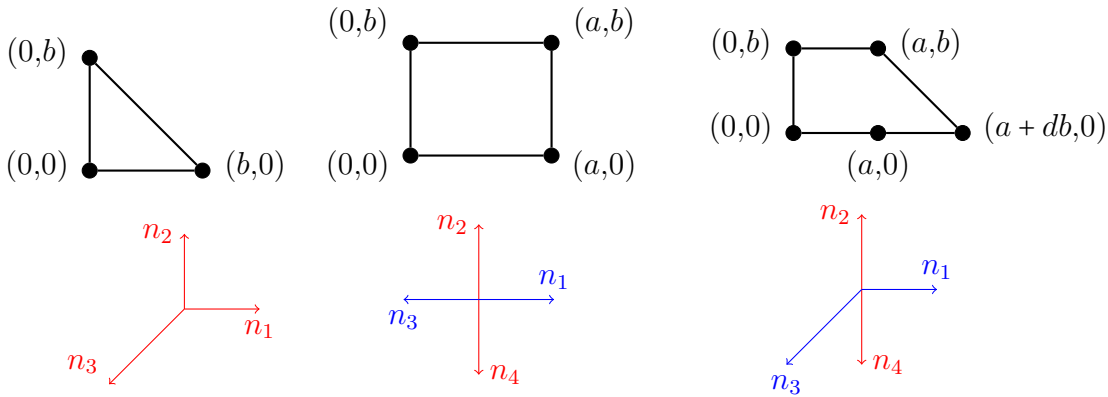
\begin{figure}[H]
\centering
\begin{subfigure}[b]{0.25\linewidth}
\begin{tikzpicture}
    \begin{scope}[scale=1.5]
  \draw (0,0) node[draw,circle,inner sep=2pt,fill, label={[xshift=-0.7cm,yshift=-0.4cm](0,0)}](2){};
  \draw (0,1) node[draw,circle,inner sep=2pt,fill, label={[xshift=-0.7cm,yshift=-0.4cm](0,$b$)}](1){};
  \draw (1,0) node[draw,circle,inner sep=2pt,fill, label ={[xshift=0.7cm,yshift=-0.4cm]($b$,0)}](3){};
  \draw[-] (1) edge[thick] (2);
  \draw[-] (2) edge[thick] (3);
  \draw[-] (1) edge[thick] (3);
  \end{scope}
 \end{tikzpicture}
\end{subfigure}
\begin{subfigure}[b]{0.3\linewidth}
\begin{tikzpicture}
\begin{scope}[scale=1.5]
    \draw (0,0.5) node[draw,circle,inner sep=2pt,fill, label={[xshift=-0.7cm,yshift=-0.6cm](0,0)}](4){};
  \draw (0,1.5) node[draw,circle,inner sep=2pt,fill, label={[xshift=-0.7cm,yshift=-0.2cm](0,$b$)}](1){};
  \draw (1.25,0.5) node[draw,circle,inner sep=2pt,fill, label ={[xshift=0.7cm, yshift=-0.6cm]($a$,0)}](3){};
  \draw (1.25,1.5) node[draw,circle,inner sep=2pt,fill, label ={[xshift=0.7cm,yshift=-0.2cm]($a$,$b$)}](2){};
  \draw[-] (1) edge[thick] (4);
  \draw[-] (2) edge[thick] (3);
  \draw[-] (2) edge[thick] (1);
  \draw[-] (3) edge[thick] (4);
  \end{scope}
 \end{tikzpicture}
\end{subfigure}
\begin{subfigure}[b]{0.4\linewidth}
\begin{tikzpicture}
\begin{scope}[scale=1.5]
\draw (1.5,0.75) node[draw,circle,inner sep=2pt,fill, label={[xshift=1.1cm,yshift=-0.4cm]($a+db$,0)}](5){};
    \draw (0,0.75) node[draw,circle,inner sep=2pt,fill, label={[xshift=-0.7cm,yshift=-0.4cm](0,0)}](4){};
  \draw (0,1.5) node[draw,circle,inner sep=2pt,fill, label={[xshift=-0.7cm,yshift=-0.4cm](0,$b$)}](1){};
  \draw (0.75,0.75) node[draw,circle,inner sep=2pt,fill, label ={[yshift=-0.9cm]($a$,0)}](3){};
  \draw (0.75,1.5) node[draw,circle,inner sep=2pt,fill, label ={[xshift=0.7cm,yshift=-0.4cm]($a$,$b$)}](2){};
  \draw[-] (1) edge[thick] (4);
  \draw[-] (4) edge[thick] (3);
  \draw[-] (2) edge[thick] (1);
  \draw[-] (3) edge[thick] (5);
  \draw[-] (2) edge[thick] (5);
  \end{scope}
 \end{tikzpicture} 
\end{subfigure}

\begin{subfigure}[p]{0.25\linewidth}
 \begin{tikzpicture}
  \draw[red,->](0cm,0cm)--(0cm,1cm)node[pos=1, left]{$n_2$};
  \draw[red,->](0cm,0cm)--(-1cm,-1cm)node[pos=1, above left]{$n_3$};
  \draw[red,->](0cm,0cm)--(1cm,0cm)node[pos=1, below]{$n_1$};
 \end{tikzpicture}
\end{subfigure}
\bigskip
\begin{subfigure}[p]{0.3\linewidth}
 \begin{tikzpicture}
  \draw[red,->](0cm,0cm) --(0cm,1cm)node[pos=1, left]{$n_2$};
  \draw[red,->](0cm,0cm) --(0cm,-1cm)node[pos=1, right]{$n_4$};
  \draw[blue,->](0cm,0cm) --(1cm,0cm)node[pos=1, above]{$n_1$};
    \draw[blue,->](0cm,0cm) --(-1cm,0cm)node[pos=1, below]{$n_3$};
\end{tikzpicture}
\end{subfigure}
\begin{subfigure}[p]{0.3\linewidth}
 \begin{tikzpicture}
  \draw[red,->](0cm,0cm) --(0cm,1cm)node[pos=1, left]{$n_2$};
  \draw[red,->](0cm,0cm) --(0cm,-1cm)node[pos=1, right]{$n_4$};
  \draw[blue,->](0cm,0cm) --(1cm,0cm)node[pos=1, above]{$n_1$};
    \draw[blue,->](0cm,0cm) --(-1cm,-1cm)node[pos=1, below]{$n_3$};
 \end{tikzpicture}
\end{subfigure}
\caption{
Left: B\'ezier triangles. Middle: Tensor product patches. Right: Trapezoids. The normal fan of each polygon
is displayed in the bottom row; two rays with the same colour are in the same primitive collection.}
 \label{fig:2dpolygons}
\end{figure}
 For general $a,b,d$, the Newton polytope associated to $f_{a,b,d}$, which we will denote by $T_{a,b,d}$, will be a trapezoidal patch, in the special cases $T_{0,b,1}=b\Delta_2$ and $T_{a,b,0}=a\Delta_1\times b\Delta_1$, we will have the more familiar Bézier triangles and tensor product patches. 
The lattice points in $T_{a,b,d}\cap \Z^{2}$ are
$
\A= \{(i,j): 0\leq j \leq b, 0\leq i\leq a+d(b-j)\}
$.  By Theorems~\ref{thm:horn} and \ref{thm:rational} we know that the statistical model associated to a pair in $\mathcal{F}$ admits a Horn pair. 
\begin{proposition}\label{prop:trapezoid2D}
A Horn pair $(H,\lambda)$ of a polygon in the family $\mathcal{F}$  is given by:
\begin{align*}
H&=\begin{pmatrix}
h_1(m_1) & \ldots & h_1(m) & \ldots & h_1(m_n)\\
h_2(m_1) & \ldots & h_2(m)  & \ldots & h_2(m_n)\\
h_3(m_1) & \ldots & h_3(m)  & \ldots & h_3(m_{n})\\
h_4(m_1) & \ldots & h_4(m)  & \ldots & h_4(m_{n})\\
-(h_1+h_3)(m_1) & \ldots & -(h_1+h_3)(m)  & \ldots & -(h_1+h_3)(m_{n})\\
-(h_2+h_4)(m_1) & \ldots & -(h_2+h_4)(m) & \ldots & -(h_2+h_4)(m_n)\\
\end{pmatrix},
\end{align*}
\begin{align*}
  \lambda_{m}&= (-1)^{a+d(b-j)+b}\binom{(h_2+h_4)(m)}{j}\binom{(h_1+h_3)(m)}{i}
 \end{align*}
where $m:=(i,j)\in\mathscr{A}$ is a general lattice point,  $m_1,\ldots, m_n$ is an ordered list of elements in $\A$,
$\mathbf{t}:=(s,t)$, and $h_{1},\ldots,h_{4}$ are
\begin{align*}
h_1(\mathbf{t})=s, \;\;h_2(\mathbf{t})=t, \;\;h_3(\mathbf{t})=a+db-s-dt, \;\;h_4(\mathbf{t})=b-t.  
\end{align*}
\end{proposition}
\begin{proof}
We use \cite[Proposition 8.4]{Clarke2020}.
The terms of the polynomial $f_{a,b,d}(\mathbf{t})$ 
 specify weights and lattice points in $T_{a,b,d}\cap \Z^2$.i.e.
 \begin{align*}
f_{a,b,d}(\mathbf{t})=\sum_{m\in\mathscr{A}}w_{m}\mathbf{t}^m,\qquad w_{m}=\binom{b}{j}\binom{a+db-dj}{i}.
 \end{align*}
 The monomial parametrisation (\ref{monomialparametrisation}) of $Y_{\A,w}$ is
\begin{align*}
\overline{w\chi_\mathscr{A}}(\mathbf{t})=\frac{1}{f_{a,b,d}(\mathbf{t})}(S_{m_1},\ldots,S_{m_n}),
\end{align*}
where $S_{m}=w_{m}\mathbf{t}^m$. Composing the monomial parametrisation with the tautological map (\ref{tautologicalmap}) gives the following birational map:
\begin{align*}
(\tau_\mathscr{A}\circ \overline{w\chi_\mathscr{A}})(\mathbf{t})&=\left(\dfrac{s((a+db)(1+s)^d + at)}{((1 + s)^d + t)(1 + s)},  \dfrac{tb}{((1+s)^d + t)}\right)   \\
&=\left(\dfrac{s((a+db)(1+s)^d + at)}{f_{1,1,d}(\mathbf{t})},  \dfrac{tb}{f_{0,1,d}(\mathbf{t})}\right)   
\end{align*}
with the following inverse:
\begin{align*}
\varphi(\mathbf{t})&=\left(\dfrac{s}{a+db-s-dt}, \dfrac{(a+db-dt)^d t}{(a+db-s-dt)^d(b-t)}\right) \\
&=\left(\dfrac{h_1(\mathbf{t})}{h_3(\mathbf{t})}, \dfrac{((h_1+h_3)(\mathbf{t}))^d h_2(\mathbf{t})}{(h_3(\mathbf{t}))^d h_4(\mathbf{t})}\right)
\end{align*}
The component of the monomial parametrisation corresponding to a lattice point $m$, composed with  $\varphi(\mathbf{t})$ is given by 
\begin{align} \nonumber
\left((\overline{w\chi_\mathscr{A}}\circ\varphi)(\mathbf{t})\right)_{m}&=\frac{S_{m}(\varphi(\mathbf{t}))}{f_{a,b,d}(\varphi(\mathbf{t}))}\\ \nonumber
&=\binom{b}{j}\binom{a+db-dj}{i}\dfrac{s^it^j(a+db-s-dt)^{a+db-i-dj}(b-t)^{b-j}}{(a+db-dt)^{a+db-dj}b^{b}}\\ 
&=w_{m}(-1)^{a+d(b-j)+b}h(\mathbf{t})^{h(m)}\label{eq:reparam}
\end{align}
where $h(q)=(h_1(q),h_2(q),h_3(q),h_4(q),
-h_5(q),-h_6(q))$ ($q\in\{\textbf{t},m\}$), $h_5=h_1+h_3$ and $h_6=h_2+h_4=b$.
It follows from \cite[Proposition 8.4]{Clarke2020} that the Horn parametrisation is
$(\overline{w\chi_\mathscr{A}}\circ\varphi)(p)$
where
\begin{align*}
p=\sum_{m\in\mathscr{A}}\frac{u_{m}}{u_+}(m),\qquad  u_{+}=\sum_{m\in\mathscr{A}}u_{m}.
\end{align*}
Therefore, the columns of the Horn matrix are  the exponents of 
\[
w_{m}(-1)^{a+d(b-j)+b}h(p)^{h(m)}, \;\;\; m\in \A,
\]
namely $h(m)$. 
It follows that 
\begin{align*}
\lambda_{m}&=(-1)^{a+d(b-j)+b}w_{m}\\
&=(-1)^{a+d(b-j)+b}\binom{(h_2+h_4)(m)}{j}\binom{(h_1+h_3)(m)}{i}
\end{align*}
\end{proof}

\begin{remark}
The  blending functions $\{\hat{\beta}_m: m\in \A \}$ for each pair $(P,w)$ in $\mathcal{F}$  that satisfy Definition~\ref{def:patches2} $(3)$ are given in equation~(\ref{eq:reparam}) in the previous proof. For the case
$a=d=1$ and $b=2$, these are written in Example~\ref{ex:trapex}.
\end{remark}
\begin{remark}
For general $a,b,d$, Proposition~\ref{prop:trapezoid2D} gives the minimal  Horn pair for $T_{a,b,d}$; this is not the case for  $T_{0,b,1}$ and $T_{a,b,0}$. For the last two cases,  the minimal Horn pair is obtained after row reduction operations or from Theorem \ref{thm:strictlinhorn}.
\end{remark}
Using Proposition~\ref{prop:trapezoid2D} and Theorem~\ref{thm:strictlinhorn}
we obtain an affirmative answer to Question~\ref{q1} for pairs $(P,w)$ in 2D.
A closer look at the primitive collections in Figure~\ref{fig:2dpolygons} also reveals an affirmative answer to Question~\ref{q2}.
This is contained in the next theorem.

\begin{theorem}\label{thm: 2D Horn matrix}
Every pair $(P,w)$ in 2D with rational linear precision has a Horn pair $(H,\lambda)$ with $H=M_{\A,\Sigma_P}$.
\end{theorem}
\begin{proof}
The normal fans of the polygons in $\mathcal{F}$ are
depicted in Figure~\ref{fig:2dpolygons}, in each subcase
the shape of the normal fan and its primitive collections are independent of the values of $a,b,d$.
The minimal Horn pair $(H,\lambda)$ for  the 2D simplex, $T_{0,b,1}=b\Delta_2$, given in  Theorem~\ref{thm:strictlinhorn} satisfies $H=M_{\A,\Sigma_P}$.
This follows because $b\Delta_2$ has one primitive collection,  $\{n_1,n_2,n_3\}$ and hence $M_{\A,\Sigma_P}$ has a single negative row.
For the tensor product patch $T_{a,b,0}=a\Delta_1 \times b\Delta_1$ and the general trapezoid $T_{a,b,d}$, the primitive
collections are $\{n_1,n_3\}$ and $\{n_2,n_4\}$. In these cases, the Horn pair $(H,\lambda)$ in Proposition~\ref{prop:trapezoid2D} 
satisfies $H=M_{\A,\Sigma_P}$.
\end{proof}

\subsection{A family of prismatoids with rational linear precision}\label{3Dmodels}
Unlike the 2D case, there is no classification for 3D lattice polytopes with rational linear precision. In this section we consider the family of prismatoids
\begin{align*}
    \mathcal{P}:=\{(P,w)&: P \text{ is the Newton polytope of }f_{\A,w}(\mathbf{t})=(f_{a,b,d}(\mathbf{t})+vf_{a',b',d}(\mathbf{t}))^l,\\
    & w \text{ is the vector of coefficients of } f_{\A,w}(\mathbf{t}),\\ &a,a',b,b',d,l\in\mathbb{Z}_{\geq 0} \text{ with }a'\leq a,\,b'\leq b\}.
\end{align*}
\begin{figure}[t]
\centering
  \begin{tikzpicture}
\begin{scope}[scale=6.5]
\node at (0.15,1.1) {$F_1$};
\node at (0.5,0.9) {$F_3$};
\node at (0.8,1.15) {$F_4$};
\node at (0.5,1.15) {$F_5$};
\node at (0.4,1.3) {$F_6$};
\draw (1.2,0.75) node[draw,circle,inner sep=2pt,fill, label={[xshift=1.4cm,yshift=-0.4cm]($(a+db)l$,0,0)}](5){};
\draw (0,0.75) node[draw,circle,inner sep=2pt,fill, label={[xshift=-0.7cm,yshift=-0.4cm](0,0,0)}](4){};
\draw (0,1.25) node[draw,circle,inner sep=2pt,fill, label={[xshift=-0.7cm,yshift=-0.4cm](0,0,$l$)}](6){};
\draw (0.25,1.35) node[draw,circle,inner sep=2pt,fill, label={[xshift=-0.9cm,yshift=-0.4cm](0,$b'l$,$l$)}](7){};
\draw (0.6,1.35) node[draw,circle,inner sep=2pt,fill, label={[xshift=-0.2cm,yshift=-0.1cm]($a'l$,$b'l$,$l$)}](8){};
\draw (0.4,1.25) node[draw,circle,inner sep=2pt,fill, label={[xshift=0.1cm,yshift=-0.8cm]($a'l$,0,$l$)}](9){};
\draw (0.8,1.25) node[draw,circle,inner sep=2pt,fill, label={[xshift=1.5cm,yshift=-0.5cm]($(a'+db')l$,0,$l$)}](10){};
\draw (0.25,1) node[draw,circle,inner sep=2pt,fill,label={[xshift=-0.8cm,yshift=-0.4cm](0,$bl$,0)}](1){};
\draw (0.65,0.75) node[draw,circle,inner sep=2pt,fill, label ={[yshift=-0.9cm]($al$,0,0)}](3){};
\draw (0.75,1) node[draw,circle,inner sep=2pt,fill, label ={[xshift=0.7cm,yshift=-0.1cm]($al$,$bl$,0)}](2){};
  \draw[-,dotted] (1) edge[thick] (4);
  \draw[-] (4) edge[thick] (3);
  \draw[-,dotted] (2) edge[thick] (1);
  \draw[-] (3) edge[thick] (5);
  \draw[-,dotted] (2) edge[thick] (5);
  \draw[-,dotted] (1) edge[thick] (7);
  \draw[-,dotted] (2) edge[thick] (8);
  \draw[-] (4) edge[thick] (6);
  \draw[-] (5) edge[thick] (10);
   \draw[-] (6) edge[thick] (7);
  \draw[-] (7) edge[thick] (8);
  \draw[-] (6) edge[thick] (9);
  \draw[-] (8) edge[thick] (10);
  \draw[-] (9) edge[thick] (10);
  \path[draw, fill=green!20,opacity=0.25] (0.25,1.35)--(0,1.25)--(0.8,1.25)--(0.6,1.35)--cycle;
  \path[draw, fill=green!20,opacity=0.5] (0,0.75)--(1.2,0.75)--(0.8,1.25)--(0,1.25)--cycle;
\draw (1.2,0.75) node[draw,circle,inner sep=2pt,fill, label={[xshift=1.4cm,yshift=-0.4cm]($(a+db)l$,0,0)}](5){};
\draw (0,0.75) node[draw,circle,inner sep=2pt,fill, label={[xshift=-0.7cm,yshift=-0.4cm](0,0,0)}](4){};
\draw (0,1.25) node[draw,circle,inner sep=2pt,fill, label={[xshift=-0.7cm,yshift=-0.4cm](0,0,$l$)}](6){};
\draw (0.25,1.35) node[draw,circle,inner sep=2pt,fill, label={[xshift=-0.9cm,yshift=-0.4cm](0,$b'l$,$l$)}](7){};
\draw (0.6,1.35) node[draw,circle,inner sep=2pt,fill, label={[xshift=-0.2cm,yshift=-0.1cm]($a'l$,$b'l$,$l$)}](8){};
\draw (0.4,1.25) node[draw,circle,inner sep=2pt,fill, label={[xshift=0.1cm,yshift=-0.8cm]($a'l$,0,$l$)}](9){};
\draw (0.8,1.25) node[draw,circle,inner sep=2pt,fill, label={[xshift=1.5cm,yshift=-0.5cm]($(a'+db')l$,0,$l$)}](10){};
\draw (0.65,0.75) node[draw,circle,inner sep=2pt,fill, label ={[yshift=-0.9cm]($al$,0,0)}](3){};

  \end{scope}
 \end{tikzpicture} 
 \caption{
 The general representative of a prismatoid in $\mathcal{P}$ is the
 convex hull of two trapezoids, $\conv(T_{a,b,d}\times\{0\}, T_{a',b',d}\times\{1\})$, dilated by $l$. For the labelling of facets, we refer to Notation~\ref{notation: labelling of the faces}.}\label{fig:bigfamily}
\end{figure}
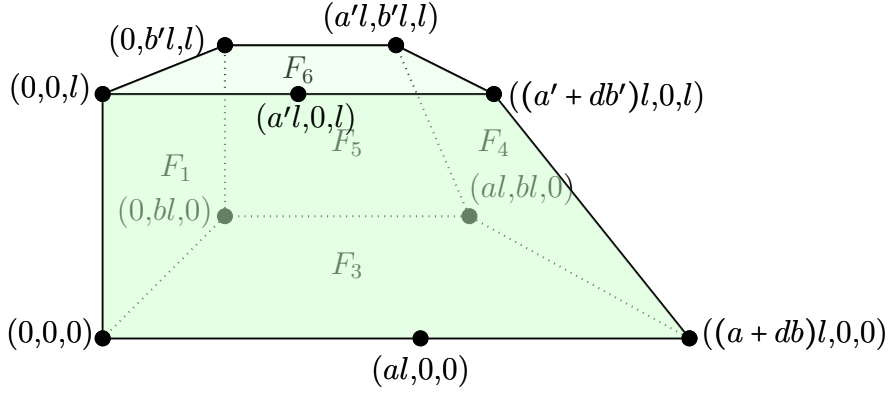
A general element of $\mathcal{P}$ is depicted in
Figure \ref{fig:bigfamily}, prismatoids for  different specialisations of
 $a,a',b,b',d$ are displayed in Table~\ref{table:3Dtrapezoids}. Note that some 3D B\'ezier simploids are also obtained by specialisation. Even though  Remark~\ref{rmk:rationalpower} says  it suffices to show that $P$ has rational linear precision for $l=1$,  we do not use this extra assumption.

\begin{table}
\begin{center}
\begin{tabular}{ |m{3.5cm}|m{4cm}| m{3.4cm}|m{4cm}|  }
 \hline
 \multicolumn{4}{|c|}{\textbf{ (A) Prismatoids with trapezoidal base} $a>0,b>0,d>0,l>0$} \\ \hline
 Trapezoidal \newline frusta  \newline $a'>0$, $b'>0$   & 
  \begin{tikzpicture}[every path/.style={>=latex},every node/.style={draw,circle,fill=black,scale=0.6}]
\begin{scope}[scale=3]
\draw (1.2,0.75) node[draw,circle,inner sep=2pt,fill](5){};
    \draw (0,0.75) node[draw,circle,inner sep=2pt,fill](4){};
    \draw (0,1.25) node[draw,circle,inner sep=2pt,fill](6){};
    \draw (0.25,1.35) node[draw,circle,inner sep=2pt,fill](7){};
     \draw (0.6,1.35) node[draw,circle,inner sep=2pt,fill](8){};
     \draw (0.4,1.25) node[draw,circle,inner sep=2pt,fill](9){};
     \draw (0.8,1.25) node[draw,circle,inner sep=2pt,fill](10){};
  \draw (0.3,1) node[draw,circle,inner sep=2pt,fill](1){};
  \draw (0.6,0.75) node[draw,circle,inner sep=2pt,fill](3){};
  \draw (0.75,1) node[draw,circle,inner sep=2pt,fill](2){};
  \draw[-,dotted] (1) edge[thick] (4);
  \draw[-] (4) edge[thick] (3);
  \draw[-,dotted] (2) edge[thick] (1);
  \draw[-] (3) edge[thick] (5);
  \draw[-,dotted] (2) edge[thick] (5);
  \draw[-,dotted] (1) edge[thick] (7);
  \draw[-,dotted] (2) edge[thick] (8);
  \draw[-] (4) edge[thick] (6);
  \draw[-] (5) edge[thick] (10);
   \draw[-] (6) edge[thick] (7);
  \draw[-] (7) edge[thick] (8);
  \draw[-] (6) edge[thick] (9);
  \draw[-] (8) edge[thick] (10);
  \draw[-] (9) edge[thick] (10);
  \end{scope}
 \end{tikzpicture} &
     Triangle top 
     \newline(simplex if $d=1$) \newline $a'=0$, $b'>0$ &   \begin{tikzpicture}[every path/.style={>=latex},every node/.style={draw,circle,fill=black,scale=0.6}]
\begin{scope}[scale=3]
\draw (1.1,0.75) node[draw,circle,inner sep=2pt,fill](5){};
    \draw (0,0.75) node[draw,circle,inner sep=2pt,fill](4){};
    \draw (0,1.25) node[draw,circle,inner sep=2pt,fill](6){};
    \draw (0.25,1.35) node[draw,circle,inner sep=2pt,fill](7){};
     \draw (0.3,1.25) node[draw,circle,inner sep=2pt,fill](9){};
  \draw (0.33,1) node[draw,circle,inner sep=2pt,fill](1){};
  \draw (0.35,0.75) node[draw,circle,inner sep=2pt,fill](3){};
  \draw (0.8,1) node[draw,circle,inner sep=2pt,fill](2){};
  \draw[-,dotted] (1) edge[thick] (4);
  \draw[-] (4) edge[thick] (3);
  \draw[-,dotted] (2) edge[thick] (1);
  \draw[-] (3) edge[thick] (5);
  \draw[-] (2) edge[thick] (5);
  \draw[-,dotted] (1) edge[thick] (7);
  \draw[-] (2) edge[thick] (7);
  \draw[-] (4) edge[thick] (6);
  \draw[-] (5) edge[thick] (9);
   \draw[-] (6) edge[thick] (7);
  \draw[-] (7) edge[thick] (9);
  \draw[-] (6) edge[thick] (9);
  \end{scope}
 \end{tikzpicture}  \\ \hline
      Trapezoidal \newline wedges  \newline $a'>0$, $b'=0$ &   \begin{tikzpicture}[every path/.style={>=latex},every node/.style={draw,circle,fill=black,scale=0.6}]
\begin{scope}[scale=3]
\draw (1.2,0.75) node[draw,circle,inner sep=2pt,fill](5){};
    \draw (0,0.75) node[draw,circle,inner sep=2pt,fill](4){};
    \draw (0,1.25) node[draw,circle,inner sep=2pt,fill](6){};
  \draw (0.3,1) node[draw,circle,inner sep=2pt,fill](1){};
  \draw (0.6,0.75) node[draw,circle,inner sep=2pt,fill](3){};
  \draw (0.6,1) node[draw,circle,inner sep=2pt,fill](2){};
  \draw (0.5,1.25) node[draw,circle,inner sep=2pt,fill](7){};
  \draw[-,dotted] (1) edge[thick] (4);
  \draw[-] (4) edge[thick] (3);
  \draw[-,dotted] (2) edge[thick] (1);
  \draw[-] (3) edge[thick] (5);
  \draw[-,dotted] (2) edge[thick] (5);
  \draw[-,dotted] (1) edge[thick] (6);
  \draw[-,dotted] (2) edge[thick] (7);
  \draw[-] (6) edge[thick] (7);
  \draw[-] (4) edge[thick] (6);
  \draw[-] (5) edge[thick] (7);
  \end{scope}
 \end{tikzpicture}&
     Trapezoidal \newline 
     Pyramids \newline $a'=0$, $b'=0$ &  \begin{tikzpicture}[every path/.style={>=latex},every node/.style={draw,circle,fill=black,scale=0.6}]
\begin{scope}[scale=3]
\draw (1.2,0.75) node[draw,circle,inner sep=2pt,fill](5){};
    \draw (0,0.75) node[draw,circle,inner sep=2pt,fill](4){};
    \draw (0,1.35) node[draw,circle,inner sep=2pt,fill](6){};
  \draw (0.15,1) node[draw,circle,inner sep=2pt,fill](1){};
  \draw (0.6,0.75) node[draw,circle,inner sep=2pt,fill](3){};
  \draw (0.55,1) node[draw,circle,inner sep=2pt,fill](2){};
  \draw[-,dotted] (1) edge[thick] (4);
  \draw[-] (4) edge[thick] (3);
  \draw[-,dotted] (2) edge[thick] (1);
  \draw[-] (3) edge[thick] (5);
  \draw[-,dotted] (2) edge[thick] (5);
  \draw[-,dotted] (1) edge[thick] (6);
  \draw[-,dotted] (2) edge[thick] (6);
  \draw[-] (3) edge[thick] (6);
  \draw[-] (4) edge[thick] (6);
  \draw[-] (5) edge[thick] (6);
  \end{scope}
 \end{tikzpicture}  \\ \hline
 \multicolumn{4}{|c|}{\textbf{(B) Prismatoids with tensor product base $a>0$, $b>0$, $d=0$, $l>0$} }\\ \hline
   Tensor \newline product \newline frusta \newline $a'>0,\,b'>0$ 
   \newline $\phantom{b'>0}$ & \begin{tikzpicture}[every path/.style={>=latex},every node/.style={draw,circle,fill=black,scale=0.6}]
\begin{scope}[scale=3]
    \draw (0,0.75) node[draw,circle,inner sep=2pt,fill](4){};
    \draw (0,1.25) node[draw,circle,inner sep=2pt,fill](6){};
    \draw (0.35,1.35) node[draw,circle,inner sep=2pt,fill](7){};
     \draw (0.75,1.35) node[draw,circle,inner sep=2pt,fill](8){};
     \draw (0.45,1.25) node[draw,circle,inner sep=2pt,fill](9){};
  \draw (0.3,1) node[draw,circle,inner sep=2pt,fill](1){};
  \draw (0.7,0.75) node[draw,circle,inner sep=2pt,fill](3){};
  \draw (1,1) node[draw,circle,inner sep=2pt,fill](2){};
  \draw[-,dotted] (1) edge[thick] (4);
  \draw[-] (4) edge[thick] (3);
  \draw[-,dotted] (2) edge[thick] (1);
  \draw[-] (2) edge[thick] (3);
  \draw[-,dotted] (1) edge[thick] (7);
  \draw[-] (2) edge[thick] (8);
  \draw[-] (4) edge[thick] (6);
   \draw[-] (6) edge[thick] (7);
  \draw[-] (7) edge[thick] (8);
  \draw[-] (6) edge[thick] (9);
  \draw[-] (8) edge[thick] (9);
  \draw[-] (3) edge[thick] (9);
  \end{scope}
 \end{tikzpicture} &
    Tensor \newline product \newline wedges \newline $a'=0$, $b'>0$    $\phantom{b'>0}$  &
    \multirow{2}{*}{ \begin{tikzpicture}[every path/.style={>=latex},every node/.style={draw,circle,fill=black,scale=0.6}]
\begin{scope}[scale=3]
    \draw (0,0.75) node[draw,circle,inner sep=2pt,fill](4){};
    \draw (0,1.25) node[draw,circle,inner sep=2pt,fill](6){};
    \draw (0.2,1.35) node[draw,circle,inner sep=2pt,fill](7){};
  \draw (0.2,1) node[draw,circle,inner sep=2pt,fill](1){};
  \draw (0.7,0.75) node[draw,circle,inner sep=2pt,fill](3){};
  \draw (1,1) node[draw,circle,inner sep=2pt,fill](2){};
  \draw[-,dotted] (1) edge[thick] (4);
  \draw[-] (4) edge[thick] (3);
  \draw[-,dotted] (2) edge[thick] (1);
  \draw[-] (2) edge[thick] (3);
  \draw[-,dotted] (1) edge[thick] (7);
  \draw[-] (2) edge[thick] (7);
  \draw[-] (4) edge[thick] (6);
   \draw[-] (6) edge[thick] (7);
  \draw[-] (3) edge[thick] (6);
  \end{scope}
 \end{tikzpicture}} \\
 \cline{1-3}
 
   Tensor \newline product \newline pyramids \newline $a'=0,\,b'=0$ \newline $\phantom{b'>0}$  & \begin{tikzpicture}[every path/.style={>=latex},every node/.style={draw,circle,fill=black,scale=0.6}]
\begin{scope}[scale=3]
    \draw (0,0.75) node[draw,circle,inner sep=2pt,fill](4){};
    \draw (0,1.25) node[draw,circle,inner sep=2pt,fill](6){};
  \draw (0.2,1) node[draw,circle,inner sep=2pt,fill](1){};
  \draw (0.8,0.75) node[draw,circle,inner sep=2pt,fill](3){};
  \draw (0.9,1) node[draw,circle,inner sep=2pt,fill](2){};
  \draw[-,dotted] (1) edge[thick] (4);
  \draw[-] (4) edge[thick] (3);
  \draw[-,dotted] (2) edge[thick] (1);
  \draw[-] (2) edge[thick] (3);
  \draw[-] (4) edge[thick] (6);
   \draw[-,dotted] (6) edge[thick] (1);
  \draw[-] (6) edge[thick] (2);
  \draw[-] (6) edge[thick] (3);
  \end{scope}
 \end{tikzpicture} &
 Tensor \newline product \newline wedges
 \newline $a'>0$, $b'=0$ &
 
\\
 \hline
 \multicolumn{4}{|c|}{\textbf{(C) Prismatoids with triangular base $a=0$, $b>0$, $d>0$, $l>0$} }\\ \hline
     Triangular frusta
     \newline (simplices if $d=1$)\newline $a'=0$, $b'>0$ & \begin{tikzpicture}[every path/.style={>=latex},every node/.style={draw,circle,fill=black,scale=0.6}]
\begin{scope}[scale=3]
    \draw (0,0.75) node[draw,circle,inner sep=2pt,fill](4){};
    \draw (0,1.25) node[draw,circle,inner sep=2pt,fill](6){};
    \draw (0.2,1.35) node[draw,circle,inner sep=2pt,fill](7){};
     \draw (0.6,1.25) node[draw,circle,inner sep=2pt,fill](9){};
  \draw (0.3,1) node[draw,circle,inner sep=2pt,fill](1){};
  \draw (0.9,0.75) node[draw,circle,inner sep=2pt,fill](3){};
  \draw[-,dotted] (1) edge[thick] (4);
  \draw[-] (4) edge[thick] (3);
  \draw[-,dotted] (1) edge[thick] (3);
  \draw[-,dotted] (1) edge[thick] (7);
  \draw[-] (4) edge[thick] (6);
   \draw[-] (6) edge[thick] (7);
  \draw[-] (7) edge[thick] (9);
  \draw[-] (6) edge[thick] (9);
  \draw[-] (3) edge[thick] (9);
  \end{scope}
 \end{tikzpicture} &
   Pyramid (simplex-based if $d=1$) \newline $a'=0$, $b'=0$  & \begin{tikzpicture}[every path/.style={>=latex},every node/.style={draw,circle,fill=black,scale=0.6}]
\begin{scope}[scale=3]
    \draw (0,0.75) node[draw,circle,inner sep=2pt,fill](4){};
    \draw (0,1.25) node[draw,circle,inner sep=2pt,fill](6){};
  \draw (0.3,1) node[draw,circle,inner sep=2pt,fill](1){};
  \draw (1.1,0.75) node[draw,circle,inner sep=2pt,fill](3){};
  \draw[-,dotted] (1) edge[thick] (4);
  \draw[-] (4) edge[thick] (3);
  \draw[-,dotted] (1) edge[thick] (3);
  \draw[-,dotted] (1) edge[thick] (6);
  \draw[-] (4) edge[thick] (6);
  \draw[-] (3) edge[thick] (6);
  \end{scope}
 \end{tikzpicture} \\
 \hline
 \multicolumn{4}{|c|}{ \textbf{ (D) 3D B\'ezier simploids} $l>0$ } \\ \hline
  3D Tensor Product  \newline  $a'=a>0$
  \newline $b'=b>0$\newline $d=0$ \newline $la\Delta_{1}\times lb\Delta_{1}\times l\Delta_{1}$ & \begin{tikzpicture}[every path/.style={>=latex},every node/.style={draw,circle,fill=black,scale=0.6}]
\begin{scope}[scale=3]
    \draw (0,0.75) node[draw,circle,inner sep=2pt,fill](4){};
    \draw (0,1.25) node[draw,circle,inner sep=2pt,fill](6){};
    \draw (0.3,1.5) node[draw,circle,inner sep=2pt,fill](7){};
     \draw (1,1.5) node[draw,circle,inner sep=2pt,fill](8){};
     \draw (0.7,1.25) node[draw,circle,inner sep=2pt,fill](9){};
  \draw (0.3,1) node[draw,circle,inner sep=2pt,fill](1){};
  \draw (0.7,0.75) node[draw,circle,inner sep=2pt,fill](3){};
  \draw (1,1) node[draw,circle,inner sep=2pt,fill](2){};
  \draw[-,dotted] (1) edge[thick] (4);
  \draw[-] (4) edge[thick] (3);
  \draw[-,dotted] (2) edge[thick] (1);
  \draw[-] (2) edge[thick] (3);
  \draw[-,dotted] (1) edge[thick] (7);
  \draw[-] (2) edge[thick] (8);
  \draw[-] (4) edge[thick] (6);
   \draw[-] (6) edge[thick] (7);
  \draw[-] (7) edge[thick] (8);
  \draw[-] (6) edge[thick] (9);
  \draw[-] (8) edge[thick] (9);
  \draw[-] (3) edge[thick] (9);
  \end{scope}
 \end{tikzpicture} &
    
      Triangular Prism \newline $a'=a=0$ \newline $b'= b>0$\newline$ d=1$ \newline $lb\Delta_{2}\times l\Delta_{1}$& \begin{tikzpicture}[every path/.style={>=latex},every node/.style={draw,circle,fill=black,scale=0.6}]
\begin{scope}[scale=3]
    \draw (0,0.75) node[draw,circle,inner sep=2pt,fill](4){};
    \draw (0,1.25) node[draw,circle,inner sep=2pt,fill](6){};
    \draw (0.3,1.5) node[draw,circle,inner sep=2pt,fill](7){};
     \draw (0.7,1.25) node[draw,circle,inner sep=2pt,fill](9){};
  \draw (0.3,1) node[draw,circle,inner sep=2pt,fill](1){};
  \draw (0.7,0.75) node[draw,circle,inner sep=2pt,fill](3){};
  \draw[-,dotted] (1) edge[thick] (4);
  \draw[-] (4) edge[thick] (3);
  \draw[-,dotted] (1) edge[thick] (3);
  \draw[-,dotted] (1) edge[thick] (7);
  \draw[-] (4) edge[thick] (6);
   \draw[-] (6) edge[thick] (7);
  \draw[-] (7) edge[thick] (9);
  \draw[-] (6) edge[thick] (9);
  \draw[-] (3) edge[thick] (9);
  \end{scope}
 \end{tikzpicture} \\
 \hline
   3D simplex \newline $a'=a=0$\newline $b'=0,b=1$\newline $ d=1$ \newline $l\Delta_{3}$& \begin{tikzpicture}[every path/.style={>=latex},every node/.style={draw,circle,fill=black,scale=0.6}]
\begin{scope}[scale=3]
    \draw (0,0.75) node[draw,circle,inner sep=2pt,fill](4){};
    \draw (0,1.25) node[draw,circle,inner sep=2pt,fill](6){};
  \draw (0.4,1.1) node[draw,circle,inner sep=2pt,fill](1){};
  \draw (0.75,0.75) node[draw,circle,inner sep=2pt,fill](3){};
  \draw[-,dotted] (1) edge[thick] (4);
  \draw[-] (4) edge[thick] (3);
  \draw[-] (1) edge[thick] (3);
  \draw[-] (1) edge[thick] (6);
  \draw[-] (4) edge[thick] (6);
  \draw[-] (3) edge[thick] (6);
  \end{scope}
 \end{tikzpicture} & \multicolumn{2}{c|}{} \\
 \hline
\end{tabular}
\end{center}
\caption{Representative members of $\mathcal{P}$. The coordinates of the vertices of each polytope in this table are obtained by specializing the parameters $a,a',b,b',d$ in the coordinates the vertices of the prismatoid in Figure~\ref{fig:bigfamily}.}
\label{table:3Dtrapezoids}
\end{table}

\begin{proposition}
\label{thm:bigfamily}
The pairs in $\mathcal{P}$ have rational linear precision with a
Horn pair $(H,\lambda)$:
\begin{align*}
H=\begin{pmatrix}
h_1(m_1) & \ldots & h_1(m) & \ldots & h_1(m_n)\\
h_2(m_1) &\ldots & h_2(m) & \ldots & h_2(m_n)\\
h_3(m_1) & \ldots & h_3(m) & \ldots & h_3(m_n)\\
h_4(m_1) &\ldots & h_4(m) & \ldots & h_4(m_n)\\
h_5(m_1) & \ldots & h_5(m) & \ldots & h_5(m_n)\\
h_6(m_1) & \ldots & h_6(m) & \ldots & h_6(m_n)\\
-(h_1+h_4)(m_1) & \ldots & -(h_1+h_4)(m) & \ldots & -(h_1+h_4)(m_n)\\
-(h_2+h_5)(m_1) & \ldots & -(h_2+h_5)(m) & \ldots & -(h_2+h_5)(m_n)\\
-(h_3+h_6)(m_1) & \ldots & -(h_3+h_6)(m) & \ldots & -(h_3+h_6)(m_n)\\
\end{pmatrix},
\end{align*}
\begin{align*}
\lambda_{m} = (-1)^{(\sum_{\gamma=1}^6h_{\gamma})(m)}\binom{(h_3\!+\!h_6)(m)}{k}\binom{(h_2\!+\!h_5)(m)}{j}\binom{(h_1\!+\!h_4)(m)}{i},
\end{align*}
where $m:=(i,j,k)\in\A$ is a general lattice point,  $m_1,\ldots, m_n$ is an ordered list of elements in $\A$,
$\mathbf{t}:=(s,t,v)$, and $h_{1},\ldots,h_{6}$ are 
\begin{align*}
h_1(\mathbf{t})&=s,   & h_4(\mathbf{t})&=(a+db)l-s-dt-((a+db)-(a'+db'))v,\\
 h_2(\mathbf{t})&=t & h_5(\mathbf{t})&=bl-t-(b-b')v, \\
 h_3(\mathbf{t})&=v, & h_6(\mathbf{t})&=l-v.
\end{align*}

\end{proposition}
\begin{proof}
The polynomial $f_{\A,w}(\mathbf{t})$ in the definition of $\mathcal{P}$, can be expressed as a sum
\begin{align*}
f_{\A,w}(\mathbf{t})&=\sum_{k=0}^{(h_3+h_6)(m)}\sum_{j=0}^{(h_2+h_5)(m)}\sum_{i=0}^{(h_1+h_4)(m)}S_{m}
\end{align*}
where
\begin{align*}
S_{m}(\mathbf{t})=\binom{(h_3+h_6)(m)}{k}\binom{(h_2+h_5)(m)}{j}\binom{(h_1+h_4)(m)}{i}\mathbf{t}^m.
\end{align*}
We let
\begin{align*}
(\overline{w\chi_\mathscr{A}}(\mathbf{t}))_{m}&=\frac{S_{m}(\mathbf{t})}{f_{\A,w}(\mathbf{t})},
\end{align*}
then, the vector of all $(\overline{w\chi_\mathscr{A}}(\mathbf{t}))_{m}$ gives the monomial parametrisation (\ref{monomialparametrisation}) of $Y_{\A,w}$ with weights
\begin{align*}
w_{m}=\binom{(h_3+h_6)(m)}{k}\binom{(h_2+h_5)(m)}{j}\binom{(h_1+h_4)(m)}{i}.    
\end{align*}
 Composing the monomial parametrisation with the tautological map (\ref{tautologicalmap}) gives the following birational map:
\begin{align*}
 &(\tau_\mathscr{A}\circ\overline{w\chi_\mathscr{A}})(\mathbf{t})=\\
&=\frac{1}{f_{\A,w}(\mathbf{t})}\left(\sum_{k=0}^{(h_3+h_6)(m)}\sum_{j=0}^{(h_2+h_5)(m)}\sum_{i=0}^{(h_1+h_4)(m)}S_{m}\right)(\mathbf{t})\\
&=\Bigg(\frac{ls\left(((a+db)(1+s)^d+at)f_{a,b',d}(\mathbf{t})+v(a'+db')(1+s)^d+a't)f_{a',b',d}(\mathbf{t})\right)}{f_{1,1,d}(\mathbf{t})\left(f_{a,b,d}(\mathbf{t})b+vf_{a',b',d}(\mathbf{t}))\right)},\\
 &\qquad\frac{lt\left(b'vf_{a',b',d}(\mathbf{t})+bf_{a,b,d}(\mathbf{t})\right)}{f_{0,1,d}(\mathbf{t})\left(f_{a,b,d}(\mathbf{t})+vf_{a',b',d}(\mathbf{t}))\right)},
 \frac{lvf_{a',b',d}(\mathbf{t})}{\left(f_{a,b,d}(\mathbf{t})+vf_{a',b',d}(\mathbf{t}))\right)}\Bigg)
\end{align*}
with the following inverse:
\begin{align*}
\varphi(\mathbf{t})=&\Bigg(
\frac{h_1(\mathbf{t})}{h_4(\mathbf{t})},\frac{h_2(\mathbf{t})((h_1+h_4)(\mathbf{t}))^d}{(h_4(\mathbf{t}))^dh_5(\mathbf{t})},\\
&\qquad\frac{h_3(\mathbf{t})((h_1+h_4)(\mathbf{t}))^{(a+db)-(a'+db')}((h_2+h_5)(\mathbf{t}))^{b-b'}}{(h_4(\mathbf{t}))^{(a+db)-(a'+db')}(h_5(\mathbf{t}))^{b-b'}(h_6(\mathbf{t}))} \Bigg)
\end{align*}
Composing $\varphi(\mathbf{t})$ with the monomial parametrisation
gives 
\begin{align*}\frac{S_{m}(\varphi(\mathbf{t}))}{f_{\A,w}(\varphi(\mathbf{t}))} = 
w_{m} (-1)^{(\sum_{\gamma=1}^6h_\gamma)(m)}h(\mathbf{t})^{h(m)}
\end{align*}
where $h(q)=(h_1(q),\dotsm,h_6(q),-h_7(q),-h_8(q),-h_9(q))$, $(q\in\{\mathbf{t},m\})$, the functions $h_1,\ldots,h_6$ are as in the statement of the theorem and $h_7=h_1+h_4$, $h_8=h_2+h_5$, and $ h_9=h_3+h_6.$

According to \cite[Proposition 8.4]{Clarke2020}, the polytope has rational linear precision with weights $w_{m}$ as defined above and the Horn parametrisation of $Y_{\mathscr{A},w}$ is given by:
\begin{align*}
\frac{S_{m}(\varphi(p))}{f_{\A,w}(\varphi(p))}&=w_{m}(-1)^{(\sum_{\gamma=1}^6h_\gamma)(m)}h(p)^{h(m)}
\end{align*}
where
\begin{align*}
p=\sum_{m\in\mathscr{A}}\frac{u_{m}}{u_+}(m),\qquad  u_{+}=\sum_{(m)\in\mathscr{A}}u_{m}.
\end{align*}
Since the Horn parametrisation is, by definition, a product of linear forms whose exponents match their coefficients, we know that
the columns of H are the vectors $h(m)$.
It follows  that $\lambda_{m}=(-1)^{(\sum_{\gamma=1}^6h_\gamma)(m)} w_{m}$.
\end{proof}

 \subsection{Minimal Horn pairs for prismatoids in $\mathcal{P}$}
  We now study Questions \ref{q1} and \ref{q2} for elements
 in $\mathcal{P}$.
 Proposition~\ref{thm:bigfamily} gives a Horn
 pair $(H,\lambda)$ for each $(P,w)\in \mathcal{P}$ in Table~\ref{table:3Dtrapezoids}, however $H$  need not be
 the minimal Horn matrix in each case.
 By \cite[Lemma 9]{Duarte2020},  we can find \emph{the minimal Horn matrix associated to $(P,w)$} using row reduction operations on $H$.

  \begin{notation}\label{notation: labelling of the faces}
We will denote the facets of a general element in $\mathcal{P}$ as follows:
\begin{align*}
F_1&=\text{left facet}, \qquad &F_2&=\text{front facet}, \qquad &F_3&=\text{bottom facet},\\
F_4&=\text{right facet},\qquad &F_5&=\text{back facet}, \qquad &F_6&=\text{upper facet}.
\end{align*}
This labelling is used in Figure~\ref{fig:bigfamily}. The  normal vectors of each facet are:
\begin{align*}
n_1&=(1,0,0),  &n_2&=(0,1,0),  &n_3&=(0,0,1), \\
n_4&=(-1,-d,-((a+db)-(a'+db'))), &n_5&=(0,-1,-(b-b')), &n_6&=(0,0,-1).
\end{align*}\end{notation}

\subsubsection{The non-simple prismatoids}\label{subsec:nonsimple}
The trapezoidal pyramids, tensor product pyramids and prismatoids with triangle on top, depicted  in Table~\ref{table:3Dtrapezoids} (A) and (B), are all examples of non-simple polytopes in $\mathcal{P}$. Their primitive collections are:
\begin{align*}
\text{Prismatoids with triangle on top}&\quad&\{n_1,n_3,n_4\},\,\{n_1,n_2,n_4\},\,\{n_2,n_5\},\,\{n_3,n_6\}\\
\text{Trapezoidal pyramids}&\quad&\{n_1,n_3,n_4\},\,\{n_2,n_3,n_5\}\\
\text{Tensor product pyramids}&\quad&\{n_1,n_3,n_4\},\,\{n_2,n_3,n_5\}.
\end{align*}
There is no $n_6$ for the two pyramids since the facet $F_6$ has collapsed to a point.\\
For a pair $(P,w)$ in the subfamily of non-simple prismatoids in $\mathcal{P}$, the matrix $M_{\A,\Sigma_P}$ cannot be a Horn matrix since the primitive collections are not a partition of the 1-dimensional rays of the normal fan and therefore the columns cannot add to zero.

\begin{example}
It follows from Proposition~\ref{thm:bigfamily} that the minimal Horn matrix associated to the tensor product pyramid in Table~\ref{table:3Dtrapezoids} (B) is:
\begin{align*}
H=\begin{pmatrix}
h_1(m_1) &  \ldots & h_1(m_{n})\\
h_2(m_1) &  \ldots & h_2(m_{n})\\
h_3(m_1) &  \ldots & h_3(m_{n})\\
h_4(m_1) &  \ldots & h_4(m_{n})\\
h_5(m_1) &  \ldots & h_5(m_{n})\\
-(h_1+h_2+h_4+h_5-h_6)(m_1) &  \ldots & -(h_1+h_2+h_4+h_5-h_6)(m_{n})\\
-(h_3+h_6)(m_1) & \ldots & -(h_3+h_6)(m_{n})\\
\end{pmatrix}
\end{align*}
where $m_{1},\ldots m_{n}\in\A$, $\mathbf{t}:=(s,t,v)$ and $h_{1},\ldots,h_{6}$ are defined to be
\begin{align*}
h_1(\mathbf{t})&=s,  &h_2(\mathbf{t})&=t, &h_3(\mathbf{t})&=v, \\
  h_4(\mathbf{t})&=al-s-av & h_5(\mathbf{t})&=bl-t-bv,  
 & h_6(\mathbf{t})&=l-v.
\end{align*}
We were able to add $h_6$ to the negative rows $-(h_1+h_4)$ and $-(h_2+h_5)$ since all three rows are colinear in this case. As a result, the positive part of the minimal Horn matrix coincides with the lattice distance matrix of $\A$.
\end{example}

\subsubsection{The simple prismatoids with fewer facets}\label{section:fewerfacets}
The trapezoidal wedges (A), tensor product wedges (B), triangular frusta (C) and triangular based pyramids (C) from Table~\ref{table:3Dtrapezoids} are simple prismatoids with less than $6$ facets. The primitive collections in each case are:
\begin{align*}
\text{Trapezoidal wedges}&\qquad&\{n_1,n_4\},\,\{n_2,n_3,n_5\}\\
\text{Tensor product wedges $a'= 0$}&\qquad&\{n_1,n_3,n_4\},\,\{n_2,n_5\}\\
\text{Tensor product wedges $b'=0$}&\qquad&\{n_1,n_4\},\,\{n_2,n_3,n_5\}\\
\text{Triangular based pyramid}&\qquad&\{n_1,n_2,n_3,n_4\}\\
\text{Triangular frusta}&\qquad&\{n_1,n_2,n_4\},\,\{n_3,n_6\}\\
\end{align*}
None of the polytopes above, except the triangular frusta, have an upper facet $F_6$ and hence their normal fans and primitive collections do not include $n_6$. Also, the triangular based pyramid and triangular frusta have no back facet $F_5$ and hence their normal fans and primitive collections do not include $n_5$. In each case, the primitive collections give a partition of the rays in the normal fan, hence the matrix $M_{\A,\Sigma_P}$ associated to $(P,w)$ is a Horn matrix for these cases. The question is whether this Horn matrix belongs to a Horn pair for $(P,w)$.
\begin{example}
Proposition~\ref{thm:bigfamily} gives a Horn pair  for the trapezoidal wedge in Table~\ref{table:3Dtrapezoids} (A), which can be reduced to a Horn pair $(H,\lambda)$,
with:
\begin{align*}
H=\begin{pmatrix}
h_1(m_1) &  \ldots & h_1(m_{n})\\
h_2(m_1) &  \ldots & h_2(m_{n})\\
h_3(m_1) &  \ldots & h_3(m_{n})\\
h_4(m_1) &  \ldots & h_4(m_{n})\\
h_5(m_1) &  \ldots & h_5(m_{n})\\
-(h_1+h_4)(m_1) &  \ldots & -(h_1+h_4)(m_{n})\\
-(h_2+h_5-h_6)(m_1) &  \ldots & -(h_2+h_5-h_6)(m_{n})\\
-(h_3+h_6)(m_1) &  \ldots & -(h_3+h_6)(m_{n})\\
\end{pmatrix}
\end{align*}
where $m_{1},\ldots m_{n}\in\A$, $\mathbf{t}:=(s,t,v)$  and $h_{1},\ldots,h_{9}$ are defined to be
\begin{align*}
h_1(\mathbf{t})&=s,&h_4(\mathbf{t})&=(a+db)l-s-dt-(a-a'+db)v,\\
h_2(\mathbf{t})&=t & h_5(\mathbf{t})&=bl-bv-t, \\
h_3(\mathbf{t})&=v  & h_6(\mathbf{t})&=l-v. \end{align*}
Let us compare $H$ with the matrix $M_{\A,\Sigma_P}$
\begin{align*}
M_{\A,\Sigma_P}=\begin{pmatrix}
h_1(m_1) &  \ldots & h_1(m_{n})\\
h_2(m_1) &  \ldots & h_2(m_{n})\\
h_3(m_1) &  \ldots & h_3(m_{n})\\
h_4(m_1) &  \ldots & h_4(m_{n})\\
h_5(m_1) &  \ldots & h_5(m_{n})\\
-(h_1+h_4)(m_1) &  \ldots & -(h_1+h_4)(m_{n})\\
-(h_2+h_3+h_5)(m_1) &  \ldots & -(h_2+h_3+h_5)(m_{n})\\
\end{pmatrix}
\end{align*}
where $m_{1},\ldots m_{n}$  and $h_{1},\ldots,h_{5}$ are  as in $H$.
For $b=1$, we see that 
$
h_2+h_5-h_6=0$,
$h_3+h_6=l$, and
$h_2+h_3+h_5=l$, thus $H=M_{\A,\Sigma_P}$.
For $b>1$, $H$ and $M_{\A,\Sigma_P}$ are minimal Horn matrices, hence, by uniqueness, $M_{\A,\Sigma_P}$ cannot give rise to a Horn pair for $(P,w)$.
\end{example}
\noindent For all other examples of simple prismatoids with fewer facets, we noticed a similar phenomenon. Firstly, if $n_i$ is not in the normal fan, then the positive row $h_i$ is collinear with a negative row. In particular, for all these examples the positive part of the minimal Horn matrix coincides with the lattice distance matrix of $\A$. Below we summarise for which parameters the matrix $M_{\A,\Sigma_P}$ gives rise to a Horn pair for $(P,w)$, this is not true in general for these families of `simple prismatoids with fewer facets'.
\begin{align*}
\text{Trapezoidal wedges}&\qquad& b=1\\
\text{Tensor product wedges $a'= 0$}&\qquad& a=1\\
\text{Tensor product wedges $b'=0$}&\qquad& b=1\\
\text{Triangular based pyramid}&\qquad& b=d=1\\
\text{Triangular frusta}&\qquad& d=1\\
\end{align*}
These seemingly arbitrary constraints have a nice geometrical interpretation. The constraint $b=1$ forces the triangular facet $F_1$ in the trapezoidal wedges and tensor product wedges ($b'=0$) to be a simplex. The constraint $a=1$ forces the triangular facet $F_2$ in the tensor product wedges ($a'=0$) to be a simplex. The constraint $b=d=1$ on the triangular based pyramid, means it is a 3D simplex and the constraint $d=1$ on the triangular frusta forces the two triangular facets $F_3$ and $F_6$ to be simplices. All the prismatoids considered in this section, except the ones just described, are examples of polytopes with simplicial normal fans for which the answer to Question~\ref{q2} is negative.

\subsubsection{The trapezoidal and tensor product frusta}

The primitive collections for the trapezoidal frusta and the tensor product frusta are:
\begin{align*}
\{n_1,n_4\}\qquad\{n_2,n_5\}\qquad\{n_3,n_6\}.
\end{align*}
It follows easily that the Horn matrix given by Proposition~\ref{thm:bigfamily} is $M_{\A,\Sigma_P}$.
This matrix is also the minimal Horn matrix for all trapezoidal frusta and for general tensor product
frusta. However, there are cases of tensor product frusta, where two or more rows of this matrix are collinear and hence the minimal Horn matrix is not exactly $M_{\A,\Sigma_P}$.
For an overview of all minimal Horn matrices for the family $\mathcal{P}$ of prismatoids, see Table~\ref{table:minimal} in Appendix~\ref{ap:table}.
\begin{example}
A Horn matrix associated to the tensor product frusta in Table~\ref{table:3Dtrapezoids} (B)   is
\begin{align*}
M_{\A,\Sigma_P}=\begin{pmatrix}
h_1(m_1) & \ldots & h_1(m_{n})\\
h_2(m_1) & \ldots & h_2(m_{n})\\
h_3(m_1) & \ldots & h_3(m_{n})\\
h_4(m_1) & \ldots & h_4(m_{n})\\
h_5(m_1) & \ldots & h_5(m_{n})\\
h_6(m_1) & \ldots & h_6(m_{n})\\
-(h_1+h_4)(m_1) &  \ldots & -(h_1+h_4)(m_{n})\\
-(h_2+h_5)(m_1) &  \ldots & -(h_2+h_5)(m_{n})\\
-(h_3+h_6)(m_1) &  \ldots & -(h_3+h_6)(m_{n})\\
\end{pmatrix}
\end{align*}
where $m_{1},\ldots m_{n}\in\A$, $\mathbf{t}:=(s,t,v)$  and $h_{1},\ldots,h_{6}$ are defined to be
\begin{align*}
h_1(\mathbf{t})&=s, & h_4(\mathbf{t})&=al-s-(a-a')v\\ h_2(\mathbf{t})&=t, & h_5(\mathbf{t})&=bl-t-(b-b')v,\\
h_3(\mathbf{t})&=v,&  h_6(\mathbf{t})&=l-v.
\end{align*}
If we consider the subfamily of tensor product frusta such that $a=\lambda b,\,a'=\lambda b'$ for $\lambda\geq 1$ or $\lambda=\frac{1}{\mu}$ with $\mu\geq 1$, then the minimal Horn matrix is
\begin{align*}
H=\begin{pmatrix}
h_1(m_1) & \ldots & h_1(m_{n})\\
h_2(m_1) & \ldots & h_2(m_{n})\\
h_3(m_1) & \ldots & h_3(m_{n})\\
h_4(m_1) & \ldots & h_4(m_{n})\\
h_5(m_1) & \ldots & h_5(m_{n})\\
h_6(m_1) & \ldots & h_6(m_{n})\\
-(h_1+h_2+h_4+h_5)(m_1) &  \ldots & -(h_1+h_2+h_4+h_5)(m_{n})\\
-(h_3+h_6)(m_1) &  \ldots & -(h_3+h_6)(m_{n})\\
\end{pmatrix},
\end{align*}
where $m_{1},\ldots m_{n}\in\A$  and $h_{1},\ldots,h_{6}$, are as above.
\end{example}

\begin{theorem}\label{thm: 3D primitivecollection compatible}
For all pairs in $\mathcal{P}$, the positive part of the minimal Horn matrix is the lattice distance matrix of $\A$. For the subfamilies of $\mathcal{P}$ in Table~\ref{table:primitivesubfams}, the matrix $M_{\A,\Sigma_P}$ gives rise to a Horn pair for $(P,w)$.

\begin{table}[H]
\begin{center}
\begin{tabular}{ |c|c|}
\hline
\textbf{Name of subfamily}&\textbf{Constraints on $a'\leq a,b'\leq b,d$} \\
\hline \\ [-1em]
Trapezoidal wedges&$a'>0,\,b'=0,\,b=1,\,d>0$\\
\hline\\ [-1em]
Tensor product wedges ($a'=0$)&$a'=0,\,a=1,\,b'>0,\,d=0$\\
\hline\\ [-1em]
Tensor product wedges ($b'=0$)&$a'>0,\,b'=0,\,b=1,\,d=0$\\
\hline\\ [-1em]
3D simplex&$a'=a=b'=0,\,b=1,\,d=1$\\
\hline\\ [-1em]
Triangular frusta&$a'=a=0,\,b>0,\,d=1$\\
\hline\\ [-1em]
Tensor product frusta&$a'>0,\,b'>0,\,d=0$\\
\hline\\ [-1em]
Trapezoidal frusta& $a'>0,\,b'>0,\,d>0$\\
\hline
\end{tabular} 
\caption{
Subfamilies of prismatoids for which there exists a Horn pair $(H,\lambda)$ with $H=~M_{\A,\Sigma_{P}}$.}\label{table:primitivesubfams}
\end{center}
\end{table}
\end{theorem}
\begin{proof}
Let $(P,w)\in \mathcal{P}$, if $M_{\A,\Sigma_P}$ is a Horn matrix, then after row reduction operations, we get a  minimal Horn matrix. Comparing this matrix with the minimal Horn matrix associated to $(P,w)$ in Table~\ref{table:minimal} (Appendix~\ref{ap:table}) and by uniqueness of minimal Horn matrices,   one can verify both statements on the theorem.
\end{proof}

\section{Multinomial staged tree models}\label{multinomialstagedtreemodels}
 
In this section we define multinomial staged tree models,
we prove that every such model has rational MLE and
we give criteria to determine when such models
are toric varieties for binary multinomial staged trees,
see Theorem~\ref{rationalMLEmultinomialtrees} and Theorem~\ref{thm:balancedformultinomialtrees} respectively.  To each toric binary multinomial staged tree one can associate a polytope, by Theorem~\ref{thm:rational} such a polytope has rational linear precision. These results imply our main theorem:
\begin{theorem} \label{thm:main}
Polytopes of toric binary  multinomial staged trees have rational linear precision.
\end{theorem}
\noindent Our motivation to introduce this model class arose from the observation
that the Horn pairs of all 2D and 3D polytopes in Section~\ref{sec:2Dand3Dpolytopes} could be interpreted
as a statistical model defined by an event tree with a specific choice of parametrisation. Multinomial staged tree models improve the understanding of polytopes with
rational linear precision in $2D$ and $3D$. They also offer a generalisation for polytopes with rational linear precision in higher dimensions.

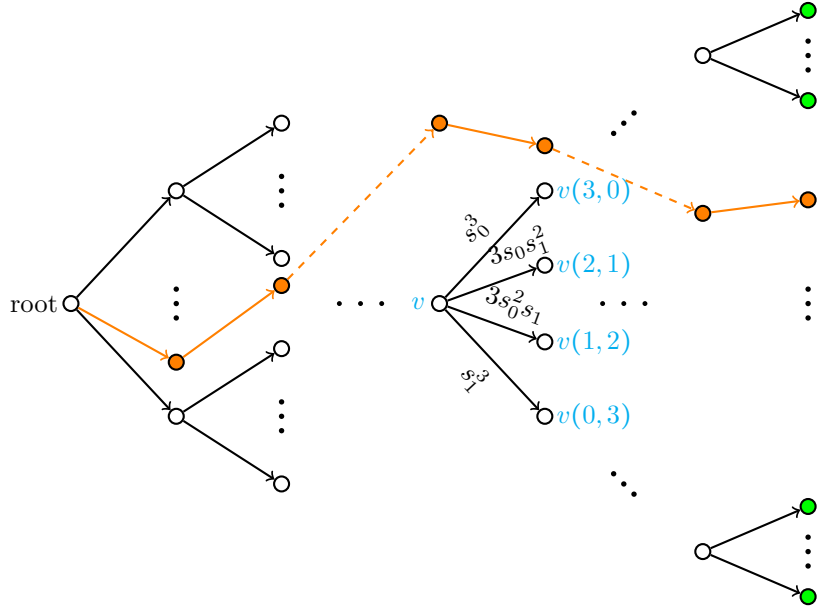
\begin{figure} 
\centering
 \begin{tikzpicture}[thick,scale=0.2]
\renewcommand{\xx}{7}
\renewcommand{\yy}{3}

\node [circle, draw, fill=green!0, inner sep=2pt, minimum width=2pt](r) at (1*\xx,7.5*\yy) {};
\node [circle, draw, fill=green!0, inner sep=2pt, minimum width=2pt](d1) at (2*\xx,10*\yy) {};
\node [circle, draw, fill=green!0, inner sep=2pt, minimum width=2pt](d2) at (2*\xx,5*\yy) {};
\node [circle, draw, fill=green!0, inner sep=2pt, minimum width=2pt](c2) at (3*\xx,8.5*\yy) {};
\node [circle, draw, fill=green!0, inner sep=2pt, minimum width=2pt] (c3) at (3*\xx,6.5*\yy) {};
\node [circle, draw, fill=green!0, inner sep=2pt, minimum width=2pt](cc) at (3*\xx,11.5*\yy) {};
\node [circle, draw, fill=green!0, inner sep=2pt, minimum width=2pt](ccc) at (3*\xx,3.5*\yy) {};
\node [circle, draw, fill=green!0, inner sep=2pt, minimum width=2pt](f1) at (4.5*\xx,7.5*\yy) {};
\node [circle, draw, fill=green!0, inner sep=2pt, minimum width=2pt](g1) at (5.5*\xx,5*\yy) {};
\node [circle, draw, fill=green!0, inner sep=2pt, minimum width=2pt](g2) at (5.5*\xx,10*\yy) {};
\node [circle, draw, fill=green!0, inner sep=2pt, minimum width=2pt](g3) at (5.5*\xx,6.65*\yy) {};
\node [circle, draw, fill=green!0, inner sep=2pt, minimum width=2pt](g4) at (5.5*\xx,8.35*\yy) {};
\node [circle, draw, fill=green!0, inner sep=2pt, minimum width=2pt](ee) at (7*\xx,13*\yy) {};
\node [circle, draw, fill=green!0, inner sep=2pt, minimum width=2pt](eee) at (7*\xx,2*\yy) {};

\node [circle, draw, fill=green, inner sep=2pt, minimum width=2pt](bb1) at (8*\xx,14*\yy) {};
\node [circle, draw, fill=green, inner sep=2pt, minimum width=2pt](bb2) at (8*\xx,12*\yy) {};

\node [circle, draw, fill=green, inner sep=2pt, minimum width=2pt](bbb1) at (8*\xx,3*\yy) {};
\node [circle, draw, fill=green, inner sep=2pt, minimum width=2pt](bbb2) at (8*\xx,1*\yy) {};

\node [scale=1.6](cx1) at (2*\xx,7.5*\yy) {$\vdots$};
\node [scale=1.6](cx2) at (3*\xx,10*\yy) {$\vdots$};
\node [scale=1.6](cx3) at (3*\xx,5*\yy) {$\vdots$};
\node [scale=1.6](cx3) at (3.75*\xx,7.5*\yy) {$\ldots$};
\node [scale=1.6](cx3) at (6.25*\xx,7.5*\yy) {$\ldots$};

\node [scale=1.6](vdots) at (8*\xx,7.5*\yy) {$\vdots$};
\node [scale=1.6](ddots1) at (6.25*\xx,3.5*\yy) {$\ddots$};
\node [scale=1.6](ddots2) at (6.25*\xx,11.5*\yy) {$\udots$};
\node [scale=1.6](vdots1) at (8*\xx,2*\yy) {$\vdots$};
\node [scale=1.6](vdots2) at (8*\xx,13*\yy) {$\vdots$};

\node (fl1) at (4.3*\xx,7.5*\yy) {\footnotesize{\color{cyan}$v$}};
\node (gl1) at (5.97*\xx,5*\yy) {\footnotesize{\color{cyan}$v(0,3)$}};
\node (gl2) at (5.97*\xx,10*\yy) {\footnotesize{\color{cyan}$v(3,0)$}};
\node (gl3) at (5.97*\xx,6.65*\yy) {\footnotesize{\color{cyan}$v(1,2)$}};
\node (gl4) at (5.97*\xx,8.35*\yy){\footnotesize{\color{cyan}$v(2,1)$}};
\node (r2)  at (0.65*\xx,7.5*\yy) {\footnotesize  root};


\draw[->] (r) -- node [above,draw=none, fill=none, scale=5] {} (d1);
\draw[->] (r) -- node [below,draw=none, fill=none, scale=5] {} (d2);
\draw[->] (f1) -- node [midway,sloped,below] {\footnotesize $s_1^{3}$} (g1);
\draw[->] (f1) -- node [midway,sloped,above] {\footnotesize $s_0^{3}$} (g2);
\draw[->] (f1) -- node [midway,sloped,above] {\footnotesize $\phantom{++}3s_0^{2}s_1$} (g3);
\draw[->] (f1) -- node [midway,sloped,above] {\footnotesize $\phantom{++++}3s_0s_1^{2}$} (g4);

\draw[->] (d1) -- (c2);
\draw[->] (d2) -- (c3);

\draw[->] (d1) -- (cc);
\draw[->] (d2) -- (ccc);

\draw[->] (ee) -- (bb1);
\draw[->] (ee) -- (bb2);
\draw[->] (eee) -- (bbb1);
\draw[->] (eee) -- (bbb2);


\node [circle, draw, fill=orange, inner sep=2pt, minimum width=2pt](path1) at (2*\xx,6.2*\yy) {};
\node [circle, draw, fill=orange, inner sep=2pt, minimum width=2pt](path2) at (3*\xx,7.9*\yy) {};
\node [circle, draw, fill=orange, inner sep=2pt, minimum width=2pt](path3) at (4.5*\xx,11.5*\yy) {};
\node [circle, draw, fill=orange, inner sep=2pt, minimum width=2pt](path4) at (5.5*\xx,11*\yy) {};
\node [circle, draw, fill=orange, inner sep=2pt, minimum width=2pt](path5) at (7*\xx,9.5*\yy) {};
\node [circle, draw, fill=orange, inner sep=2pt, minimum width=2pt](path6) at (8*\xx,9.8*\yy) {};

\draw[->,dashed, orange] (path4) --(path5) ;
\draw[->,orange] (path5) -- (path6);
\draw[->,orange] (r) -- (path1);
\draw[->,orange] (path1) -- (path2);
\draw[->,dashed, orange] (path2) --(path3);
\draw[->,orange] (path3) -- node [midway,sloped,above]{}(path4);
\end{tikzpicture} 
\caption{General sketch of a multinomial staged tree. The vertex $v$ is labeled by the floret of degree 3 on $S_l$, denoted by $f_{l,3}$. The green vertices are the leaves and a root-to-leaf path is shown in orange.} \label{fig:generaltree}
\end{figure}

\subsection{Definition of multinomial staged trees} \label{subsec:defmultitrees}
We start by introducing the multinomial model as an 
event tree. This model is the building block of
multinomial staged tree models. Throughout this section $m$ denotes a positive integer and $[m]:=\{1,2,\ldots,m\}$, this differs from Section~\ref{sec:2Dand3Dpolytopes} where $m$ was used for lattice point.
\begin{example}\label{Example:multinomialmodel}
The multinomial model encodes the experiment of rolling a $q$-sided die $n$ independent times and recording the side that came up each time.
The outcome space for this model is the set $\Omega$ of all tuples $K=(k_1,\dotsm,k_q)\in \mathbb{N}^q$ whose entries sum to $n$.
We can depict this model by a rooted tree $\T=(V,E)$ with vertices $V=\{r\}\cup\{r(K):K\in \Omega\}$ and edges $E=\{r\to r(K): K\in \Omega\}$. To keep track of
the probability of each outcome we can further label
$\T$ with monomials on the set of symbols $\{s_1,\ldots, s_q\}$. Each symbol $s_i$ represents the probability 
that the die shows side $i$ when rolled once. The monomial
representing the probability of outcome $K$
is the term with vector of exponents $K$ in the multinomial
expansion of $(s_1+\ldots+s_q)^n$, namely ${n\choose K}\prod_{i=1}^ns_i^{k_i}$, where ${n\choose K}:={n\choose k_1,\dotsm,k_m}$. The labelled tree $b\T_{\Delta_2}$ in Figure~\ref{fig:my_2dtrees}, represents
the multinomial model with $n=b$ and $q=3$.
\end{example}

In general terms a multinomial staged tree, is a labelled and directed event tree such that at each vertex, the subsequent event
is given by a multinomial model as in Example~\ref{Example:multinomialmodel}. To introduce this
concept formally, we start with a rooted and directed tree $\T=(V,E)$ with vertex set $V$ and edge set $E$ such that edges are directed away from the root. The directed edge from $v$ to $w$ is
denoted $v\to w$, the set of children of 
a vertex $v\in V$ is $\ch(v):=\{u\in V:v\rightarrow u\in E\}$ and 
the set of outgoing edges from $v$ is $E(v):=\{v\to u : u\in\ch(v)\}$. If $\ch(v)=\emptyset$ then we say that $v$ is a leaf and we let $\widetilde{V}$ denote the set of non-leaf vertices of $\T$.

\noindent Given a rooted and directed tree $\T$, we now explain how
to label its edges using monomials terms. Figure~\ref{fig:generaltree} shows
a general sketch of a multinomial staged tree.
\begin{definition} \label{ref: multinomial staged tree}
Fix a set
of symbols $S=\{s_i: i\in I\}$ indexed by a set $I$. Let $I_1,\ldots,I_m$ be a partition of $I$ and $S_1,\ldots,S_m$
the induced partition in the set $S$. 
\begin{enumerate}
\item[(1)] The sets $S_1,\ldots,S_m$ are called \emph{stages}.
    \item[(2)] For $a\in \Z_{\geq 1}$
and $\ell\in [m]$, a \emph{ floret of degree} $a$ on $S_{\ell}$
is the set of terms in the multinomial expansion of the expression $(\sum_{i\in I_{\ell}}s_i)^a$, we denote this set by $f_{\ell,a}$.
\item[(3)] A function $\LL:E \to \bigcup_{\ell\in [m], a\in \Z_{\geq 1}}f_{\ell,a}$ is a \emph{labelling} of $\T$ if for every $v\in \widetilde{V}$, $\LL(E(v))=f_{\ell,a}$ for some $\ell\in[m],\;a\in \Z_{\geq 1}$, and the restriction   $\LL_{v}: E(v)\to f_{\ell,a}$  is a bijection.
\item[(4)] A multinomial staged tree
is a pair $(\T,\LL)$, where $\T$ is a rooted directed tree
and $\LL$ is a labelling of $\T$ as in condition (3).
\end{enumerate}
\end{definition}

In a multinomial staged tree $(\T,\LL)$, each $v\in \widetilde{V}$
is associated to the floret $f_{\ell,a}$ that satisfies $\mathrm{im}(\LL_v)=f_{\ell,a}$. In this case we index the children of $v$ by $v(K)$ where 
$K=(k_{i_1},\ldots, k_{i_{|I_{\ell}|}}) \in \mathbb{N}^{|I_{\ell}|}$ is  a tuple of nonnegative integers that add to $a$ and $i_1,\ldots, i_{|I_{\ell}|}$ is a fixed ordering of the elements in $I_{\ell}$. It follows that when $\mathrm{im}(\LL_{v})=f_{\ell,a}$, then $E(v)=\{v\to v(K): K\in \mathbb{N}^{|I_{\ell}|}, |K|=a\}$, where $|K|:= \sum_{q=1}^{|I_{\ell}|}k_{i_{q}}$.  We further assume that the indexing of the children $v$ is compatible with the labelling $\LL$, namely for all multinomial staged trees, $\LL_{v}(v\to v(K))
 = {a\choose K}\prod_{q=1}^{|I_{\ell}|}s_{i_q}^{k_{i_q}}$, where ${a\choose K}={a \choose k_{i_1},\ldots, k_{i_{|I_{\ell}|}}}$.
 It is important to note that this local description of
 the tree at the vertex $v$ is the multinomial model described in Example~\ref{Example:multinomialmodel} up to a change of notation. To clarify
 the notation just introduced we revisit Example~\ref{Example:multinomialmodel}
with a concrete choice of parameters.

\begin{example}
Consider the multinomial model for $q=2$ and $n=3$, the outcome space are all possible outcomes of flipping a coin 3 times. Here $S=S_1=\{s_1,s_2\}$ and the root vertex $v$ will have $4$ children, all of which are leaves. The 4 edges of the tree will be labelled by the elements in the floret $f_{1,3}=\{s_1^3,3s_1^2s_2,3s_1s_2^2,3s_2^3\}$. The sets of children and outgoing edges of $v$ are then
$\ch(v)=\{v(3,0),v(2,1),v(1,2),v(0,3)\}$ and $E(v)=\{v\to v(3,0),v\to v(2,1),v\to v(1,2),v \to v(0,3)\}$. 
\end{example}
\begin{remark}\label{rmk:stratified}
 We will always consider a multinomial staged tree $(\T,\LL)$ as an embedded tree in the plane. This means the tree has a fixed ordering of its edges and vertices. The \emph{level} of a vertex $v$ in $\T$ is the number edges in a path from the root to $v$. All the trees we consider satisfy
 the property that two florets associated to two vertices in different levels must be on different stages. This implies that each root-to-leaf path contains at most one monomial term from each floret. Several figures in Section~\ref{sec:toricmultinomialtrees} contain multinomial staged trees, in these pictures, for simplicity, we
 omit the coefficients of the monomial edge labels.
\end{remark}

 \begin{figure} 
     \centering
     \begin{subfigure}[c]{0.45\linewidth}
     \begin{tikzpicture}[thick,scale=0.2]
	

 	 \node[circle, draw, fill=green!0, inner sep=2pt, minimum width=2pt] (v1) at (18,12) {};
 	 \node[circle, draw, fill=green!0, inner sep=2pt, minimum width=2pt] (v2) at (18,0) {};
	 \node[circle, draw, fill=green!0, inner sep=2pt, minimum width=2pt] (v3) at (18,-12) {};
	 
	 \node (v2dots) at (18,6) {$\vdots$};
	 \node (v3dots) at (18,-6) {$\vdots$};

	 \node[circle, draw, fill=green!00, inner sep=2pt, minimum width=2pt] (r) at (0,0) {};

	 \draw[->]   (r) -- node[midway,sloped,above]{}    (v1) ;
 	 \draw[->]   (r) -- node[midway,sloped,above]{\footnotesize$\phantom{++++}{b\choose i, j, b-i-j}s_0^is_1^js_2^{b-i-j}$}  (v2) ;
	 \draw[->]   (r) -- node[midway,sloped,below]{}  (v3) ;

	 \node at (0,1.5) {\footnotesize$r$} ;
	 \node at (-1,12) {$\T_{b\Delta_2}$} ;

\end{tikzpicture}
    \end{subfigure}
    \begin{subfigure}[c]{0.45\linewidth}
     \begin{tikzpicture}[thick,scale=0.2]

     \node[circle, draw, fill=green!0, inner sep=2pt, minimum width=2pt] (w1) at (24,12) {};
	 \node[circle, draw, fill=green!0, inner sep=2pt, minimum width=2pt] (w3) at (24,6) {};	
     \node[circle, draw, fill=green!0, inner sep=2pt, minimum width=2pt] (w4) at (24,3) {};
	 \node[circle, draw, fill=green!0, inner sep=2pt, minimum width=2pt] (w6) at (24,-3) {};	
	 \node[circle, draw, fill=green!0, inner sep=2pt, minimum width=2pt] (w7) at (24,-6) {};
	 \node[circle, draw, fill=green!0, inner sep=2pt, minimum width=2pt] (w9) at (24,-12) {}; 

 	 \node[circle, draw, fill=green!0, inner sep=2pt, minimum width=2pt] (v1) at (12,9) {};
 	 \node[circle, draw, fill=green!0, inner sep=2pt, minimum width=2pt] (v2) at (12,0) {};
	 \node[circle, draw, fill=green!0, inner sep=2pt, minimum width=2pt] (v3) at (12,-9) {};
	 
	 \node (v2dots) at (12,5) {$\vdots$};
	 \node (v3dots) at (12,-4) {$\vdots$};
	 \node (w1dots) at (24,9) {$\vdots$};
	 \node (w4dots) at (24,0) {$\vdots$};
	 \node (w4dots) at (24,-9) {$\vdots$};

	 \node[circle, draw, fill=green!00, inner sep=2pt, minimum width=2pt] (r) at (0,0) {};

	 \draw[->]   (r) -- node[midway,sloped,above]{\footnotesize$s_0^b$}    (v1) ;
 	 \draw[->]   (r) -- node[midway,sloped,above]{\phantom{++}\footnotesize$s_0^js_1^{b-j}$}  (v2) ;
	 \draw[->]   (r) -- node[midway,sloped,below]{\footnotesize$s_1^b$}  (v3) ;
	 
 	 \draw[->]   (v1) -- node[midway,sloped,above]{\footnotesize$s_2^a$}   (w1) ;
	 \draw[->]   (v1) -- node[midway,sloped,below]{\footnotesize$s_3^a$}   (w3) ;
	 
	 \draw[->]   (v2) -- node[midway,sloped,above]{\footnotesize$s_2^{a+d(b-j)}$}   (w4) ;
	 \draw[->]   (v2) -- node[midway,sloped,below]{\footnotesize$s_3^{a+d(b-j)}$}   (w6) ;
	 
	 \draw[->]   (v3) --  node[midway,sloped,above]{\footnotesize$s_2^{a+db}$}  (w7) ;
	 \draw[->]   (v3) -- node[midway,sloped,below]{\footnotesize$s_3^{a+db}$}   (w9) ;

	 \node at (0,1.5) {\footnotesize$r$} ;
	 \node at (0,12) {$\T_{a,b,d}$} ;
\end{tikzpicture}
    \end{subfigure}
     \caption{The multinomial staged trees $\T_{b\Delta_2}$ and $\T_{a,b,d}$  represent the multinomial model with three outcomes and $b$ trials and the model in Example~\ref{ex:2dimtrees} respectively.}
     \label{fig:my_2dtrees}
 \end{figure}
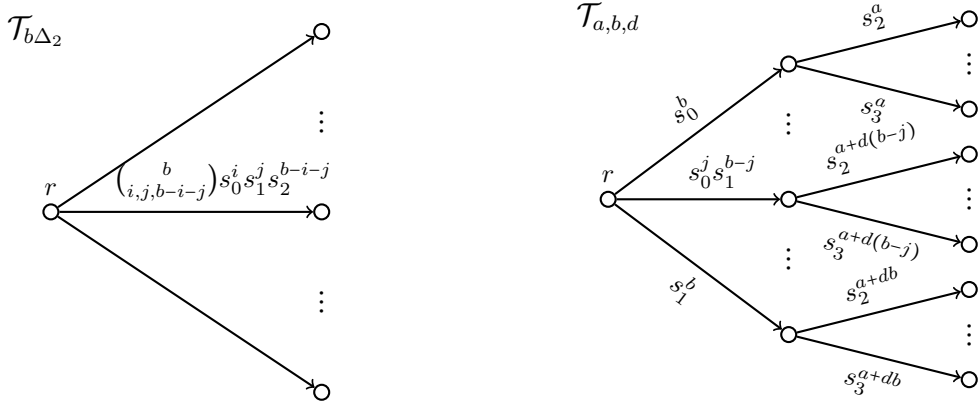
 

\begin{definition}\label{ref: multinomial staged tree model}
Let $(\T,\LL)$ be a multinomial staged tree with index set $I= \sqcup_{\ell\in [m]}I_{\ell}$.
Fix $J$ to be the set of root-to-leaf paths in $\T$, with $|J|=n$.
For $j\in J$, define $p_j$ to be the product of all edge labels in the
path $j$. Let $c_j$ be the coefficient of $p_j$ and $a_j$  the exponent
vector of the symbols $(s_i)_{i\in I}$ in $p_j$. With this
notation, $p_j=c_j\prod_{i\in I}s_{i}^{a_{ij}}$,
where $a_{ij}$ are the entries of $a_j$.
Define the parameter space
\[
\Theta_{\T}:=\{(\theta_i)_{i \in I} \in (0,1)^{|I|} : \sum_{i\in I_{\ell}}\theta_{i}=1 \text{ for all } \ell\in [m]\}
\]
The \emph{multinomial staged tree model} $\M_{(\T,\LL)}$ is the image of the parameterisation
\[
\phi_{\T}: \Theta_{\T} \longrightarrow \Delta_{n-1}^{\circ}, \ (\theta_i)_{i\in I} \mapsto (c_j \prod_{i\in I} \theta_{i}^{a_{ij}})_{j \in J}.
\]
\end{definition}

\begin{remark}
The sum-to-one conditions on the parameter space $\Theta_{\T}$ imply that the image of $\phi_{\T}$ is contained in $\Delta_{n-1}^{\circ}$. The multinomial coefficients on the labels
of $\T$ are necessary for this condition to hold. The model $\M_{(\T,\LL)}$ is an algebraic variety inside $\Delta_{n-1}^{\circ}$
with an explicit parameterisation given by $\phi_{\T}$. For $\theta\in \Theta_{\T}$, $\mathrm{eval}_{\theta}$ is the evaluation map
$s_i \mapsto \theta_i$.
The $j$-th coordinate of $\phi_{\T}$ is $\mathrm{eval}_{\theta}(p_j)$, where $p_{j}=c_j\prod_{i\in I}s_{i}^{a_{ij}}$ (Definition~\ref{ref: multinomial staged tree model}).
For this reason we also use $p_j$ to denote the $j$-th coordinate in the probability simplex $\Delta_{n-1}^\circ$.
\end{remark}
\begin{remark}
 If all of the florets in a multinomial staged tree  have degree one, then it is called a staged tree. Multinomial staged tree models are a generalisation of discrete Bayesian networks \cite{Lauritzen96} and of staged tree models introduced in \cite{Smith2008}. 
\end{remark}

\begin{example} \label{ex:2dimtrees}
Consider the following
experiment with two independent coins: Toss the first coin 
$b$ times and record the number of tails, say this number is $j$. Then toss the second coin $a+d(b-j)$ times, record the number of tails, say it is $i$. An outcome of this experiment is a
 pair $(i,j)$ where $i$ is the number of tails in the
second sequence of coin tosses and $j$ is the number of tails in
the first.
This sequence of events may be represented by a
multinomial staged tree $(\T=(V,E),\LL)$ where
\begin{align*}
   V=\{r\}\cup \{r(j): 0\leq j\leq b\}\cup\{r(i,j):
0\leq j \leq b, 0\leq i \leq a+d(b-j)\} \text{ and }\\
E=\{r\to r(j):0\leq j \leq b\}\cup \{r(j)\to r(i,j):
0\leq j\leq b, 0\leq i \leq a+d(b-j)\}.
\end{align*}
This tree has two stages $S_1=\{s_0,s_1\}, S_2=\{s_2,s_3\}$
that are a formal representation of the parameters of the Bernoulli distributions of the two independent coins.
The set $E(r)$ is labelled by the floret $f_{1,b}$
and the set $E(r(j))$ is labelled by the floret $f_{2,a+d(b-j)}$. Following the conventions set up earlier we see that $\LL(r\to r(j))={b\choose j} s_0^js_1^{b-j}$
and $\LL(r(j)\to r(i,j))={a+d(b-j) \choose i} s_{2}^{a+d(b-j)-i}s_3^i$.
The multinomial staged tree model $\M_{a,b,d}\subset \Delta_n$ assocciated to $(\T,\LL)$,  is the statistical model consisting 
of all probability distributions that follow the experiment just described. Let $p_{ij}$ denote
the probability of the outcome $(i,j)$. The model $\M_{a,b,d}$ is parameterised 
by the map $\phi:\Delta_{1}^{\circ} \times \Delta_1^{\circ} \to \M_{a,b,c}$,
\[ (\theta_0,\theta_1)\times(\theta_2,\theta_3) \mapsto \left(p_{ij}\right)_{\substack{0\leq j\leq b\\ 0\leq i \leq a+d(b-j)}} \text{where } p_{ij}=\tbinom{b }{ j }\tbinom{a+d(b-j)}{ i} \theta_{0}^{j}\theta_1^{b-j}\theta_{2}^i\theta_3^{a+d(b-j)-i}.\]
This model depends on two
independent parameters, thus it has dimension two. The model $\M_{a,b,d}$ is a binary multinomial staged tree model, its tree representation  $\T_{a,b,d}$ is displayed in  Figure~\ref{fig:my_2dtrees}.
\end{example}
\begin{definition} \label{def:ringhoms}
Let $(\T,\LL)$ be a multinomial staged tree. Fix the polynomial rings $\R[P_j: j\in J]$, 
$\R[s_i: i\in I]$ and $\R[s_i: i\in I]/\mathfrak{q}$ where $\mathfrak{q}=\langle 1-\sum_{i\in I_{\ell}} s_i : 
\ell \in [m]\rangle$. We define 
\begin{align*}
 \Psi_{\T}^{\mathrm{toric}}:\R[P_j: j\in J]\to \R[s_i: i\in I]  &\text{ by } P_j \mapsto c_{j}\prod_{i\in I} s_i^{a_{ij}} , \text{ and}
 \\
 \Psi_{\T}:\R[P_j: j\in J]\to 
\R[s_i: i\in I]/\mathfrak{q} &\text{ by } \Psi_{\T}=\pi \circ \Psi_{\T}^{\mathrm{toric}}
\end{align*}
where $\pi:\R[s_i: i\in I]\to 
\R[s_i: i\in I]/\mathfrak{q}$ is the canonical projection to the quotient ring.
The ideal $\ker(\Psi_{\T}^{\mathrm{toric}})$ is the toric ideal associated to $(\T,\LL)$
and $\ker(\Psi_{\T})$ is the model ideal associated to $\M_{(\T,\LL)}$.
Whenever $\ker (\Psi_{\T}) =\ker (\Psi_{\T}^{\mathrm{toric}})$, we call $\M_{(\T,\LL)}$
 a toric model.
\end{definition}
 \begin{remark} \label{rmk:homogenize}
The ideal $\ker (\Psi_{\T})$ defines  the model $\M_{(\T,\LL)}$ implicitly, i.e.
$\M_{(\T,\LL)}= V(\ker (\Psi_{\T}))\cap \Delta_{n}^{\circ}$. Because of the containment
$\ker(\Psi_{\T}^{\mathrm{toric}})\subset \ker (\Psi_{\T})$, $V(\ker(\Psi_{\T}^{\mathrm{toric}}))$ is a toric variety that contains $\M_{(\T,\LL)}$.
 The polynomial $1-\sum_{j\in J}P_j $ is always an element in $\ker(\Psi_{\T})$, hence using this polynomial as a homogenising element, we shall always consider
 $\ker(\Psi_{\T})$ as a homogeneous ideal in $\mathbb{R}[P_{j}:j\in J]$.
 \end{remark}

\subsection{The ideal of model invariants for $\M_{(\T,\LL)}$}
As is common in algebraic geometry,
finding the explicit equations of the prime ideal $\ker (\Psi_{\T})$ is hard. Luckily, the statistical insight of the problem allows us to find a  nonprime ideal, usually referred to as the ideal of model invariants, that defines the model inside the probability simplex. We now define this ideal and postpone the
proof that it has the aforementioned property to Section~\ref{sec:modelinvs}.
\begin{definition}\label{def:idealModelInvariants}
Let $(\T,\LL)$ be a multinomial staged tree. For a vertex $v\in V$, define $[v]:=~¸\{j\in J: \text{ the path }j \text{ goes through the vertex  } v\}$ and  set $P_{[v]}:= \sum_{j\in [v]}P_{j}$.
\begin{align*}
    I_{\mathrm{stages}}&:= \langle bP_{[w]}\left(\sum_{|K|=a, k_{i_q}\geq 1} k_{i_q}P_{[v(K)]}\right)-aP_{[v]}\left(\sum_{|K'|=b,k_{i_q}'\geq 1}k_{i_q}'P_{[w(K')]}\right) :\\ &v\sim w, \mathrm{im}(\LL_v)=f_{\ell,a}, \mathrm{im}(\LL_w)=f_{\ell,b}, \ell \in [m], 1\leq q \leq |I_{\ell}| \rangle, \text{ and } \\
    I_{\mathrm{vertices}} &:= \langle C_{(K^3,K^4)}P_{[v(K^1)]}P_{[v(K^2)]}-
    C_{(K^1,K^2)}P_{[v(K^3)]}p_{[v(K^4)]}: v\in \widetilde{V},\\
    & \; \mathrm{im}(\LL_v)=f_{\ell,a},
    K^1,K^2,K^3,K^4\in \mathbb{N}^{|I_{\ell}|}, |K^1|=
    |K^2|=|K^3|=|K^4|=a, \\ & K^1+K^2=K^3+K^4, C_{(K^i,K^j)}={a \choose K^i}{a \choose K^j},\; i=1,3,\;j=2,4\rangle.
\end{align*}
The \emph{ideal of model invariants of $(\T,\LL)$} is $I_{\M(\T,\LL)}:=I_{\mathrm{stages}}+I_{\mathrm{vertices}}+\langle 1-\sum_{j\in J}P_j\rangle$.
\end{definition}
The previous definition  indicates,
that there are equations that must hold for every pair of vertices with the same associated stage,  and  equations
that must hold for every vertex.
The motivation for this definition of the ideal of model
invariants arises from the technical Lemma~\ref{lem: identifiability} in Appendix~\ref{ap:technical}.
\begin{remark}\label{rem:veronese}
The generators of $I_{\mathrm{vertices}}$ for each fixed vertex $v$
are similar to the Veronese relations of the embedding  $\nu_{a}:\mathbb{P}^{|I_{\ell}|-1}\to \mathbb{P}^{M}$ by monomials of total
degree $a$. The only difference is in the coefficients, defined in Lemma \ref{lem: identifiability} part (2), that are needed for cancellation. 
\end{remark}
\begin{remark}
By definition, $I_{\M(\T,\LL)}$ always contains the sum to one condition $1-\sum_{j\in J}P_{j}$, thus in a similar
way as for $\ker(\Psi_{\T})$ in Remark~\ref{rmk:homogenize}, we always consider $I_{\M(\T,\LL)}$ as a homogeneous
ideal generated by $I_{\mathrm{stages}}$ and $I_{\mathrm{vertices}}$.
\end{remark}
\subsection{Algebraic lemmas for multinomial staged trees}
\label{sec:lemmas}
To understand the defining equations of  $\ker(\Psi_{\T})$ and the case when this ideal is toric, it is important to establish several lemmas
that describe algebraic relations that hold in $\mathbb{R}[P_j:j\in J]$, $\mathbb{R}[s_i:i\in I]$, $\mathbb{R}[P_j:j\in J]/\ker(\Psi_{\T})$ and $\mathbb{R}[P_j:j\in J]/I_{\M(\T,\LL)}$. The reader may decide to
skip this section and only get back to it when the lemmas are used in the
proofs of Theorem~\ref{thm:saturation} and Theorem~\ref{thm:balancedformultinomialtrees}.

\begin{definition} \label{def:interpolating}
Let $(\T,\LL)$ be a multinomial staged tree with $\T=(V,E)$. For $v\in V$, let $\Lambda_{v}$ denote the set of all $v$-to-leaf
paths in $\T$. A path $\lambda\in \Lambda_{v}$ is a sequence of
edges $v\to v_1\to \dotsm \to v_{\alpha}$ where $v_{\alpha}$ is a leaf of $\T$.
For each $v\in V$ we define the \emph{interpolating polynomial of $v$}, $t(v)\in \R[s_i: i\in I]$, by
\[
t(v):= \sum_{\lambda \in \Lambda_{v}} \prod_{e\in \lambda}\LL(e).
\]
If $v$ is a leaf, $t(v):=1$. 
We denote by $\overline{t(v)}$, the image of $t(v)$
under the canonical projection to $\R[s_i: i\in I]/\mathfrak{q}$. 
Note that for all $v\in V$, $\overline{t(v)}=1$.
\end{definition}
 
\begin{lemma} \label{lem:evaluation}
Let $(\T,\LL)$ be a multinomial staged tree where $\T=(V,E)$ and let  $v\in \widetilde{V}$ be such that $\mathrm{im}(\LL_v)=f_{\ell,a}$. 
\begin{itemize}
    \item[(1)]The polynomial $t(v)$ satisfies
\[t(v)=\sum_{|K|=a}{a \choose K}\prod_{i\in I_{\ell}} s_i^{k_i}\cdot t(v(K));\]
    \item[(2)] The image of $P_{[v]}$ under $\Psi_{\T}^{\mathrm{toric}}$ is
    $\left(\prod_{e\in \lambda_{r, v}}\LL(e)\right)\cdot t(v)$, where $\lambda_{r,v}$  is the set of edges in the root-to-$v$ path in $\T$. Moreover $\Psi_{\T}(P_{[v]})= \prod_{e\in \lambda_{r,v}}\LL(e)$.
\end{itemize}
\end{lemma}
\begin{proof}
$(1)$ Any path in $\Lambda_{v}$ goes through a child $v(K)$ of $v$.
The sum of all the edge products corresponding to the paths that go through child $v(K)$ is equal to the sum of all the edge products corresponding to the paths starting at $v(K)$ ($t(v(K))$) multiplied by the label of the edge from $v$ to $v(K)$
(${a \choose K}\prod_{i\in I_{\ell}} s_i^{k_i}$). Taking the sum of this expression over all children of $v$ gives the desired result.

\noindent $(2)$ Let $j$ be a root-to-leaf path that goes through $v$.
Then $j$ is the concatenation of a path from the root to $v$, denoted by $\lambda_{r,v}$ and a path from $v$ to the leaf denoted by $\lambda_{v,j}$. Then
\begin{align*}
    \Psi_{\T}^{\mathrm{toric}}(P_{[v]}) &= \sum_{j\in[v]}
    \prod_{e\in \lambda_{r,v}}\LL(e) \prod_{e\in \lambda_{v,j}}\LL(e) = \left(\prod_{e\in \lambda_{r,v}}\LL(e)\right)\left(\sum_{j\in [v]} \prod_{e\in \lambda_{v,j}}\LL(e)\right)\\
    &= \left(\prod_{e\in \lambda_{r,v}}\LL(e)\right) t(v).
\end{align*}
The second statement follows by noting that $\overline{t(v)}=1$.
\end{proof}

\begin{lemma}
Let $(\T,\LL)$ be a multinomial staged tree, there is a containment of ideals
$I_{\M(\T,\LL)}\subset \ker(\Psi_{\T})$ in $\mathbb{R}[P_j:j\in J]$.
\end{lemma}
\begin{proof}
To show that $I_{\M(\T,\LL)}\subset \ker(\Psi_{\T})$, it suffices to show that
the generators of $I_{\mathrm{stages}}$ and $I_{\mathrm{vertices}}$ are zero after applying
$\Psi_{\T}$. We present the proof for the generators of $I_{\mathrm{stages}}$,
the proof for $I_{\mathrm{vertices}}$ is similar and also uses Lemma~\ref{lem:evaluation}.
A generator of $I_{\M(\T,\LL)}$ is of the form
\begin{align*}
bP_{[w]}\left(\sum_{|K|=a, k_{i_q}\geq 1} k_{i_q}P_{[v(K)]}\right)-aP_{[v]}\left(\sum_{|K'|=b,k_{i_q}'\geq 1}k_{i_q}'P_{[w(K')]}\right),
\end{align*}
where $v,w\in \widetilde{V}$, $\mathrm{im}(\LL_{v})=f_{\ell,a}$, $\mathrm{im}(\LL_{w})=f_{\ell,b}$ for some $\ell\in [m]$ and a fixed $q$, $1\leq q\leq |I_{\ell}|$.\\

\noindent\underline{Claim:} $\quad\Psi_{\T}(\sum_{\substack{|K|=a\\ k_{i_q}\geq 1}}k_{i_q}P_{[v(K)]})=as_{i_q}\prod_{e\in \lambda_{r,v}}\LL(e)$.\newline
\noindent Using Lemma~\ref{lem:evaluation}, we compute $\Psi_{\T}(P_{[v(K)]})$.
\begin{align*}
    \Psi_{\T}(\sum_{\substack{|K|=a\\ k_{i_q}\geq 1}}k_{i_q}P_{[v(K)]}) &= \sum_{\substack{|K|=a\\ k_{i_q}\geq 1}}k_{i_q}\Psi_{\T}(P_{[v(K)]})
   = \sum_{\substack{|K|=a\\ k_{i_q}\geq 1}}k_{i_q}\left(\prod_{e\in \lambda_{r,v}}\LL(e)\right){a \choose K}\prod_{\alpha=1}^{|I_\ell|}s_{i}^{k_{i_\alpha}}\\
    &=\left(\prod_{e\in \lambda_{r,v}}\LL(e)\right)\left(\sum_{\substack{|K|=a\\ k_{i_q}\geq 1}}k_{i_q}{a\choose K}\prod_{\alpha=1}^{|I_\ell|}s_{i}^{k_{i_\alpha}}\right)\\
    &= \left(\prod_{e\in \lambda_{r,v}}\LL(e)\right)
    \left(\sum_{\substack{|K|=a\\ k_{i_q}\geq 1}} k_{i_q}\frac{a(a-1)!}{k_{i_1}!\cdots k_{i_{|I_{\ell}|}}!}s_{i_q}s_{i_q}^{k_{q}-1}\prod_{\substack{\alpha=1\\
    \alpha\neq q}}^{|I_\ell|}s_{i}^{k_{i_\alpha}}\right)\\
    &=\left(\prod_{e\in \lambda_{r,v}}\LL(e)\right)\left(\sum_{|K|=a-1}as_{i_q}{a-1\choose K}\prod_{\alpha=1}^{|I_\ell|}s_{i}^{k_{i_\alpha}}\right)\\
    &= \left(\prod_{e\in \lambda_{r,v}}\LL(e)\right) as_{i_q}(\sum_{\alpha=1}^{|I_\ell|}s_{i_{\alpha}})^{a-1}=\left(\prod_{e\in \lambda_{r,v}}\LL(e)\right) as_{i_q}.
\end{align*}
The last equality follows from the fact that $\sum_{\alpha=1}^{|I_\ell|}s_{i_{\alpha}}=1$ in
$\R[s_i: i\in I]/\mathfrak{q}$. The claim applied to $w\in \widetilde{V}$, implies $\Psi_{\T}(\sum_{\substack{|K|=b\\ k_{i_q}\geq 1}}k_{i_q}P_{[w(K)]})=\left(\prod_{e\in \lambda_{r,w}}\LL(e)\right)bs_{i_q}$. 
Thus, by Lemma~\ref{lem:evaluation}, \begin{align*}
    \Psi_{\T}\Bigg( bP_{[w]}\bigg(\sum_{|K|=a, k_{i_q}\geq 1} k_{i_q}P_{[v(K)]}\bigg)-aP_{[v]}\bigg(\sum_{|K'|=b,k_{i_q}'\geq 1}k_{i_q}'P_{[w(K')]}\bigg)\Bigg) &=\\
    b\prod_{e\in \lambda_{r,w}}\LL(e)\cdot as_{i_q}\prod_{e\in \lambda_{r,v}}\LL(e)-a\prod_{e\in \lambda_{r,v}}\LL(e)\cdot bs_{i_q}\prod_{e\in \lambda_{r,w}}\LL(e)&=0.
\end{align*}
\end{proof}

\subsection{Defining equations of binary multinomial staged trees}
\label{sec:modelinvs}
In this section and the next, we prove Theorem~\ref{thm:saturation} and Theorem~\ref{thm:balancedformultinomialtrees} for binary multinomial staged trees; despite being unable to provide a proof, we believe these statements also hold for non-binary multinomial staged trees. First we show that the ring homomorphism $\Psi_{\T}$ admits an inverse when localised at a suitable
element. From this it follows as a corollary that the ideal of model invariants defines $\M_{(\T,\LL)}$ inside the probability simplex.
\begin{theorem} \label{thm:saturation}
Let $(\T, \LL)$ be a binary multinomial staged tree and define $\mathbf{P}:=\prod_{v\in V}P_{[v]}$. Then, the localised map
\begin{align*}
    (\Psi_{\T})_{\mathbf{P}}&: \left(\mathbb{R}[P_{j}: j\in J]/I_{\M(\T,\LL)}\right)_{\mathbf{P}}\to \left(\mathbb{R}[s_{i}:i\in I]/\mathfrak{q}\right)_{\Psi_{\T}(\mathbf{P})},
\end{align*}
is an isomorphism of $\mathbb{R}$-algebras. Therefore $(I_{\M(\T,\LL)})_{\mathbf{P}}=(\ker(\Psi_{\T}))_{\mathbf{P}}$ and thus $(I_{\M(\T,\LL)}: \mathbf{P}^{\infty})=\ker(\Psi_{\T})$.
\end{theorem}
\begin{proof}
We define a ring homomorphism
\begin{align*}
    \varphi: \left(\mathbb{R}[s_{i}:i\in I]/\mathfrak{q}\right)_{\Psi_{\T}(\mathbf{P})} \to \left(\mathbb{R}[P_{j}: j\in J]/I_{\M(\T,\LL)}\right)_{\mathbf{P}}
\end{align*}
and show that it is a two sided inverse for $(\Psi_{\T})_{\mathbf{P}}$. For $\ell \in [m]$ and $1\leq q \leq |I_{\ell}|$, let $v$ be a vertex with
$\mathrm{im}(\LL_{v})=f_{\ell,a}$ and define
\begin{align*}
    \varphi(s_{i_q}) =\frac{\sum_{|K|=a, k_{i_q}\geq 1} k_{i_q}P_{[v(K)]}}{aP_{[v]}}.
\end{align*}
Note that $\varphi$ is well defined: If $w$ is another vertex with
$\mathrm{im}(\LL_{w})=f_{\ell,b}$, then 
\begin{align*}
    \varphi(s_{i_q}) =\frac{\sum_{|K|=a, k_{i_q}\geq 1} k_{i_q}P_{[v(K)]}}{aP_{[v]}}=\frac{\sum_{|K'|=b, k'_{i_q}\geq 1} k_{i_q}P_{[w(K)]}}{bP_{[w]}}
\end{align*}
because  $bP_{[w]}\left(\sum_{|K|=a, k_{i_q}\geq 1} k_{i_q}P_{[v(K)]}\right)-aP_{[v]}\left(\sum_{|K'|=b,k_{i_q}'\geq 1}k_{i_q}'P_{[w(K')]}\right) \in I_{\M(\T,\LL)}$.
First, we check that $(\Psi_{\T})_{\mathbf{P}}\circ \varphi = \mathrm{Id}$,
\begin{align*}
    (\Psi_{\T})_{\mathbf{P}}(\varphi(s_{i_q})) & =(\Psi_{\T})_{\mathbf{P}}\left(
    \frac{\sum_{|K|=a, k_{i_q}\geq 1}k_{i_q}P_{[v(K)]}}{aP_{[v]}}\right) = \sum_{|K|=a, k_{i_q}\geq 1}\frac{k_{i_q}}{a}(\Psi_{\T})_{\mathbf{P}}\left(\frac{P_{[v(K)]}}{P_{[v]}}\right) \\
    &= \sum_{|K|=a,k_{i_q}\geq 1}\frac{k_{i_q}}{a}\frac{\Psi_{\T}(P_{[v(K)]})}{\Psi_{\T}(P_{[v]})} =\sum_{|K|=a, k_{i_q}\geq 1} \frac{k_{i_q}}{a} {a \choose K} \prod_{\alpha=1}^{|I_{\ell}|} s_{i_\alpha}^{k_{i_\alpha}}=s_{i_q}.
\end{align*}
The second to last equality follows by  using the expression for $\Psi_{\T}(P_{[v]})$ presented in Lemma~\ref{lem:evaluation} part $(2)$, the same result is used to compute $\Psi_{\T}(P_{[v(K)]})$, finally  their quotient $\Psi_{\T}(P_{[v(K)]})/\Psi_{\T}(P_{[v]})$ is exactly ${a \choose K} \prod_{\alpha=1}^{|I_{\ell}|} s_{i_\alpha}^{k_{i_\alpha}}$.
The last equality is obtained by using the same argument as in Lemma~\ref{lem: identifiability} part $(3)$.

\noindent Next, we verify $\varphi \circ (\Psi_{\T})_{\mathbf{P}} = \mathrm{Id}$, which amounts to proving that
$(\varphi\circ\Psi_{\T,\mathbf{P}})(P_j)=P_j$ for each $j\in J$. From this point on we further assume that 
$|I_{\ell}|=2$ for all $\ell \in [m]$. Fix $j\in J$ and let $v_1\to v_2\to \cdots \to v_{\alpha}$ be the root-to-leaf path $j$. By Definition~\ref{def:ringhoms}
\begin{equation} \label{eq:label products}
   (\Psi_{\T})_{\mathbf{P}}(P_j)=c_j \prod_{i\in I} s_i^{a_{ij}}= \LL_{v_1}(v_1\to v_2)\cdots \LL_{v_{\alpha-1}}(v_{\alpha-1}\to v_\alpha). 
\end{equation}
 Where for each $\gamma\in [\alpha-1]$, $\mathrm{im}(\LL_{v_\gamma})=f_{\ell_\gamma,a_\gamma}$ for
 some $\ell_{\gamma}\in [m] $ and $ a_{\gamma}\in \Z_{\geq 1}$. By Remark~\ref{rmk:stratified}, none of the florets $f_{\ell_{\gamma},a_{\gamma}}$
 share the same set of symbols.
  Moreover, for each $\gamma\in[\alpha-1]$,
  \begin{equation} \label{eq:floret label}
       \LL_{v_{\gamma}}(v_{\gamma}\to v_{\gamma+1})={a_{\gamma}\choose k_{\gamma}}s_{\gamma,i_1}^{k_{\gamma}}s_{\gamma,i_2}^{a-k_{\gamma}} \text{ where } S_{\ell_{\gamma}}=\{s_{\gamma,i_1},s_{\gamma,i_2}\}, 0\leq k_{\gamma}\leq a_{\gamma}.
  \end{equation}
With this notation, we also deduce that $v_{\gamma+1}=v_{\gamma}(k_{\gamma},a-k_{\gamma})$. Now we apply $\varphi$ to (\ref{eq:label products}), 
  use that $\varphi$ is a ring homomorphism and use equation (\ref{eq:floret label}) to obtain
  \begin{align}
      \varphi((\Psi_{\T})_{\mathbf{P}}(P_j))&=\varphi(\LL_{v_1}(v_1\to v_2))\cdots \varphi(\LL_{v_{\alpha-1}}(v_{\alpha-1}\to v_{\alpha})) \nonumber\\
      &= \prod_{\gamma=1}^{\alpha-1}{a_{\gamma} \choose k_{\gamma}}\varphi(s_{\gamma,i_1})^{k_{\gamma}}
      \varphi(s_{\gamma,i_2})^{a-k_{\gamma}}.\label{eq:nested}
  \end{align}
  Using the definition of $\varphi$, for each $\gamma\in [\alpha-1]$,
  \begin{align*}
      \varphi(s_{\gamma,i_1})=
      \frac{
      \sum_{k=1}^{a_{\gamma}}kP_{[v_{\gamma}(k,a_{\gamma}-k)]}}{a_{\gamma}P_{[v_{\gamma}]}}  \;\;\;\text{ and }\;\;\;
      \varphi(s_{\gamma,i_2})= \frac{
      \sum_{k=1}^{a_{\gamma}}kP_{[v_{\gamma}(a_{\gamma}-k,k)]}}{a_{\gamma}P_{[v_{\gamma}]}}.
  \end{align*}
 By Lemma~\ref{lem:finally}, with $a=a_{\gamma}$,
 $l_1= \sum_{k=1}^{a_{\gamma}}kP_{[v_{\gamma}(k,a_{\gamma}-k)]}$, $l_2= \sum_{k=1}^{a_{\gamma}}kP_{[v_{\gamma}(a_{\gamma}-k,k)]}$ and $k_0=k_{\gamma}$, we conclude that
 \[
 {a_\gamma \choose k_{\gamma}}\varphi(s_{\gamma,i_1})^{k_{\gamma}}\varphi(s_{\gamma,i_1})^{a-k_{\gamma}}=\frac{P_{[v_{\gamma}(k_{\gamma},a-k_{\gamma})]}}{P_{[v_{\gamma}]}}
 .\]
Thus, continuing from (\ref{eq:nested}) we have
\begin{align*} \label{eq:alternating}
   \varphi((\Psi_{\T})_{\mathbf{P}}(P_j))&=\prod_{\gamma=1}^{\alpha-1}\frac{P_{[v_{\gamma}(k_{\gamma},a-k_{\gamma}])]}}{P_{[v_{\gamma}]}}\\
    &=\frac{P_{[v_{1}(k_1,a_1-k_1)]}}{P_{[v_1]}}\frac{P_{[v_2(k_2,a_2-k_2)]}}{P_{[v_2]}} \cdots\frac{P_{[v_{\alpha-1}(k_{\alpha-1},a_{\alpha-1}-k_{\alpha})]}}{P_{[v_{\alpha-1}]}}
    = \frac{P_{[v_{\alpha}]}}{P_{[v_1]}}.
\end{align*}
To obtain the previous cancellation we
used the fact that 
for each $\gamma\in [\alpha-1]$, $v_{\gamma+1}=v_{\gamma}(k_{\gamma},a-k_{\gamma})$, hence $P_{[v_{\gamma}(k_{\gamma},a-k_{\gamma})]}=P_{[v_{\gamma+1}]}$.
Note that $P_{[v_1]}=1$ by definition of $I_{\M(\T,\LL)}$ and $P_{[v_\alpha]}=P_{j}$ because $v_{\alpha}$ is the last vertex in the path $j$. Thus $\varphi((\Psi_{\T})_{\mathbf{P}}(P_j))=P_{j}$. The second statement of the
theorem follows from the fact that $I_{\M(\T,\LL)}\subset \ker(\Psi_{\T})$ and that the localisation $(\Psi_{\T})_{\mathbf{P}}$ is an isomorphism.
\end{proof}
\begin{corollary}
The ideal of model invariants defines the binary multinomial staged tree model inside the probability simplex, i.e. $\M_{(\T,\LL)}= V(I_{\M(\T,\LL)})\cap \Delta_{n-1}^{\circ}$.
\end{corollary}
\begin{proof}
The variety $V(I_{\M(\T,\LL)}: \mathbf{P}^{\infty})$ exactly describes the points in $V(I_{\M(\T,\LL)})$ that are not in   $V(\mathbf{P})$. The latter variety contains the boundary of the simplex, hence  restricting to positive points that add to one, yields
$\M_{(\T,\LL)}= V(I_{\M(\T,\LL)})\cap \Delta_{n-1}^{\circ}$.
\end{proof}

\subsection{Toric binary multinomial staged tree models} \label{sec:toricmultitrees}
It is not true in general the the ideal $\ker(\Psi_{\T})$ of a multinomial staged tree is
toric. For the case of staged trees,
a characterisation of when $\ker(\Psi_{\T})$ is equal to a subideal
generated by binomials is available in \cite{Duarte2020}. The goal of this section is to establish a similar criterion, based on interpolating polynomials from Definition~\ref{def:interpolating}, for multinomial
staged trees. This criterion  will allow us to study the
polyhedral geometry of these models in Section~\ref{sec:toricmultinomialtrees}. 

\begin{definition}
\label{def:balanced}
Let $(\T,\LL)$ be a multinomial staged tree and
let $v,w$ be two vertices in the same stage with $\mathrm{im}(\LL_v)=f_{\ell,a}$ and $\mathrm{im}(\LL_w)=f_{\ell,b}$ for some $\ell \in [m]$.
\begin{itemize}
    \item[(1)] The vertex $v$ is balanced if for all $K^1,K^2,K^3,K^4\in \mathbb{N}^{|I_{\ell}|}$
    with $|K^1|=|K^2|=|K^3|=|K^4|=a$ and $K^1+K^2=K^3+K^4$, the next identity holds in $\R[s_i: i\in I]$ \[
    t(v(K^1))t(v(K^2))=t(v(K^3))t(v(K^4)).\]
    \item[(2)] The pair of vertices $v,w$ is balanced if 
    for all tuples $K,K',Q,Q' \in \mathbb{N}^{|I_{\ell}|}$ with $|K|=|K'|=a$ and $|Q|=|Q'|=b$ with $K+Q'=K'+Q$ the following identity holds in $\R[s_i: i\in I]$
\[t(v(K))\cdot t(w(Q'))=
t(v(K'))\cdot t(w(Q)).\]
\end{itemize}
 The multinomial staged tree $(\T,\LL)$ is \emph{balanced} if every vertex is balanced and 
every pair of vertices in the same stage is balanced.
\end{definition}
\begin{remark}
Condition $(1)$ in
Definition~\ref{def:balanced} is an empty condition for florets of degree one. For staged trees, condition $(2)$
specialises to the definition of balanced stated in
\cite{Ananiadi2020}.
\end{remark}
\begin{remark} \label{rmk:balanced two level}
If all root-to-leaf paths in $(\T,\LL)$ have length $1$, then $(\T,\LL)$ is vacuously balanced. If $(\T,\LL)$ has all 
root-to-leaf paths of length $2$, such as $\T_{a,b,d}$ in Figure~\ref{fig: 2Dmultinomialtree}, it suffices to check that the root is balanced. For the other vertices, the conditions in Definition~\ref{def:balanced} reduce
to the trivial equality $1\cdot 1= 1\cdot 1$.
\end{remark}
\begin{theorem}\label{thm:balancedformultinomialtrees}
Let $(\T,\LL)$ be a binary multinomial staged tree. 
The model $\M_{(\T,\LL)}$ is toric if and only if $(\T,\LL)$ is balanced.
\end{theorem}
\begin{proof}
We  prove that $\ker(\Psi_{\T})=\ker(\Psi_{\T}^{\mathrm{toric}})$
if and only if $(\T,\LL)$ is balanced.
Define the ideal $J$ to be generated by all polynomials of the form
\begin{align*}
C_{(K',Q)}P_{[v(K)]}P_{[w(Q')]}-
C_{(K,Q')}P_{[v(K')]}P_{[w(Q)]},\\ C_{(K^3,K^4)}P_{[v(K^1)]}P_{[v(K^2)]}-C_{(K^1,K^2)}P_{[v(K^3)]}P_{[v(K^4)]}\end{align*}
where $v,w\in V$
are in the same stage and  $K,K',Q,Q'$ obey the condition  $(2)$ and $K^1,K^2,K^3,K^4$ obey the condition $(1)$
in Definition~\ref{def:balanced}.\newline
\noindent\underline{Claim 1:} $ J\subset \ker(\Psi_{\T}^\mathrm{toric})$ if and only if $(\T,\LL)$ 
is balanced. \newline
\noindent By Lemma~\ref{lem:evaluation},
\begin{align}
    \Psi_{\T}^{\mathrm{toric}}(P_{[v(K)]}P_{[w(Q')]}) &= 
    \left(\prod_{e\in \lambda_{r, v(K)}}\LL(e)\right) t(v(K)) \left(\prod_{e\in \lambda_{r, w(Q')}}\LL(e)\right)t(w(Q')) \text{ and }\\
    \Psi_{\T}^{\mathrm{toric}}(P_{[v(K')]}P_{[w(Q)]}) &= \left(\prod_{e\in \lambda_{r, v(K')}}\LL(e)\right) t(v(K')) \left(\prod_{e\in \lambda_{r, w(Q)}}\LL(e)\right)t(w(Q)).
\end{align}
Note that the right hand side of the two equations above share the common factor $\prod_{e\in \lambda_{v,w}}\LL(e)$ where 
$\lambda_{v,w}$ is the set of edges in the path from $v$ to $w$. Thus we extract this factor from the two previous equations and multiply times the labels of the edges $v\to v(K),w\to w(Q)$ and $v\to v(K'), w\to w(Q')$, respectively, to further simplify the two expressions into
\begin{align}
    \Psi_{\T}^{\mathrm{toric}}(P_{[v(K)]}P_{[w(Q')]}) &= 
    \prod_{e\in \lambda_{v, w}}\LL(e)\left({a \choose K}{b \choose Q'} \prod_{i\in I_{\ell}}s_i^{k_i+q_i'}\right)t(v(K))t(w(Q'))\text{ and }\\
    \Psi_{\T}^{\mathrm{toric}}(P_{[v(K')]}P_{[w(Q)]}) &= \prod_{e\in \lambda_{v, w}}\LL(e)\left(
    {a\choose K'}{b\choose Q}\prod_{i\in I_{\ell}}s_{i}^{k_i'+q_i}\right) t(v(K'))t(w(Q))
\end{align}
Finally, since $K+Q'=K'+Q$  and $(\T,\LL)$ is balanced, we obtain
\begin{align*}
    &\Psi_{\T}^{\mathrm{toric}}(C_{(K',Q)}P_{[v(K)]}P_{[w(Q')]}-
C_{(K,Q')}P_{[v(K')]}P_{[w(Q)]}) \\
&\qquad=C_{(K,Q')}C_{(K',Q)}\prod_{e\in \lambda_{v,w}} \LL(e)
\prod_{i\in I_{\ell}} s_{i}^{k_i+q_i'}\left(t(v(K))t(w(Q))- t(v(K'))t(w(Q))\right)\\
&\qquad =0
\end{align*}
A similar calculation shows that 
\begin{align*}
\Psi_{\T}^{\mathrm{toric}}(C_{(K^3,K^4)}P_{[v(K^1)]}P_{[v(K^2)]}-C_{(K^1,K^2)}P_{[v(K^3)]}P_{[v(K^4)]})=0    
\end{align*} if $\T$
is balanced. Conversely, note that if $J\subset\ker(\Psi_{\T}^{\mathrm{toric}})$, tracing these equations  backwards implies that $\T$ must be balanced.
\newline
\noindent\underline{Claim 2:} $I_{\M(\T,\LL)}\subset J$. The ideal
$I_{\M(\T,\LL))}$ is the sum of $I_{\mathrm{stages}}$ and
$I_{\mathrm{vertices}}$. By definition, the generators of 
$I_{\mathrm{vertices}}$ are also generators of $J$. Hence it suffices to show that the generators of $I_{\mathrm{stages}}$ are polynomial combinations of
the generators of $J$. From this point on we further assume that $(\T,\LL)$ is binary. Suppose $v,w$ are in the same stage, where $\mathrm{im}(\LL_v)=f_{\ell,a}$, $\mathrm{im}(\LL_{w})=f_{\ell,b}$ and $|I_{l}|=2$.
There are two equations that hold for this stage,
one for each element in $I_\ell$. We will show that
the equation
\begin{align}\label{eq:stage}
    bP_{[w]}\left(\sum_{k_1=1}^{a}k_1P_{[v(k_1,a-k_1)]}\right)- aP_{[v]}\left( \sum_{k_2=1}^{b}
    k_{2}P_{[w(k_2,b-k_2)]}\right),
\end{align}
which is the equation for the first element in 
$I_\ell$, is a combination the generators of $J$, defined at the beginning. The one for the second element in $I_{\ell}$ follows an analogous argument.
We use the following two identities:
\begin{align*}
    bP_{[w]} & = \sum_{k_2=1}^{b}k_2P_{[w(k_2,b-k_2)]}+ \sum_{k_2=1}^bk_2P_{w[(b-k_2,k_2)]} \text{ and }\\
    aP_{[v]}&  = \sum_{k_1=1}^ak_1P_{[v(k_1,a-k_1)]}+\sum_{k_1=1}^{a}k_1P_{[v(a-k_1,k_1)]}.
\end{align*}
Working from equation~(\ref{eq:stage}), using the identities, we have
\begin{align}
    &\left(\sum_{k_2=1}^{b}k_2P_{[w(k_2,b-k_2)]}+ \sum_{k_2=1}^bk_2P_{[w(b-k_2,k_2)]}\right)\left(\sum_{k_1=1}^{a}k_1P_{[v(k_1,a-k_1)]}\right)  \nonumber \\ 
    &\qquad -\left(\sum_{k_1=1}^ak_1P_{[v(k_1,a-k_1)]}+\sum_{k_1=1}^{a}k_1P_{[v(a-k_1,k_1)]}\right)\left( \sum_{k_2=1}^{b}
    k_2P_{[w(k_2,b-k_2)]}\right)  \nonumber\\
    &=\left(\sum_{k_2=1}^{b}k_2P_{[w(b-k_2,k_2)]}\right)
    \left(\sum_{k_1=1}^{a}k_1P_{[v(k_1,a-k_1)]}\right)\nonumber\\
    &\qquad -
    \left(\sum_{k_1=1}^{a}k_1P_{[v(a-k_1,k_1)]}\right)\left( \sum_{k_2=1}^{b}
    k_2P_{[w(k_2,b-k_2)]}\right)  \nonumber\\
    &=\left( \sum_{k_2=1}^{b}\sum_{k_1=1}^{a}k_2k_1P_{[w(b-k_2,k_2)]}P_{[v(k_1,a-k_1)]}\right)\nonumber\\
    &\qquad - \left(\sum_{k_2=1}^{b}\sum_{k_1=1}^{a}k_1k_2
    P_{[w(k_2,b-k_2)]}P_{[v(a-k_1,k_1)]}\right) \label{eq:vers}  \\
    &=\sum_{k_2=1}^{b}\sum_{k_1=1}^a \Big(k_2k_1P_{[w(b-k_2,k_2)]}P_{[v(k_1,a-k_1)]} \nonumber  \\ 
    &\qquad \minus
    (b\!\minus(k_2\minus1))(a\!\minus(k_1\minus1))P_{[w(b-(k_2-1),k_2-1)]}P_{[v(k_1\minus1,a\minus(k_1-1))]}\Big) \label{eq:vers2}
\end{align}
After rearranging the terms  in (\ref{eq:vers}) we get a single double summation.
 Finally, the $(k_1,k_2)$ summand in (\ref{eq:vers2}) is a multiple of the generator of
$J$, where $Q'=(b-k_2,k_2), K=(k_1,a-k_1), K'=(k_1-1,a-(k_1-1)), Q= (b-(k_2-1),k_2-1)$. 
The generator of $J$ corresponding to this
choice of $K,K',Q,Q'$ is
\begin{equation} \label{eq:Jgen}
{b\choose k_2\!-\!1}{a\choose k_1\!-\!a}P_{[w(b-k_2,k_2)]}P_{[v(k_1,a-k_1)]}-{b\choose k}{a\choose k_1}P_{[w(b-(k_2-1),k_2-1)]}P_{[v(k_1,a-(k_1-1))]}.    
\end{equation}
Note that
$(b-k_2,k_2)+(k_1,a-k_1)=(b-(k_2-1),k_2-1)+(k_1-1,a-(k_1-1))$,
thus $K+Q'=K'+Q$. Multiplying equation (\ref{eq:Jgen}) times $\frac{(b\minus(k_2\minus 1))!(a\minus(k_1\minus 1)!k_1!k_2!)}{a!b!}$ gives the $(k_1,k_2)$ summand in (\ref{eq:vers2}). This implies that  (\ref{eq:vers2}) is
a sum of multiples of the generators in $J$, hence $I_{\mathrm{stages}}\subset J$.

Finally, combining Claim 1 and 2 we conclude that $I_{\M(\T,\LL)}\subset J \subset \ker(\Psi_{\T}^{\mathrm{toric}})\subset \ker(\Psi_{\T})$ if and only if $\T$ is balanced. We now saturate this chain of ideals as in Theorem~\ref{thm:saturation} to obtain $(I_{\M(\T,\LL)}: \mathbf{P}^{\infty})=(\ker(\Psi_{\T})^{\mathrm{toric}}:\mathbf{P}^{\infty}) = (\ker(\Psi_{\T}):\mathbf{P}^{\infty})$. But $(\ker(\Psi_{\T})^{\mathrm{toric}}:\mathbf{P}^{\infty})=\ker(\Psi_{\T})^{\mathrm{toric}}$ and $(\ker(\Psi_{\T}):\mathbf{P}^{\infty})= \ker(\Psi_{\T})$ because they are prime ideals. Hence $\ker(\Psi_{\T}^{\mathrm{toric}})$ and $\ker(\Psi_{\T})$ are equal.
\end{proof}

\subsection{Multinomial staged tree models have rational MLE} \label{subsec:rationalMLE}
In this last section on multinomial staged trees we prove that they  have rational MLE. This fact together with Theorem~\ref{thm:balancedformultinomialtrees} establishes Theorem~\ref{thm:main} and
thus provides
a new class of polytopes that have rational linear precision.
\begin{theorem}\label{rationalMLEmultinomialtrees}
The multinomial staged tree model $\M_{(\T,\LL)}$
has rational MLE $\Phi$. The $j$-th coordinate of $\Phi$ is 
\[
\Phi_{j}(u_{1},\ldots, u_{n}) = c_{j}\prod_{i\in I}\hat{\theta_{i}}^{a_{ij}}\;\;,
\text{ where for  } i\in I_{\ell}, \hat{\theta}_{i}=\frac{\sum_{j\in J}u_{j}a_{ij}}{\sum_{i\in I_{\ell}}(\sum_{j\in J}u_{j}a_{ij})}.
\]
\end{theorem}
\begin{proof}
Let $u=(u_1,\ldots,u_n)$ be a vector of counts. The likelihood
function for $\mathcal{M}_{\T,\LL}$ is
\begin{align*}
    L(p|u)&=\prod_{j\in J}p_j^{u_j}= \prod_{j\in J}\left(\prod_{i\in I}c_j \theta_{i}^{a_{ij}}\right)^{u_j} \\
     &= \prod_{j\in J}\left(\prod_{i\in I}c_{j}^{u_j}\theta_{i}^{a_{ij}u_{j}}\right)
     = \left( \prod_{i\in I} c_1^{u_1}\theta_{i}^{a_{i1}u_1}\right)\cdots
     \left(\prod_{i\in I}c_{n}^{u_n}\theta_{i}^{ a_{in}u_n}\right)\\
     &=\left(\prod_{j\in J}c_j^{u_j}\right)\left(\prod_{i\in I}\theta_{i}^{\sum_{j\in J}u_j a_{ij}}\right) \text{ , let } C= \prod_{j\in J}c_j^{u_j}\\
     &= C\left(\prod_{i\in I_1}\theta_{i}^{\sum_{j\in J}u_{j}a_{ij}}\right) \cdots \left(\prod_{i\in I_k}\theta_{i}^{\sum_{j\in J}u_{j}a_{ij}}\right) 
     = C L_1\cdots L_{k},
\end{align*}
where $L_{1},\ldots,L_{m}$ denote the factors before the last equality in the previous line.
The function $L(p|u)$ is maximised when each factor is maximised.  This is because the parameters
are partitioned by $I_1,\ldots, I_m$ and hence each factor is independent. Thus we find the maximisers of each factor. The function $L_{\ell}$, $\ell \in [m]$, is
the likelihood function of the saturated model $\Delta_{|I_{\ell}|-1}$ with  parameters
$(\theta_{i})_{i\in I_{\ell}}$ and vector of counts
$\left(\sum_{j\in J}u_{j}a_{ij} \right)_{i\in I_{\ell}}.
$
Therefore $\hat{\theta}_{i}=\frac{\sum_{j\in J}u_{j}a_{ij}}{\sum_{i\in I_{\ell}}\sum_{j\in J}u_{j}a_{ij}}$, $i\in I_{\ell}$.
\end{proof}

\begin{corollary}\label{cor: Horn matrix for multinomial trees}
 Let $H_{(\T,\LL)}$ be the
$(|I|+m)\times(|J|)$ matrix with entries
\[ h_{ij}=a_{ij},\;\;\; i\in I,\;\;\; h_{\ell j} = -\sum_{i\in I_{\ell}}a_{ij},\;\;\; \ell \in [m], \text{ and }\lambda_{j}:= (-1)^{\sum_{i\in I}a_{ij}}c_j .\]
Then $(H_{(\T,\LL)},\lambda)$ is a Horn pair for $\M_{(\T,\LL)}$.
\end{corollary}
\begin{proof}
It suffices to check that the $j$-th coordinate of $\varphi_{(H,\lambda)}$
is equal to $\Phi_{j}$ in Theorem~\ref{rationalMLEmultinomialtrees}. Let $u=(u_1,\ldots,u_n)$, then
\begin{align*}
    (H_{(\T,\LL)}u)^{T}=\left( 
    \sum_{j\in J} a_{1j}u_j,\ldots,\sum_{j\in J}a_{|I|j}u_{j},
    -\sum_{j\in J}( \sum_{i\in I_1} a_{ij})u_{j},\ldots, -\sum_{j\in J}(\sum_{i\in I_{m}}a_{ij})u_{j}
    \right).
\end{align*}
The $j$-th coordinate of $\varphi_{(H_{(\T,\LL)},\lambda)}$ is
\begin{align*}
    \lambda_{j}(H_{(\T,\LL)}u)^{h_{j}}&=  \frac{(-1)^{\sum_{i\in I}a_{ij}}c_{j}\left(\sum_{j\in J}a_{1j} u_j\right)^{a_{1j}}\cdots \left(\sum_{j\in J}a_{|I|j}u_j\right)^{a_{|I|j}}}{
    \left(-\sum_{j\in J}( \sum_{i\in I_1} a_{ij})u_{j}\right)^{\sum_{i\in I_1}a_{ij}}\cdots 
    \left( -\sum_{j\in J} (\sum_{i\in I_m} a_{ij})\right)^{\sum_{i\in I_{k}}a_{ij}}}\\
    &= c_{j}\cdot \frac{\left(\sum_{j\in J}a_{1j} u_j\right)^{a_{1j}}\cdots \left(\sum_{j\in J}a_{|I|j}u_j\right)^{a_{|I|j}}}{
    \left(\sum_{j\in J}( \sum_{i\in I_1} a_{ij})u_{j}\right)^{\sum_{i\in I_1}a_{ij}}\cdots 
    \left( \sum_{j\in J} (\sum_{i\in I_m} a_{ij})\right)^{\sum_{i\in I_{k}}a_{ij}}}\\
    &= c_j \cdot  \left(\frac{\sum_{j\in J}a_{1j}u_{j}}{\sum_{j\in J}(\sum_{i\in I_1}a_{ij})u_j}\right)^{a_{1j}} \cdots 
    \left(\frac{\sum_{j\in J}a_{|I|j}u_{j}}{\sum_{j\in J}(\sum_{i\in I_m}a_{ij})u_j}\right)^{a_{|I|j}}\\ &=c_{j}\prod_{i\in I}\hat{\theta}_{i}^{a_{ij}} = \Phi_{j}.
\end{align*}
\end{proof}

\section{Polytopes arising from  toric multinomial staged trees}
\label{sec:toricmultinomialtrees}
\noindent The aim of this section is to bring together the examples of 2D and 3D polytopes with rational linear precision (Section \ref{sec:2Dand3Dpolytopes}) and multinomial staged trees (Section \ref{multinomialstagedtreemodels}). To this end, we investigate certain properties of the lattice polytopes arising from toric multinomial staged trees. This leads to a better understanding of the negative part of the Horn matrix than that provided by the primitive collections. Recall that $J$ denotes the set of root-to-leaf paths in $\T$. For $j \in J$, $p_j$ is defined to be the product of all edge labels in the path $j$. We denote the stages of $(\T,\LL)$ by $S_1, \ldots, S_m$. Throughout this section $m$ is a positive integer as in Section~\ref{multinomialstagedtreemodels} and $m_j$ ($m$ with a subindex) denotes a lattice point as in Section~\ref{sec:2Dand3Dpolytopes}.

\begin{definition}
The lattice polytope $P_{\T}$ of a balanced multinomial staged tree $(\T,\LL)$ is the convex hull of exponent vectors $a_j$ of $p_j$ for every root-to-leaf path $j$ in $\T$.
\end{definition} 
\noindent Note that $P_{\T}\subset \mathbb{R}^d$ is not a full-dimensional polytope for $d=|S_1|+\dotsm+|S_m|$. This can be observed e.g.\ in Figure \ref{fig:my_2dtrees} (left) for $P_{\T_{b\Delta_2}} \cong b\Delta_2$ (unimodularly equivalent). We call $(\T, \LL)$ a \emph{multinomial staged tree representation} of a full-dimensional polytope $P \cong P_{\T}$.

\subsection{Two dimensional multinomial staged tree models}\label{sec: two dimensional tree models}
The polytopes in 2D from Section~\ref{toricsurfacepatches} admit a multinomial staged tree representation.
\begin{proposition} \label{prop:2Dmultitrees}
All statistical models associated to pairs $(P,w)$ in $2D$ with rational linear precision are toric multinomial staged tree models.
The multinomial staged tree representations for each family in $2D$ are described in
Figure~\ref{fig:my_2dtrees}.
\end{proposition}
\begin{proof}
For the model $b\Delta_2=T_{0,b,1}$, it suffices to note that the polytope $b\Delta_n$ with weights given by multinomial coefficients has a Horn pair given by Theorem~\ref{thm:strictlinhorn} which is
equal to that one described in \cite[Example 20]{Duarte2020} for multinomial models with $b$ trials
and $n+1$ outcomes. The statistical model for  $T_{a,b,d}$ is the binary multinomial staged tree $\mathcal{M}_{a,b,d}$ in Example~\ref{ex:2dimtrees}, denoted by $\T_{a,b,d}$ in Figure~\ref{fig:my_2dtrees}. The Horn matrix in Proposition~\ref{prop:trapezoid2D},
associated to the model for $T_{a,b,d}$, is equal to the Horn matrix of the model $\mathcal{M}_{a,b,d}$. Firstly,  in both cases the columns are indexed by pairs $(i,j)$ such that
$0\leq j\leq b, 0\leq i\leq a+d(b-j)$ so these matrices have the same number of columns. Using Corollary~\ref{cor: Horn matrix for multinomial trees}, we see that the column corresponding to the outcome $(i,j)$ in $\mathcal{M}_{a,b,d}$ is $(i,j,a+d(b-j)-i,b-j,-(a+d(b-j)),-b)$, which equals the column associated to the lattice point $(i,j)$ in Proposition~\ref{prop:trapezoid2D}.
Uniqueness of the minimal Horn matrix, implies that the model associated to $T_{a,b,d}$ is $\mathcal{M}_{a,b,d}$.
\newline
\noindent It remains to show that $\T_{b \Delta_2}$ and $\T_{a,b,d}$ are balanced. 
By Remark~\ref{rmk:balanced two level}, $\T_{b \Delta_2}$ is balanced
because all root-to-leaf paths have length 1. For $\T_{a,b,d}$, it suffices to prove that the root $r$ is balanced. Following the notation in Definition~\ref{def:balanced}, let $K^1=(j_1,b-j_1)$, $K^2=(j_2,b-j_2)$, $K^3=(j_3,b-j_3)$, and $K_4=(j_4,b-j_4)$ be such that $K^1+K^2=K^3+K^4$.
Then 
\begin{align*}
    t(r(K^1))t(r(K^2))& =(s_2+s_3)^{a+d(b-j_1)}(s_2+s_3)^{a+d(b-j_2)} 
                       = (s_2+s_3)^{2a+2db-d(j_1+j_2)}\\ &= (s_2+s_3)^{2a+2db-d(j_3+j_4)} 
                       =(s_2+s_3)^{a+d(b-j_3)}(s_2+s_3)^{a+d(b-j_4)}\\ &= t(r(K^3))t(r(K^4))
\end{align*}
\end{proof}

\begin{figure} 
     \centering
     \begin{subfigure}[c]{0.45\linewidth}
     \begin{tikzpicture}[thick,scale=0.3]
	

     \node[circle, draw, fill=green!0, inner sep=2pt, minimum width=2pt] (r) at (0,0) {};

 	 \node[circle, draw, fill=green!0, inner sep=2pt, minimum width=2pt] (v1) at (5,7) {};
 	 \node[circle, draw, fill=green!0, inner sep=2pt, minimum width=2pt] (v2) at (5,-7) {};
 	 
 	 \node[circle, draw, fill=green!0, inner sep=2pt, minimum width=2pt] (w1) at (10,12) {};
	 \node (d) at (10,10) {$\vdots$};
 	 \node[circle, draw, fill=green!0, inner sep=2pt, minimum width=2pt] (w2) at (10,7) {};
	 \node (d) at (10,5) {$\vdots$};
 	 \node[circle, draw, fill=green!0, inner sep=2pt, minimum width=2pt] (w3) at (10,2) {};
	 
 	 \node[circle, draw, fill=green!0, inner sep=2pt, minimum width=2pt] (w4) at (10,-2) {};
	 \node (d) at (10,-4) {$\vdots$};
 	 \node[circle, draw, fill=green!0, inner sep=2pt, minimum width=2pt] (w5) at (10,-7) {};
	 \node (d) at (10,-9) {$\vdots$};
 	 \node[circle, draw, fill=green!0, inner sep=2pt, minimum width=2pt] (w6) at (10,-12) {};
 	 
 	 \node[circle, draw, fill=green!0, inner sep=2pt, minimum width=2pt] (w11) at (15,13.5) {};
	 \node (d) at (15,12.15) {$\vdots$};
 	 \node[circle, draw, fill=green!0, inner sep=2pt, minimum width=2pt] (w12) at (15,10.5) {};
 	 \node[circle, draw, fill=green!0, inner sep=2pt, minimum width=2pt] (w21) at (15,8.5) {};
	 \node (d) at (15,7.15) {$\vdots$};
 	 \node[circle, draw, fill=green!0, inner sep=2pt, minimum width=2pt] (w22) at (15,5.5) {};
 	 \node[circle, draw, fill=green!0, inner sep=2pt, minimum width=2pt] (w31) at (15,3.5) {};
 	 \node (d) at (15,2.15) {$\vdots$};
 	 \node[circle, draw, fill=green!0, inner sep=2pt, minimum width=2pt] (w32) at (15,0.5) {};
 	 \node[circle, draw, fill=green!0, inner sep=2pt, minimum width=2pt] (w41) at (15,-0.5) {};
	 \node (d) at (15,-1.85) {$\vdots$};
 	 \node[circle, draw, fill=green!0, inner sep=2pt, minimum width=2pt] (w42) at (15,-3.5) {};
 	 \node[circle, draw, fill=green!0, inner sep=2pt, minimum width=2pt] (w51) at (15,-5.5) {};
	 \node (d) at (15,-6.85) {$\vdots$};
 	 \node[circle, draw, fill=green!0, inner sep=2pt, minimum width=2pt] (w52) at (15,-8.5) {};
 	 \node[circle, draw, fill=green!0, inner sep=2pt, minimum width=2pt] (w61) at (15,-10.5) {};
	 \node (d) at (15,-11.85) {$\vdots$};
 	 \node[circle, draw, fill=green!0, inner sep=2pt, minimum width=2pt] (w62) at (15,-13.5) {};

	 \draw[->]   (r) -- node[midway,sloped,above]{\footnotesize$s_0$}    (v1) ;
 	 \draw[->]   (r) -- node[midway,sloped,below]{\footnotesize$s_1$}  (v2) ;
	 
	 \draw[->]   (v1) -- node[midway,sloped,above]{}    (w1) ;
	 \draw[->]   (v1) -- node[midway,sloped,above]{\footnotesize$\phantom{++}s_2^js_3^{b-j}$}    (w2) ;
	 \draw[->]   (v1) -- node[midway,sloped,above]{}    (w3) ;
	 
 	 \draw[->]   (v2) -- node[midway,sloped,above]{}  (w4) ;
 	 \draw[->]   (v2) -- node[midway,sloped,above]{\footnotesize$\phantom{++}s_2^js_3^{b'-j}$}  (w5) ;
 	 \draw[->]   (v2) -- node[midway,sloped,above]{}  (w6) ;

	 \draw[->]   (w1) -- node[midway,sloped,above]{}    (w11) ;
     \draw[->]   (w1) -- node[midway,sloped,above]{}    (w12) ;
     \draw[->]   (w2) -- node[midway,sloped,above]{\footnotesize$\phantom{+}s_4^{a+d(b-j)}$}    (w21) ;
     \draw[->]   (w2) -- node[midway,sloped,below]{\footnotesize$\phantom{+}s_5^{a+d(b-j)}$}    (w22) ;
     \draw[->]   (w3) -- node[midway,sloped,above]{}    (w31) ;
     \draw[->]   (w3) -- node[midway,sloped,above]{}    (w32) ;
     \draw[->]   (w4) -- node[midway,sloped,above]{}    (w41) ;
     \draw[->]   (w4) -- node[midway,sloped,above]{}    (w42) ;
     \draw[->]   (w5) -- node[midway,sloped,above]{\footnotesize$\phantom{+}s_4^{a'+d(b'-j)}$}    (w51) ;
     \draw[->]   (w5) -- node[midway,sloped,below]{\footnotesize$\phantom{+}s_5^{a'+d(b'-j)}$}    (w52) ;
     \draw[->]   (w6) -- node[midway,sloped,above]{}    (w61) ;
     \draw[->]   (w6) -- node[midway,sloped,above]{}    (w62) ;

	 \node at (0,1.5) {\footnotesize$r$} ;
	 \node at (-1,12) {$\T_{(A)_1}$} ;

\end{tikzpicture}
    \end{subfigure}\qquad
    \begin{subfigure}[c]{0.45\linewidth}
     \begin{tikzpicture}[thick,scale=0.3]
	

     \node[circle, draw, fill=green!0, inner sep=2pt, minimum width=2pt] (r) at (0,0) {};

 	 \node[circle, draw, fill=green!0, inner sep=2pt, minimum width=2pt] (v1) at (5,7) {};
 	 \node[circle, draw, fill=green!0, inner sep=2pt, minimum width=2pt] (v2) at (5,-7) {};
 	 
 	 \node[circle, draw, fill=green!0, inner sep=2pt, minimum width=2pt] (w1) at (10,12) {};
	 \node (d) at (10,9.85) {$\vdots$};
 	 \node[circle, draw, fill=green!0, inner sep=2pt, minimum width=2pt] (w2) at (10,7) {};
	 \node (d) at (10,5) {$\vdots$};
 	 \node[circle, draw, fill=green!0, inner sep=2pt, minimum width=2pt] (w3) at (10,2) {};
	 
 	 \node[circle, draw, fill=green!0, inner sep=2pt, minimum width=2pt] (w4) at (10,-2) {};
	 \node (d) at (10,-4) {$\vdots$};
 	 \node[circle, draw, fill=green!0, inner sep=2pt, minimum width=2pt] (w5) at (10,-7) {};
	 \node (d) at (10,-9) {$\vdots$};
 	 \node[circle, draw, fill=green!0, inner sep=2pt, minimum width=2pt] (w6) at (10,-12) {};
 	 
 	 \node[circle, draw, fill=green!0, inner sep=2pt, minimum width=2pt] (w11) at (15,13.5) {};
	 \node (d) at (15,12.15) {$\vdots$};
 	 \node[circle, draw, fill=green!0, inner sep=2pt, minimum width=2pt] (w12) at (15,10.5) {};
 	 \node[circle, draw, fill=green!0, inner sep=2pt, minimum width=2pt] (w21) at (15,8.5) {};
	 \node (d) at (15,7.15) {$\vdots$};
 	 \node[circle, draw, fill=green!0, inner sep=2pt, minimum width=2pt] (w22) at (15,5.5) {};
 	 \node[circle, draw, fill=green!0, inner sep=2pt, minimum width=2pt] (w31) at (15,3.5) {};
	 \node (d) at (15,2.15) {$\vdots$};
 	 \node[circle, draw, fill=green!0, inner sep=2pt, minimum width=2pt] (w32) at (15,0.5) {};
 	 \node[circle, draw, fill=green!0, inner sep=2pt, minimum width=2pt] (w51) at (15,-5.5) {};
	 \node (d) at (15,-6.85) {$\vdots$};
 	 \node[circle, draw, fill=green!0, inner sep=2pt, minimum width=2pt] (w52) at (15,-8.5) {};
 	 \node[circle, draw, fill=green!0, inner sep=2pt, minimum width=2pt] (w61) at (15,-10.5) {};
	 \node (d) at (15,-11.85) {$\vdots$};
 	 \node[circle, draw, fill=green!0, inner sep=2pt, minimum width=2pt] (w62) at (15,-13.5) {};

	 \draw[->]   (r) -- node[midway,sloped,above]{\footnotesize$s_0$}    (v1) ;
 	 \draw[->]   (r) -- node[midway,sloped,below]{\footnotesize$s_1$}  (v2) ;
	 
	 \draw[->]   (v1) -- node[midway,sloped,above]{}    (w1) ;
	 \draw[->]   (v1) -- node[midway,sloped,above]{\footnotesize$\phantom{++}s_2^js_3^{b-j}$}    (w2) ;
	 \draw[->]   (v1) -- node[midway,sloped,above]{}    (w3) ;
	 
 	 \draw[->]   (v2) -- node[midway,sloped,above]{}  (w4) ;
 	 \draw[->]   (v2) -- node[midway,sloped,above]{\footnotesize$\phantom{++}s_2^js_3^{b'-j}$}  (w5) ;
 	 \draw[->]   (v2) -- node[midway,sloped,above]{}  (w6) ;

	 \draw[->]   (w1) -- node[midway,sloped,above]{}    (w11) ;
     \draw[->]   (w1) -- node[midway,sloped,above]{}    (w12) ;
     \draw[->]   (w2) -- node[midway,sloped,above]{\footnotesize$\phantom{+}s_4^{a+d(b-j)}$}    (w21) ;
     \draw[->]   (w2) -- node[midway,sloped,below]{\footnotesize$\phantom{+}s_5^{a+d(b-j)}$}    (w22) ;
     \draw[->]   (w3) -- node[midway,sloped,above]{}    (w31) ;
     \draw[->]   (w3) -- node[midway,sloped,above]{}    (w32) ;
     \draw[->]   (w5) -- node[midway,sloped,above]{\footnotesize$\phantom{+}s_4^{d(b'-j)}$}    (w51) ;
     \draw[->]   (w5) -- node[midway,sloped,below]{\footnotesize$\phantom{+}s_5^{d(b'-j)}$}    (w52) ;
     \draw[->]   (w6) -- node[midway,sloped,above]{}    (w61) ;
     \draw[->]   (w6) -- node[midway,sloped,above]{}    (w62) ;

	 \node at (0,1.5) {\footnotesize$r$} ;
	 \node at (-1,12) {$\T_{(A)_2}$} ;

\end{tikzpicture}
    \end{subfigure}
\vfill
 \bigskip \bigskip \bigskip
\begin{subfigure}[c]{0.45\linewidth}
     \begin{tikzpicture}[thick,scale=0.3]
	

     \node[circle, draw, fill=green!0, inner sep=2pt, minimum width=2pt] (r) at (0,0) {};

 	 \node[circle, draw, fill=green!0, inner sep=2pt, minimum width=2pt] (v1) at (5,7) {};
 	 \node[circle, draw, fill=green!0, inner sep=2pt, minimum width=2pt] (v2) at (5,-7) {};
 	 
 	 \node[circle, draw, fill=green!0, inner sep=2pt, minimum width=2pt] (w1) at (10,12) {};
	 \node (d) at (10,10) {$\vdots$};
 	 \node[circle, draw, fill=green!0, inner sep=2pt, minimum width=2pt] (w2) at (10,7) {};
	 \node (d) at (10,5) {$\vdots$};
 	 \node[circle, draw, fill=green!0, inner sep=2pt, minimum width=2pt] (w3) at (10,2) {};
	 
 	 \node[circle, draw, fill=green!0, inner sep=2pt, minimum width=2pt] (w4) at (10,-2) {};
	 \node (d) at (10,-4) {$\vdots$};
 	 \node[circle, draw, fill=green!0, inner sep=2pt, minimum width=2pt] (w5) at (10,-7) {};
	 \node (d) at (10,-9) {$\vdots$};
 	 \node[circle, draw, fill=green!0, inner sep=2pt, minimum width=2pt] (w6) at (10,-12) {};
 	 
 	 \node[circle, draw, fill=green!0, inner sep=2pt, minimum width=2pt] (w11) at (15,13.5) {};
	 \node (d) at (15,12.15) {$\vdots$};
 	 \node[circle, draw, fill=green!0, inner sep=2pt, minimum width=2pt] (w12) at (15,10.5) {};
 	 \node[circle, draw, fill=green!0, inner sep=2pt, minimum width=2pt] (w21) at (15,8.5) {};
	 \node (d) at (15,7.15) {$\vdots$};
 	 \node[circle, draw, fill=green!0, inner sep=2pt, minimum width=2pt] (w22) at (15,5.5) {};
  	 \node[circle, draw, fill=green!0, inner sep=2pt, minimum width=2pt] (w31) at (15,3.5) {};
 	 \node (d) at (15,2.15) {$\vdots$};
  	 \node[circle, draw, fill=green!0, inner sep=2pt, minimum width=2pt] (w32) at (15,0.5) {};

	 \draw[->]   (r) -- node[midway,sloped,above]{\footnotesize$s_0$}    (v1) ;
 	 \draw[->]   (r) -- node[midway,sloped,below]{\footnotesize$s_1$}  (v2) ;
	 
	 \draw[->]   (v1) -- node[midway,sloped,above]{}    (w1) ;
	 \draw[->]   (v1) -- node[midway,sloped,above]{\footnotesize$\phantom{++}s_2^js_3^{b-j}$}    (w2) ;
	 \draw[->]   (v1) -- node[midway,sloped,above]{}    (w3) ;
	 
 	 \draw[->]   (v2) -- node[midway,sloped,above]{}  (w4) ;
 	 \draw[->]   (v2) -- node[midway,sloped,above]{\footnotesize$\phantom{++}s_4^{a'-i}s_5^i$}  (w5) ;
 	 \draw[->]   (v2) -- node[midway,sloped,above]{}  (w6) ;

	 \draw[->]   (w1) -- node[midway,sloped,above]{}    (w11) ;
     \draw[->]   (w1) -- node[midway,sloped,above]{}    (w12) ;
     \draw[->]   (w2) -- node[midway,sloped,above]{\footnotesize$\phantom{+}s_4^{a+d(b-j)}$}    (w21) ;
     \draw[->]   (w2) -- node[midway,sloped,below]{\footnotesize$\phantom{+}s_5^{a+d(b-j)}$}    (w22) ;
     \draw[->]   (w3) -- node[midway,sloped,above]{}    (w31) ;
     \draw[->]   (w3) -- node[midway,sloped,above]{}    (w32) ;

	 \node at (0,1.5) {\footnotesize$r$} ;
	 \node at (-1,12) {$\T_{(A)_3}$} ;

\end{tikzpicture}
    \end{subfigure}\qquad 
    \begin{subfigure}[c]{0.45\linewidth}
     \vspace{-1.4cm}
     \begin{tikzpicture}[thick,scale=0.3]
	

     \node[circle, draw, fill=green!0, inner sep=2pt, minimum width=2pt] (r) at (0,0) {};

 	 \node[circle, draw, fill=green!0, inner sep=2pt, minimum width=2pt] (v1) at (5,7) {};
 	 \node[circle, draw, fill=green!0, inner sep=2pt, minimum width=2pt] (v2) at (5,-7) {};
 	 
 	 \node[circle, draw, fill=green!0, inner sep=2pt, minimum width=2pt] (w1) at (10,12) {};
	 \node (d) at (10,10) {$\vdots$};
 	 \node[circle, draw, fill=green!0, inner sep=2pt, minimum width=2pt] (w2) at (10,7) {};
	 \node (d) at (10,5) {$\vdots$};
 	 \node[circle, draw, fill=green!0, inner sep=2pt, minimum width=2pt] (w3) at (10,2) {};

 	 \node[circle, draw, fill=green!0, inner sep=2pt, minimum width=2pt] (w11) at (15,13.5) {};
	 \node (d) at (15,12.15) {$\vdots$};
 	 \node[circle, draw, fill=green!0, inner sep=2pt, minimum width=2pt] (w12) at (15,10.5) {};
 	 \node[circle, draw, fill=green!0, inner sep=2pt, minimum width=2pt] (w21) at (15,8.5) {};
	 \node (d) at (15,7.15) {$\vdots$};
 	 \node[circle, draw, fill=green!0, inner sep=2pt, minimum width=2pt] (w22) at (15,5.5) {};
 	 \node[circle, draw, fill=green!0, inner sep=2pt, minimum width=2pt] (w31) at (15,3.5) {};
	 \node (d) at (15,2.15) {$\vdots$};
 	 \node[circle, draw, fill=green!0, inner sep=2pt, minimum width=2pt] (w32) at (15,0.5) {};

	 \draw[->]   (r) -- node[midway,sloped,above]{\footnotesize$s_0$}    (v1) ;
 	 \draw[->]   (r) -- node[midway,sloped,below]{\footnotesize$s_1$}  (v2) ;
	 
	 \draw[->]   (v1) -- node[midway,sloped,above]{}    (w1) ;
	 \draw[->]   (v1) -- node[midway,sloped,above]{\footnotesize$\phantom{++}s_2^js_3^{b-j}$}    (w2) ;
	 \draw[->]   (v1) -- node[midway,sloped,above]{}    (w3) ;
	 

	 \draw[->]   (w1) -- node[midway,sloped,above]{}    (w11) ;
     \draw[->]   (w1) -- node[midway,sloped,above]{}    (w12) ;
     \draw[->]   (w2) -- node[midway,sloped,above]{\footnotesize$\phantom{+}s_4^{a+d(b-j)}$}    (w21) ;
     \draw[->]   (w2) -- node[midway,sloped,below]{\footnotesize$\phantom{+}s_5^{a+d(b-j)}$}    (w22) ;
     \draw[->]   (w3) -- node[midway,sloped,above]{}    (w31) ;
     \draw[->]   (w3) -- node[midway,sloped,above]{}    (w32) ;

	 \node at (0,1.5) {\footnotesize$r$} ;
	 \node at (-1,12) {$\T_{(A)_4}$} ;

\end{tikzpicture}
    \end{subfigure}
     \caption{Multinomial staged trees in 3D for pairs in Table~\ref{table:3Dtrapezoids}(A).}
     \label{fig: 2Dmultinomialtree}
 \end{figure}
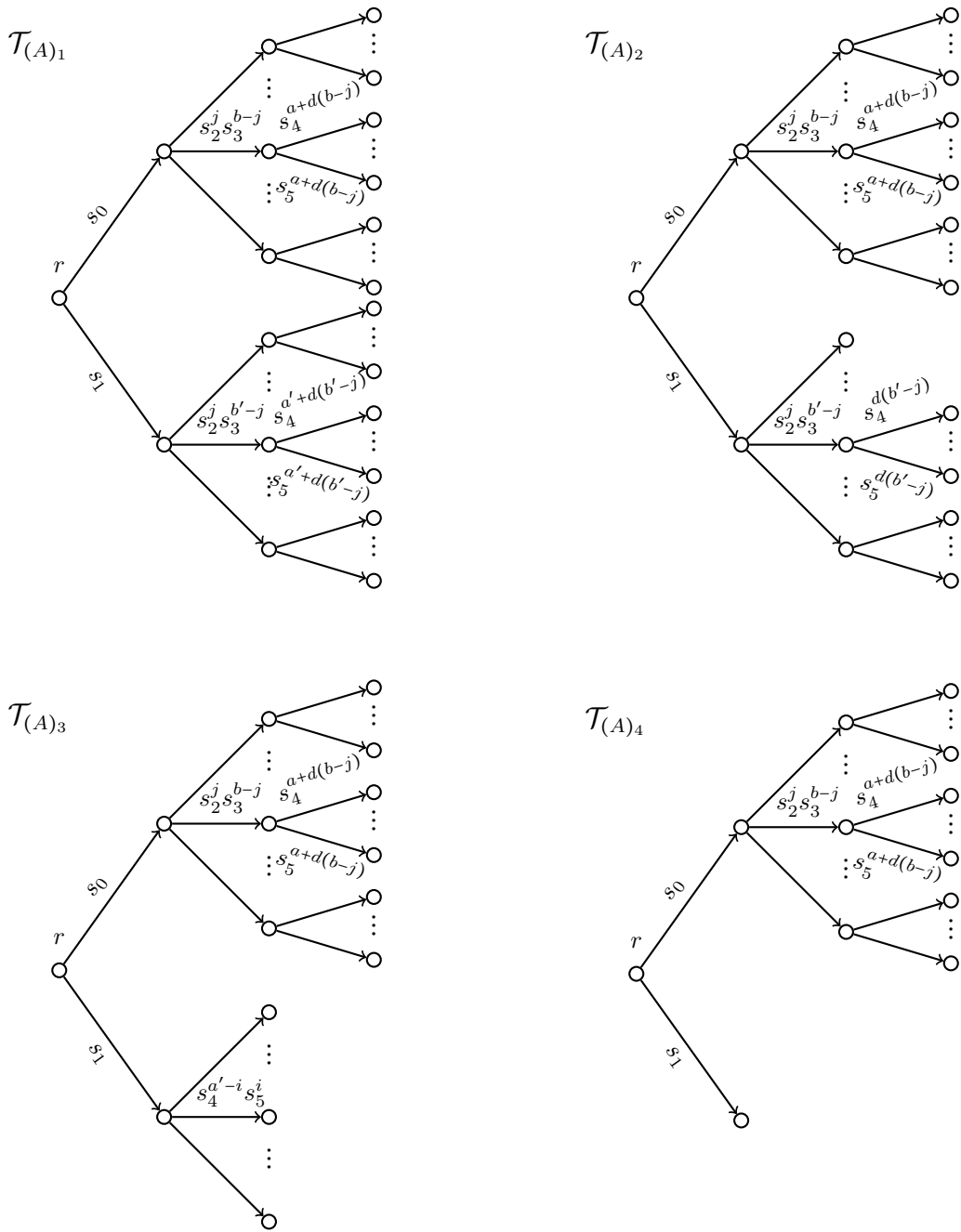








\noindent Note that we obtain $P_{\T_{0,b,1}} \cong b \Delta_2$ i.e.\ we have two different tree representations of $b\Delta_2$: $\T_{b \Delta_2}$ and $\T_{0,b,1}$. For the investigation of the shape of a Horn matrix, we will be interested in those trees $\T$ where the positive part of the Horn matrix $H_{(\T,\LL)}$ from Corollary \ref{cor: Horn matrix for multinomial trees} is the lattice distance matrix of $P_{\T}$ (Definition \ref{def: minimaltreerepresentation})(1). For simple polytopes $P_{\T}$, these trees with an additional property provide us an explanation for the negative part of $H_{(\T,\LL)}$ in terms of primitive collections in Theorem \ref{thm: simplepolytopeprimitivecollections}.  

\subsection{Three dimensional binary multinomial staged tree models}\label{sec: three dimensional multinomial trees}
Before we examine the multinomial staged tree representations more generally, we present the multinomial staged trees for the family $\mathcal{P}$ in Section~\ref{sec:2Dand3Dpolytopes}.
\begin{proposition} \label{prop:3dmultitrees}
All statistical models associated to pairs   in $\mathcal{P}$ are toric binary   multinomial staged trees.
\end{proposition}
\begin{proof}
We first show that the Horn matrix of the statistical model associated to a general element
in $\mathcal{P}$ is equal to the Horn matrix of a  binary multinomial staged tree.
The general element in $\mathcal{P}$ is a frustum with parameters $a,a',b,b',d,l>0$.
Let $S=\{\{s_0,s_1\},\{s_2,s_3\},\{s_4,s_5\}\}$ be a set of symbols. We define the
labelled tree $(\T,\LL)$ by specifying its set of leaves, its set of root-to-leaf paths,
and the labelling for the edges in each path. The set of leaves in $\T$ is
\begin{align*}
    J=\{(i,j,k)&: 0\leq k \leq l,\;\; 0\leq j \leq bl -(b-b')k, \\
    & 0\leq i \leq (a+db)l-((a+db)-(a'+db')k-dj)\}.
\end{align*}
The labelled root-to-leaf path that ends at leaf $(i,j,k)$ is
$r\to v \to w \to (i,j,k)$, where
\begin{align*}
    \LL(r\to v) &= \tbinom{l}{k}s_0^{k}s_1^{l-k}, \\
    \LL(v\to w) &= \tbinom{b-(b-b')k}{ j} s_2^j s_3^{bl-(b-b')k-j}, \text{ and }\\
    \LL(w\to (i,j,k))&= \tbinom{(a+db)l-((a+db)-(a+db'))k-dj}{ i}s_4^i s_5 ^{(a+db)l-((a+db)-(a+db'))k-dj -i}.
\end{align*}

Thus $p_{(i,j,k)}$ is the product of the labels in this path. A picture of this tree when $l=1$ is contained in Figure~\ref{fig: 2Dmultinomialtree} as $\T_{(A)_1}$. Using Corollary~\ref{cor: Horn matrix for multinomial trees}, applied to $(\T,\LL)$ just defined,
we see that the column corresponding to $(i,j,k)$ in the Horn matrix for $\M_{(\T,\LL)}$ is 
equal to the column of the matrix in Proposition~\ref{thm:bigfamily} evaluated at $(i,j,k)$. \newline
\noindent It remains to prove that $(\T,\LL)$ is balanced. Let us first prove that the root $r$ of $(\mathcal{T},\mathcal{L})$ is balanced. The exponents of the outgoing edges of $r$ can be written as pairs of natural numbers that sum to the degree $\ell$ of the floret. Thus they are pairs of the form $(k,\ell-k).$ Let us consider four such pairs, denoted by $Q_1:=(k_1,\ell -k_1)$, $Q_2:=(k_2,\ell -k_2)$,
$Q_3:=(k_3,\ell -k_3)$,
$Q_4:=(k_4,\ell -k_4)$. Suppose that $Q_1+Q_2=Q_3+Q_4$. Then we have $  k_1+k_2=k_3+k_4.$

We further need to check the following equality:
\begin{equation*}
    t\big (r(Q_1)\big )t\big (r(Q_2)\big )=t\big (r(Q_3)\big )t\big (r(Q_4)\big ).
\end{equation*}
We have: 
\begin{align*}
    t\big (r(Q_1)\big )&=\sum_{j=0}^{b\ell -(b-b')k_1}\bigg [
    \tbinom{b\ell -(b-b')k_1}{ j}s_2^j s_3^{b\ell -(b-b')k_1-j} t\big (v(j,b\ell -(b-b')k_1-j)\big )\bigg ]\\
    &=\sum_{j=0}^{b\ell -(b-b')k_1}\bigg [\tbinom{b\ell -(b-b')k_1}{ j}s_2^j s_3^{b\ell -(b-b')k_1-j} (s_4+s_5)^{(a+db)\ell -((a+db)-(a'+db'))k_1-dj}\bigg ]\\
    &=(s_4+s_5)^{(a+db)\ell -((a+db)-(a'+db'))k_1} \\
    & \qquad \cdot \sum_{j=0}^{b\ell -(b-b')k_1}\bigg [\tbinom{b\ell -(b-b')k_1}{ j}\bigg (\frac{s_2}{(s_4+s_5)^d}\bigg )^j s_3^{b\ell -(b-b')k_1-j} \bigg ]\\
    &=(s_4+s_5)^{(a+db)\ell -((a+db)-(a'+db'))k_1} \bigg ( \frac{s_2}{(s_4+s_5)^d} +s_3\bigg )^{b\ell -(b-b')k_1}.
\end{align*}
We obtain similar formulae for $t\big (r(Q_2)\big )$, $t\big (r(Q_3)\big )$ and $t\big (r(Q_4)\big )$. It follows that \begin{align*}
    t\big (r(Q_1)\big ) t\big (r(Q_2)\big )&=(s_4+s_5)^{2(a+db)\ell -((a+db)-(a'+db'))(k_1+k_2)}  \\
    &\qquad \cdot \bigg( \frac{s_2}{(s_4+s_5)^d} +s_3\bigg )^{2 b\ell -(b-b')(k_1+k_2)}\\
    &=(s_4+s_5)^{2(a+db)\ell -((a+db)-(a'+db'))(k_3+k_4)}\\ &\qquad \cdot \bigg ( \frac{s_2}{(s_4+s_5)^d} +s_3\bigg )^{2 b\ell -(b-b')(k_3+k_4)} \\
    &=t\big (r(Q_3)\big ) t\big (r(Q_4)\big ).
\end{align*}

\noindent A similar argument can be used to prove that the children of the root are balanced vertices. 
Next, let us denote by $v$ such a vertex, whose parent is $r$. By Remark \ref{rmk:balanced two level}, any child of $v$ is trivially balanced. Finally, we prove that all pairs of vertices in the same stage are balanced. There are three stages $S_1=\{s_0,s_1\},S_2=\{s_2,s_3\}$ and $S_3=\{s_4,s_5\}$. Denote by $v$ and $v'$ two children of the root $r$. 
The exponents of the outgoing edges of $v$ can be written as pairs of natural numbers that sum to the degree $b\ell-(b-b')k_1$ of the floret. Thus they are pairs of the form $(j,b\ell-(b-b')k_1-j).$ Let us consider two such pairs, denoted by $Q_1:=(j_1,b\ell-(b-b')k_1-j_1)$, $Q_2:=(j_2,b\ell-(b-b')k_1-j_2)$. Similarly, we consider two children of $v'$ and we denote by 
$Q_3:=(j_3,b\ell-(b-b')k_2-j_3)$,
$Q_4:=(j_4,b\ell-(b-b')k_2-j_4)$. Suppose that $Q_1+Q_4=Q_2+Q_3$. It follows that $j_1+j_4=j_2+j_3$. We want to prove that  
\begin{equation*}
    t\big (v(Q_1)\big )t\big (v(Q_4)\big )=t\big (v(Q_2)\big )t\big (v(Q_3)\big ).
\end{equation*}
We obtain the following equality:
\begin{align*}
    t\big (v(Q_1)\big )=t\big (v(j_1,b\ell-(b-b')k_1-j_1)\big )=(s_4+s_5)^{(a+db)\ell -((a+db)-(a'+db'))k_1-d j_1},
\end{align*}
and its analogues for $t\big (v(Q_2)\big )$, $t\big (v(Q_3)\big )$ and $t\big (v(Q_4)\big )$.
Then 
\begin{align*}
    t\big (v(Q_1)\big ) t\big (v(Q_4)\big )&= (s_4+s_5)^{2(a+db)\ell -((a+db)-(a'+db'))(k_1+k_2)-d (j_1+j_4)}\\
    &=(s_4+s_5)^{2(a+db)\ell -((a+db)-(a'+db'))(k_1+k_2)-d (j_2+j_3)}\\
    &=t\big (v(Q_2)\big ) t\big (v(Q_3)\big ).
\end{align*}
Therefore the pairs of vertices we considered are balanced.
\end{proof}
To obtain the multinomial staged tree representations for the models of the polytopes
in Table~\ref{table:3Dtrapezoids}, we use the tree in the proof of  Proposition~\ref{prop:3dmultitrees} and specialise the values of the parameters $a,a',b,b',d,l$ accordingly. The trees for the family of prismatoids with trapezoidal base in
Table~\ref{table:3Dtrapezoids} (A), with $l=1$, are depicted in Figure~\ref{fig: 2Dmultinomialtree}. The trapezoidal frusta is represented by
$\T_{(A)_1}$, the upper branch is the  model for  $T_{a,b,d}$ and the lower branch is the the model
for $T_{a',b',d}$. The other trees, $\T_{(A)_2}$, $\T_{(A)_3}$ and $\T_{(A)_4}$,  have the same upper branch as $\T_{(A)_1}$.
For the prismatoid with simplex on top, the substitution $a'=0$ has the effect of chopping a floret from  $\T_{(A)_1}$, this gives $\T_{(A)_2}$.
For the trapezoidal wedge, $b'=0$, the edges in $\T_{(A)_1}$ that contain $b'$ contract to a single vertex, yielding $\T_{(A)_3}$. For the trapezoidal pyramid, $a'=b'=0$, we chop off the lower part of the tree after the edge labelled by $s_1$.
The trees for the remaining part  of the Table~\ref{table:3Dtrapezoids}, (B), (C) and (D)  are obtained similarly.

\subsection{Properties of the polytope $P_{\T}$}
In this section we study certain properties of $P_{\T}$ that can be formulated in terms of the combinatorics of its tree $\T$. We start by looking at root-to-leaf paths in $\T$ that represent vertices of
$P_{\T}$, this allows us to work with the normal fan $\Sigma_{P_{\T}}$. For simplicity we assume that $(\T,\LL)$ has a root-to-leaf path of length $m$ where $S_1, \ldots, S_m$ are the stages of $\T$.
\begin{definition}\label{def: vertexrepresenting}
A root-to-leaf path $j$ is \emph{vertex representing} if the exponent vector $a_j$ of $p_j$ is a vertex of $P_{\T}$.
\end{definition}

\begin{lemma}\label{lem: verticesofpolytope}
Let $P_{\T} \subset \mathbb{R}^d$ be a polytope where $(\T,\LL)$ is 
a multinomial staged tree from Proposition~\ref{prop:2Dmultitrees} or Proposition~\ref{prop:3dmultitrees}. Then the 
vertex representing paths in $P_{\T}$ are those for which $p_j$ is divisible by at most one symbol from
each stage.
\end{lemma}

\begin{proof}
The upper and lower branches of the tree $\T_{(A)_1}$ are the same up to a choice of parameters, thus we prove it only for the upper branch.  Consider a root-to-leaf path $\mathbf{j}$ in $\T_{(A)_1}$ such that $a_{\mathbf{j}} = (1,0,j,b-j, k,a+d(b-j)-k)$ for $0<j<b$ and $0<k<a+d(b-j)$. Let $\mathbf{j}_1$ and $\mathbf{j}_2$ be two root-to-leaf paths such that $a_{\mathbf{j}_1} = (1,0,j,b-j,a+d(b-j),0)$ and $a_{\mathbf{j}_2} = (1,0,j,b-j,0,a+d(b-j))$. Then we obtain the equality $a_{\mathbf{j}} = \frac{k}{a+d(b-j))} a_{\mathbf{j}_1} + \frac{a+d(b-j)-k}{a+d(b-j))} a_{\mathbf{j}_2}$ and hence $a_{\mathbf{j}}$ cannot be a vertex of $P_{\T_{(A)_1}}$. It remains to show that $a_{\mathbf{j}} = (1,0,j,b-j,a+d(b-j),0))$ is not a vertex. Let now $\mathbf{j}'_1$ and $\mathbf{j}'_2$ be two root-to-leaf paths such that $p_{\mathbf{j}'_1}$ and $p_{\mathbf{j}'_2}$ are divisible by one symbol from each stage and $a_{\mathbf{j}'_1} = (1,0,b,0,a,0)$, $a_{\mathbf{j}'_2} = (1,0,0,b, a+db,0)$. Hence $a_{\mathbf{j}} = \frac{j}{b} a_{\mathbf{j}'_1} + \frac{b-j}{b} a_{\mathbf{j}'_2}$ and $a_{\mathbf{j}}$ is not a vertex of $P_{\T_{(A)_1}}$. The proofs for the remaining trees follow similary. 
\end{proof}

\noindent Next, we investigate the relation between primitive collections, Horn matrices, and stages of multinomial staged trees. The following definition was motivated by the observations on the trees from Section \ref{sec: two dimensional tree models}, \ref{sec: three dimensional multinomial trees} and by an attempt to answer Question \ref{q2}. 
\begin{definition}\label{def: minimaltreerepresentation}
Let $(\T, \LL)$ be a balanced multinomial staged tree representation of $P \cong P_{\T}$. We say that the polytope $P_{\T}$ has property $(\star)$ if,
\begin{enumerate}
    \item The positive part of the Horn matrix $H_{(\T,\LL)}$ from Corollary \ref{cor: Horn matrix for multinomial trees} is the lattice distance matrix of $P_{\T}$.
    \item The vertices of $P_{\T}$ satisfy the conclusion of Lemma \ref{lem: verticesofpolytope}.
\end{enumerate}  \end{definition}
\noindent Note that isomorphic polytopes have the same lattice distance matrices.
\begin{lemma}\label{lem: simpleness of PT}
Let $P_{\T}\subset \mathbb{R}^d$ be a polytope with property $(\star)$. Then $P_{\T}$ is simple if and only if all root-to-leaf paths have the same length $m$ and $|S_1| + \dotsm + |S_{m}| - m = \dim(P_{\T})$.
\end{lemma}
\begin{proof}
Recall that we assumed that there exists a root-to-leaf path of length $m$. All root-to-leaf paths have the same length $m$ if and only if all vertex representing root-to-leaf paths have the same length $m$. Recall also that the vertices of $P_{\T}$ are in one-to-one correspondence with the maximal cones of the normal fan $\Sigma_{P_{\T}}$. First suppose that $P_{\T}$ is simple and there exists a vertex representing root-to-leaf path $j'$ of length $m' <m$. Since the positive part of $H_{(\T,\LL)}$ is the  lattice  distance matrix of $P_{\T}$, the symbols which do not divide $p_{j'}$ represent the facets of $P_{\T}$ which are lattice distance $0$ to $a_{j'}$. Since $P_{\T}$ satisfies the conclusion of Lemma \ref{lem: verticesofpolytope}, the maximal cone associated to the vertex $a_{j'}$ in $\Sigma_{P_{\T}}$ has more 1-face (ray) generators than the one associated to $a_{j}$ where $j$ has length $m$. Thus $\Sigma_{P_{\T}}$ is not simplicial, contradiction. Moreover we obtain that the maximal cone associated to a vertex $a_j$ is generated by the normal vectors associated to $\bigcup_{l=1}^m S_l \backslash s_{i_l}$ for some $s_{i_l} \in S_l$ and where $\prod_{l=1}^m s_{i_l}$ divides $p_j$. This implies that $\dim(P_{\T})=|S_1|+\dotsm+|S_m|-m$. Now suppose that all vertex representing root-to-leaf paths have the same length $m$. Then the number of symbols which do not divide $p_j$ is $|S_1| + \dotsm + |S_{m}| - m$ where $j$ is a vertex representing root-to-leaf path. If this number is equal to $\dim(P_{\T})$, then $P_{\T}$ is simple.
\end{proof}
\noindent Remark that the equality $|S_1|+\dotsm + |S_m| - m = \dim(P_{\T})$ holds for all models from Proposition~\ref{prop:2Dmultitrees} and Proposition~\ref{prop:3dmultitrees}.
\begin{example}\label{ex: vertexrepresentingleaf}
The multinomial staged tree $\T_{(A)_4}$ in Figure~\ref{fig: 2Dmultinomialtree} for the trapezoidal pyramid, does not satisfy Definition \ref{def: minimaltreerepresentation} (1). However when $b=1$, we can find such a balanced multinomial staged tree representation for this polytope, it is shown in Figure~\ref{fig:exVertexLeaf} (left).
This tree $\T$ and $\T_{(A)_4}$ represent the same model because their minimal Horn matrices are equal.
When $a=b=d=1$, the tree and its polytope are in Figure~\ref{fig:exVertexLeaf} (center) and (right).
There are five vertex representing root-to-leaf paths namely 1, 3, 4, 5, and 6 and thus $a_1, a_3, a_4,a_5$ and $a_6$ are the vertices of the trapezoidal pyramid. In particular $a_2 =\frac{1}{2} a_1 + \frac{1}{2} a_3.$ Hence $P_{\T}$ has property $(\star)$. Moreover, $P_{\T}$ is not simple by Lemma \ref{lem: simpleness of PT}, since not all root-to-leaf paths have the length 2.

\end{example}

\begin{figure}[H]
\centering
\begin{tikzpicture}[thick,scale=0.8]
\tikzstyle{every node}=[font=\footnotesize]
\renewcommand{\xx}{1.5}
\renewcommand{\yy}{0.7}

\node[circle, draw, fill=green!0, inner sep=2pt, minimum width=2pt] (d1) at (2*\xx,12.5*\yy) {};
\node[circle, draw, fill=green!0, inner sep=2pt, minimum width=2pt] (c1) at (3*\xx,16*\yy) {};
\node[circle, draw, fill=green!0, inner sep=2pt, minimum width=2pt] (c2) at (3*\xx,9*\yy) {};
\node[circle, draw, fill=green!0, inner sep=2pt, minimum width=2pt] (cc) at (3*\xx,12.5*\yy) {};
\node[circle, draw, fill=green!0, inner sep=2pt, minimum width=2pt] (b1) at (4*\xx,17*\yy+0.23) {};
\node[circle, draw, fill=green!0, inner sep=2pt, minimum width=2pt] (b2) at (4*\xx,15*\yy-0.23) {};
\node (b*) at (4*\xx,16*\yy) {$\vdots$};
\node[circle, draw, fill=green!0, inner sep=2pt, minimum width=2pt] (bb1) at (4*\xx,13.5*\yy+0.23) {};
\node (bb*) at (4*\xx,12.5*\yy) {$\vdots$};
\node[circle, draw, fill=green!0, inner sep=2pt, minimum width=2pt] (bb2) at (4*\xx,11.5*\yy-0.23) {};

\draw[->] (d1) -- node [midway,sloped,above] {$s_0$} (c1);
\draw[->] (d1) -- node [midway,sloped,below] {$s_2$} (c2);
\draw[->] (d1) -- node [midway,sloped,above] {$s_1$} (cc);

\draw[->] (cc) -- node [midway,sloped,above]  {$s_3^{a}$}  (bb1);
\draw[->] (cc) -- node [midway,sloped,below] {$s_4^{a}$} (bb2);

\draw[->] (c1) -- node [midway,sloped,above] {$s_3^{a+d}$} (b1) ;
\draw[->] (c1) -- node [midway,sloped,below] {$s_4^{a+d}$} (b2);
\end{tikzpicture}
\qquad
\begin{tikzpicture}[thick,scale=0.8]
\tikzstyle{every node}=[font=\footnotesize]
\renewcommand{\xx}{1.5}
\renewcommand{\yy}{0.7}

\node[circle, draw, fill=green!0, inner sep=2pt, minimum width=2pt] (d1) at (2*\xx,12.5*\yy) {};
\node[circle, draw, fill=green!0, inner sep=2pt, minimum width=2pt] (c1) at (3*\xx,16*\yy) {};
\node[circle, draw, fill=green!0, inner sep=2pt, minimum width=2pt] (c2) at (3*\xx,9*\yy) {};
\node [right] (c2x) at (3.1*\xx,9*\yy) {\ $6$};
\node[circle, draw, fill=green!0, inner sep=2pt, minimum width=2pt] (cc) at (3*\xx,12.5*\yy) {};
\node[circle, draw, fill=green!0, inner sep=2pt, minimum width=2pt] (b1) at (4*\xx,17*\yy+0.23) {};
\node [right] (b1x) at (4.1*\xx,17*\yy+0.23) {\ $1$};
\node[circle, draw, fill=green!0, inner sep=2pt, minimum width=2pt] (b2) at (4*\xx,15*\yy-0.23) {};
\node [right] (b2x) at (4.1*\xx,15*\yy-0.23) {\ $3$};
\node[circle, draw, fill=green!0, inner sep=2pt, minimum width=2pt] (br) at (4*\xx,16*\yy) {};
\node [right] (b*) at (4.1*\xx,16*\yy) {\ $2$};
\node[circle, draw, fill=green!0, inner sep=2pt, minimum width=2pt] (bb1) at (4*\xx,13.5*\yy+0.23) {};
\node [right] (bb1x) at (4.1*\xx,13.5*\yy+0.23) {\ $4$};
\node[circle, draw, fill=green!0, inner sep=2pt, minimum width=2pt] (bb2) at (4*\xx,11.5*\yy-0.23) {};
\node[right] (bb2x) at (4.1*\xx,11.5*\yy-0.23) {\ $5$};

\draw[->,BrickRed] (d1) -- node [midway,sloped,above] {$s_0$} (c1);
\draw[->] (d1) -- node [midway,sloped,below] {$s_2$} (c2);
\draw[->] (d1) -- node [midway,sloped,above] {$s_1$} (cc);
\draw[->] (cc) -- node [midway,sloped,above]  {$s_3$}  (bb1);
\draw[->] (cc) -- node [midway,sloped,below] {$s_4$} (bb2);
\draw[->] (c1) -- node [midway,sloped,above] {$s_3^{2}$} (b1) ;
\draw[->] (c1) -- node [midway,sloped,below] {$s_4^{2}$} (b2);
\draw[->,BrickRed] (c1) -- node [midway,sloped,above] {$\phantom{++++}2s_3s_4$} (br);
\end{tikzpicture}
\qquad
\begin{tikzpicture}[baseline=1]
\tikzstyle{every node}=[font=\footnotesize]
\begin{scope}[scale=2]
\draw (1.5,0.75) node[draw,circle,inner sep=2pt,fill, label={[xshift=0cm,yshift=-0.8cm]$m_3 = $(2,0,0)}](5){};
\draw (0,0.75) node[draw,circle,inner sep=2pt,fill, label={[xshift=-0.3cm,yshift=-0.8cm]$m_1 = $(0,0,0)}](4){};

\draw (0,1.5) node[draw,circle,inner sep=2pt,fill, label={[xshift=0cm,yshift=0.1cm]$m_6$=(0,0,1)}](6){};
\draw (0.3,1) node[draw,circle,inner sep=2pt,fill, label={[xshift=-1.4cm,yshift=-0.1cm]$m_4 =$(0,1,0)}](1){};
\draw (0.75,1) node[draw,circle,inner sep=2pt,fill, label ={[xshift=1cm,yshift=0cm]$m_5 = $(1,1,0)}](2){};
\draw[-] (1) edge[thick,dashed] (4);
\draw[-] (4) edge[thick] (5);
\draw[-] (2) edge[thick,dashed] (1);
\draw[-] (2) edge[thick,dashed] (5);
\draw[-] (1) edge[thick,dashed] (6);
\draw[-] (2) edge[thick,dashed] (6);
\draw[-] (4) edge[thick] (6);
\draw[-] (5) edge[thick] (6);
\draw (0.75,0.75) node[draw,circle,inner sep=2pt,fill=BrickRed, label={[xshift=-0.2cm,yshift=-0.8cm]$m_2 = $(1,0,0)}](11){};
\end{scope}
 \end{tikzpicture} 
 \caption{Multinomial staged tree representation of the non-simple trapezoidal pyramid, $b=1$.}
 \label{fig:exVertexLeaf}
 \end{figure}
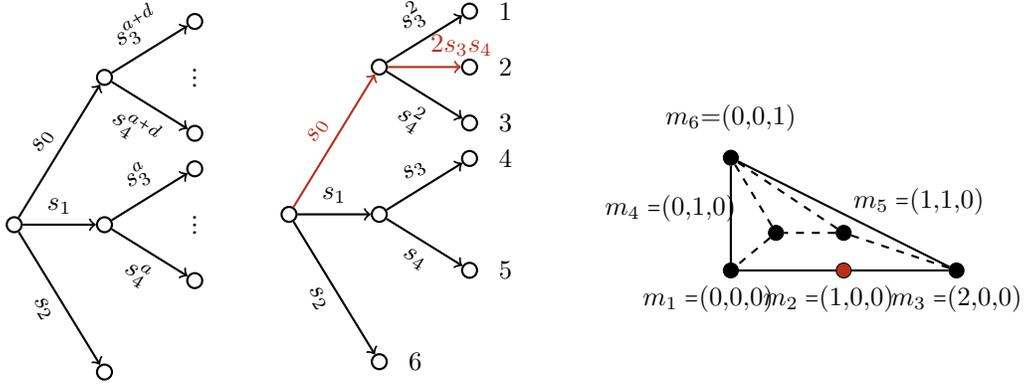
\noindent Furthermore,  the minimal Horn matrix for this example with $b=1$ (below left) coincides with $H_{(\T,\LL)}$. As mentioned also in Section \ref{subsec:nonsimple}, the primitive collections $\{n_1,n_3,n_4\},\{n_2,n_3,n_5\}$ do not offer an explanation for the negative part of the minimal Horn matrix, however the stages $\{s_0,s_1,s_2\}$,$\{s_3,s_4\}$ do.
 \begin{equation*}
\kbordermatrix{
    &m_1 & m_2 & m_3 & m_4 & m_5 & m_6\\
s_0=h_5 & 1 & 1 & 1 & 0 & 0 & 0\\
s_1=h_2 & 0 & 0 & 0 & 1 & 1 & 0\\
s_2=h_3 & 0 & 0 & 0 & 0 & 0 & 1\\
s_3=h_4 & 2 & 1 & 0 & 1 & 0 & 0\\
s_4=h_1 & 0 & 1 & 2 & 0 & 1 & 0\\
-(s_0+s_1+s_2) & -1 & -1 & -1 & -1 & -1 & -1\\
-(s_3+s_4) &-2 & -2 &-2 &-1 &-1  &0
}
\end{equation*}
\begin{equation*}
 \kbordermatrix{
    &m_1 & m_2 & m_3 & m_4 & m_5 & m_6 & m_7 & m_8 & m_9 & m_{10}\\
s_1 = h_3 & 0 & 0 & 0 & 0  & 0 & 0 & 0 & 0 & 0 & 1\\
s_2 = h_2& 0 & 0 & 0 & 0  & 2 & 2 & 1 & 1 & 1 & 0\\
s_3 = h_5 & 2 & 2 & 2 & 2  & 0 & 0 & 1 & 1 & 1 & 0\\
s_4 =h_4 & 3 & 2 & 1 & 0  & 1 & 0 & 2 & 1 & 0 & 0\\
s_5 =h_1 & 0 & 1 & 2 & 3  & 0 & 1 & 0 & 1 & 2 & 0\\
-(s_0 +s_1) & -1 & -1 & -1 & -1  & -1 & -1 & -1 & -1 & -1 & -1\\
s_0-(s_2 + s_3) &-1 & -1 & -1 & -1 & -1 & -1 & -1 & -1 & -1 & 0 \\
(s_4 + s_5) &-3 & -3 & -3 & -3  & -1 & -1 & -3 & -2 & -2 & 0
}
\end{equation*}
\noindent On the other hand we observe that there exists no multinomial staged tree representation for $b=2$, $a=d=1$ fitting Definition \ref{def: minimaltreerepresentation} by looking at the lattice distance matrix seen in the positive part of its minimal Horn matrix (above right). This matrix can also be obtained by applying Corollary \ref{cor: Horn matrix for multinomial trees} to $\T_{(A)_4}$ and performing the row operations explained in \cite[Lemma 3]{Duarte2020} eliminating the row $s_0 = h_6$. This demonstrates how multinomial staged trees provide a wider understanding for the negative part of the Horn matrix. \\

\noindent For simple polytopes $P_{\T}$ with property $(\star)$ we show that the stages coincide with the primitive colections of $\Sigma_{P_{\T}}$.
\begin{theorem}\label{thm: simplepolytopeprimitivecollections}
Let $P_{\T} \subset \mathbb{R}^d$ be a simple polytope with property $(\star)$. Then the primitive collections of the simplicial normal fan $\Sigma_{P_{\T}}$ are represented by the stages $S_1, \ldots, S_m$. 
\end{theorem}
\begin{proof}
By Definition \ref{def: minimaltreerepresentation}(1), the symbols of the stages represent the facets of $P_{\T}$. Let now $j$ be a vertex representing root-to-leaf path. Recall by the proof of Lemma \ref{lem: simpleness of PT}, the maximal cone associated to $a_j$ is generated by the normal vectors (1-faces) associated to $\bigcup_{l=1}^m S_l \backslash s_{i_l}$ for some $s_{i_l} \in S_l$ and where $\prod_{l=1}^m s_{i_l}$ divides $p_j$. Since any intersection of two cones in $\Sigma_{P_{\T}}$ is also a cone in $\Sigma_{P_{\T}}$, we obtain that $\bigcup_{l=1}^m S_l \backslash S'_l$ for all $S'_l \subseteq S_l$ with $|S'_l| \geq 1$ is a cone of $\Sigma_{P_{\T}}$. By Definition \ref{def: primitive collection}, since $\Sigma_{P_{\T}}$ is simplicial, a primitive collection is a set of 1-faces which does not generate a cone itself but any proper subset does. This concludes that the partition $S_1, \ldots. S_m$ are the primitive collections of $\Sigma_{P_{\T}}$. 
\end{proof}
\noindent The following corollary gives an affirmative answer to Question~\ref{q2}.
\begin{corollary}\label{cor: primitivecollections and Horn matrix}
Let $P_{\T} \subset \mathbb{R}^d$ be a simple polytope with property $(\star)$. Then the negative rows are given by the primitive collections of $\Sigma_{P_{\T}}$, i.e.\ $H_{(\T,\LL)}=M_{\A, \Sigma_{P_{\T}}}$.
\end{corollary}
\begin{proof}
It follows from Corollary \ref{cor: Horn matrix for multinomial trees} and Theorem \ref{thm: simplepolytopeprimitivecollections}.
\end{proof}

\begin{example}\label{ex: 2D minimal trees}

The multinomial staged trees $\T_{b\Delta_2}$ and $\T_{a,b,d}$ satisfy Definition \ref{def: minimaltreerepresentation}(1) for the simplex and trapezoid ($a,b,d>0$) respectively. That means the facets of the polytopes are in one-to-one correspondence with the symbols in the stages. Moreover $P_{\T_{b\Delta_2}}$ and $P_{\T_{a,b,d}}$ are simple polytopes. Hence by Theorem \ref{thm: simplepolytopeprimitivecollections} we obtain that the primitive collections are given by the partition of the stages. For the simplex $P_{\T_{b\Delta_2}} \cong a\Delta_{2}$ we have only one primitive collection $\{s_0,s_1,s_2\}$. Similarly, for $P_{\T_{a,b,d}}\cong T_{a,b,d}$ we have the partition of the stages as $\{s_0,s_1\}$ and $\{s_2,s_3\}$, which correspond exactly to the  primitive collections obtained in Theorem \ref{thm: 2D Horn matrix}.  
\end{example}

\begin{figure}[H]
\centering
\begin{tikzpicture}[scale=1,thick]
\tikzstyle{every node}=[font=\footnotesize]
\renewcommand{\xx}{1.1}
\renewcommand{\yy}{0.4}

\node[circle, draw, fill=green!0, inner sep=2pt, minimum width=2pt]  (d1) at (2*\xx,12.5*\yy) {};
\node[circle, draw, fill=green!0, inner sep=2pt, minimum width=2pt]  (c1) at (3*\xx,16*\yy) {};
\node[circle, draw, fill=green!0, inner sep=2pt, minimum width=2pt]  (c2) at (3*\xx,9*\yy) {};

\node[circle, draw, fill=green!0, inner sep=2pt, minimum width=2pt]  (cc) at (3*\xx,12.5*\yy) {};

\node[circle, draw, fill=green!0, inner sep=2pt, minimum width=2pt]  (b1) at (4*\xx,17*\yy) {};
\node [right] (b1x) at (4.1*\xx,17*\yy) {\ $1$};
\node[circle, draw, fill=green!0, inner sep=2pt, minimum width=2pt]  (b2) at (4*\xx,15*\yy) {};
\node [right] (b2x) at (4.1*\xx,15*\yy) {\ $2$};
\node (b*) at (4*\xx,16*\yy) {$\vdots$};

\node[circle, draw, fill=green!0, inner sep=2pt, minimum width=2pt]  (b3) at (4*\xx,10*\yy) {};
\node [right] (b3x) at (4.1*\xx,10*\yy) {\ $5$};
\node[circle, draw, fill=green!0, inner sep=2pt, minimum width=2pt]  (b4) at (4*\xx,8*\yy) {};
\node [right] (b4x) at (4.1*\xx,8*\yy) {\ $6$};
\node (b**) at (4*\xx,9*\yy) {$\vdots$};

\draw[->] (c2) -- node [midway,sloped,above] {$s_3^{a'}$}  (b3);
\draw[->] (c2) -- node [midway,sloped,below]  {$s_4^{a'}$} (b4);

\node[circle, draw, fill=green!0, inner sep=2pt, minimum width=2pt]  (bb1) at (4*\xx,13.5*\yy) {};
\node [right] (bb1x) at (4.1*\xx,13.5*\yy) {\ $3$};

\node (bb*) at (4*\xx,12.5*\yy) {$\vdots$};
\node[circle, draw, fill=green!0, inner sep=2pt, minimum width=2pt]  (bb2) at (4*\xx,11.5*\yy) {};
\node [right] (bb2x) at (4.1*\xx,11.5*\yy) {\ $4$};

\draw[->] (d1) -- node [midway,sloped,above]  {$s_0$} (c1);
\draw[->] (d1) -- node [midway,sloped,below]  {$s_2$} (c2);
\draw[->] (d1) -- node [midway,sloped,above]  {$s_1$} (cc);

\draw[->] (cc) -- node [midway,sloped,above]  {$s_3^{a}$}  (bb1);
\draw[->] (cc) -- node [midway,sloped,below] {$s_4^{a}$} (bb2);

\draw[->] (c1) -- node [midway,sloped,above]  {$s_3^{a+d}$} (b1) ;
\draw[->] (c1) -- node [midway,sloped,below]  {$s_4^{a+d}$} (b2);
\end{tikzpicture}
\qquad
\begin{tikzpicture}[scale=3.5]
\tikzstyle{every node}=[font=\footnotesize]
\draw (1.2,0.75) node[draw,circle,inner sep=2pt,fill](5){};
\node [right] (m2) at (1.2,0.75) {$m_2 =(a+d,0,0)$};

\draw (0,0.75) node[draw,circle,inner sep=2pt,fill](4){};
\node [left] (m1) at  (0,0.75) {$m_1 = (0,0,0)$};

\draw (0,1.35) node[draw,circle,inner sep=2pt,fill](6){};
\node [left] (m5) at  (0,1.35) {$m_5 = (0,0,1)$};

\draw (0.3,1) node[draw,circle,inner sep=2pt,fill](1){};
\node [left] (m3) at  (0,1) {$m_3 = (0,1,0)$};

\draw (0.75,1) node[draw,circle,inner sep=2pt,fill](2){};
\node [right] (m6) at  (1,1) {$m_4 = (a,1,0)$};

\draw (0.6,1.35) node[draw,circle,inner sep=2pt,fill](7){};
\node [right] (m6) at  (0.6,1.35) {$m_6 = (a',0,1)$};

  \draw[-] (1) edge[thick,dashed] (4);
  \draw[-] (4) edge[thick] (5);
  \draw[-] (2) edge[thick,dashed] (1);
  \draw[-] (2) edge[thick,dashed] (5);
  \draw[-] (1) edge[thick,dashed] (6);
  \draw[-] (2) edge[thick,dashed] (7);
  \draw[-] (6) edge[thick] (7);
  \draw[-] (4) edge[thick] (6);
  \draw[-] (5) edge[thick] (7);
  \end{tikzpicture}
 \captionof{figure}{Multinomial staged tree representation of the simple trapezoidal wedge with $b=1$ considered in
 Example~\ref{ex: simple polytope minimal tree primitive collections}.
 }\label{fig: minimal tree for trapezoid with a line segment}
 \end{figure}
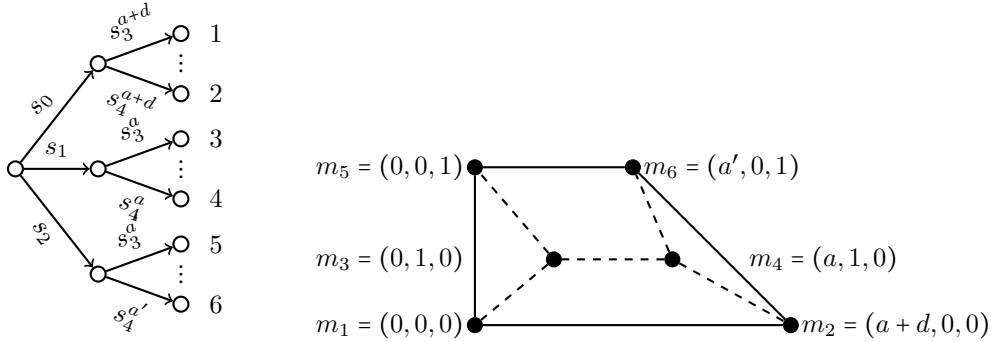
\begin{example}\label{ex: simple polytope minimal tree primitive collections}
Let us consider the balanced multinomial staged tree representation, satisfying Definition \ref{def: minimaltreerepresentation}, of the trapezoidal wedge from Table \ref{table:3Dtrapezoids} (A) with $b=1$ seen in Figure \ref{fig: minimal tree for trapezoid with a line segment}.
This tree representation encodes the same model as the tree $\T_{(A)_3}$ in Figure~\ref{fig: 2Dmultinomialtree}, because they have the same minimal Horn matrix. In particular we observe by Lemma \ref{lem: simpleness of PT} that $P_{\T}$ is simple. By Theorem \ref{thm: simplepolytopeprimitivecollections}, the primitive collections are represented by the partition of the stages: $\{n_2,n_3,n_5\} = \{s_0, s_1, s_2\}$ and $\{n_1,n_4\}= \{s_3,s_4\}$. By Corollary \ref{cor: primitivecollections and Horn matrix} the negative part of the minimal Horn matrix $a=b=d=a'=1$ (top), is explained by primitive collections. In Section~\ref{section:fewerfacets} we saw that even for simple polytopes, the negative part of the Horn matrix is not always explained by the primitive collections. 
\noindent From the perspective of staged trees,  for $a'=a=d=1,\,b=2$, this polytope can be represented by
$\T_{(A)_3}$. After minimising the Horn matrix  $H_{(\T_{(A)_3},\LL_{(A)_3})}$, constructed by Corollary \ref{cor: Horn matrix for multinomial trees}, we obtain the matrix (bottom) whose positive part is the lattice distance matrix of $P_{\T_{(A)_3}}$ (see Table \ref{table:minimal}). However a simple computation shows that there exists no tree $(\T,\LL)$ such that the positive part of $H_{(\T,\LL)}$ is the lattice distance matrix of $P_{\T_{(A)_3}}$ i.e.\ satisfying Definition \ref{def: minimaltreerepresentation} (1).
\begin{equation*}
\kbordermatrix{
    &m_1 & m_2 & m_3 & m_4 & m_5 & m_6 & m_7 \\
h_5=s_0 & 1 & 1 & 0 & 0 & 0 & 0 & 1\\
h_2=s_1 & 0 & 0 & 1 & 1 & 0 & 0 & 0\\
h_3=s_2 & 0 & 0 & 0 & 0 & 1 & 1 & 0\\
h_4=s_3 & 2 & 0 & 1 & 0 & 1 & 0 & 1\\
h_1=s_4 & 0 & 2 & 0 & 1 & 0 & 1 & 1\\
-(s_0+s_1+s_2) & -1 & -1 & -1 & -1 & -1 & -1 & -1\\
-(s_3+s_4) & -2 & -2 & -1 &-1 & -1  &-1 & -2 }
\end{equation*}
\begin{equation*}
  \kbordermatrix{
    &m_1 & m_2 & m_3 & m_4 & m_5 & m_6 & m_7 & m_8 & m_9 & m_{10} & m_{11}\\
h_3=s_1 & 0 & 0 & 0 & 0 & 0 & 0 & 0 & 0 & 0 & 1& 1\\
h_5=s_2 & 2 & 2 & 2 & 2 & 1 & 1 & 1 & 0 & 0 & 0 & 0\\
h_2=s_3 & 0 & 0 & 0 & 0 & 1 & 1 & 1 & 2 & 2 & 0 & 0\\
h_4=s_4 & 3 & 2 & 1 & 0 & 2 & 1 & 0 & 1 & 0 & 1 & 0\\
h_1=s_5 & 0 & 1 & 2 & 3 & 0 & 1 & 2 & 0 & 1 & 0 & 1\\
-s_0 -s_1 & -1 & -1 & -1 & -1 & -1 & -1 & -1 & -1 & -1 & -1 & -1\\
s_0-s_2 - s_3 &-1 & -1 & -1 & -1 & -1 & -1 & -1 & -1 & -1 & 0 & 0 \\
-s_4 - s_5 &-3 & -3 & -3 & -3 & -2 & -2 & -2 & -1 & -1 & -1 & -1
}
\end{equation*}
\end{example}

\section*{Acknowledgements}
Eliana Duarte was supported by the Deutsche Forschungsgemeinschaft DFG under grant 314838170, GRK 2297 MathCoRe,  by the FCT grant 2020.01933.CEECIND, and partially supported by CMUP under the FCT grant UIDB/00144/2020. 
Computations using the software system Sagemath \cite{sagemath} and Macaulay2 \cite{M2} were crucial for the development of this paper. We thank Bernd Sturmfels for introducing us to each other and for encouraging us to work on this project.

\bibliographystyle{amsplain}      

\bibliography{refs}   

\noindent
\footnotesize {\bf Authors' addresses:}

\noindent Isobel Davies \\
Otto-von-Guericke Universit\"at Magdeburg, Universit\"atsplatz 2, 39106 Magdeburg, Germany.\\
\hfill {\tt isobel.davies@ovgu.de}
\vspace{\baselineskip}

\noindent Eliana Duarte \\
Universidade do Porto, Rua do Campo Alegre 687, 4169-007 Porto, Portugal.\\
\hfill {\tt eliana.gelvez@fc.up.pt}
\vspace{\baselineskip}

\noindent Irem Portakal \\
Technische Universität München, Lehrstuhl für Mathematische Statistik 85748 Garching b. München, Boltzmannstr. 3.,  Germany.\\
\hfill {\tt 
mail@irem-portakal.de}
\vspace{\baselineskip}

\noindent Miruna-\c Stefana Sorea\\
SISSA, via Bonomea, 265 - 34136 Trieste, Italy and RCMA Lucian Blaga University Sibiu, Bd-ul Victoriei nr.10, Sibiu, 550024,  Romania.\\
\hfill {\tt msorea@sissa.it}

\vskip5cm

\appendix
\section{Appendix A} \label{ap:table}
\begin{table}[H]
\begin{center}
\begin{tabular}{ |m{6.5cm}|m{10cm}|}
\hline
\textbf{Name of subfamily}&\textbf{Columns of minimal Horn matrix}\\ \hline
 \multicolumn{2}{|c|}{\textbf{ (A) Prismatoids with trapezoidal base} $a>0,b>0,d>0,l>0$} \\ \hline
 Trapezoidal frusta&$(h_1,h_2,h_3,h_4,h_5,h_6,-(h_1+h_4),-(h_2+h_5),-(h_3+h_6))$\\ \hline
  Triangle top&$(h_1,h_2,h_3,h_4,h_5,h_6,-(h_1+h_4),-(h_2+h_5),-(h_3+h_6))$\\ \hline
  Trapezoidal wedges with $b\neq 1$&$(h_1,h_2,h_3,h_4,h_5,-(h_1+h_4),-(h_2+h_5-h_6),-(h_3+h_6))$\\ \hline
  Trapezoidal wedges with $b=1$&$(h_1,h_2,h_3,h_4,h_5,-(h_1+h_4),-(h_2+h_3+h_5))$ \\ \hline
  Trapezoidal pyramids with $b\neq 1$&$(h_1,h_2,h_3,h_4,h_5,-(h_1+h_4),-(h_2+h_5-h_6),-(h_3+h_6))$\\ \hline
  Trapezoidal pyramids with $b=1$&$(h_1,h_2,h_3,h_4,h_5,-(h_1+h_4),-(h_2+h_3+h_5))$ \\ \hline
 \multicolumn{2}{|c|}{\textbf{(B) Prismatoids with tensor product base $a>0$, $b>0$, $d=0$, $l>0$} }\\ \hline
 General tensor product frusta&$(h_1,h_2,h_3,h_4,h_5,h_6,-(h_1+h_4),-(h_2+h_5),-(h_3+h_6))$\\ \hline
 Tensor product frusta with $a'=a$&$(h_1,h_2,h_3,h_4,h_5,h_6,-(h_2+h_5),-(h_1+h_3+h_4+h_6))$\\ \hline
 Tensor product frusta with $b'=b$&$(h_1,h_2,h_3,h_4,h_5,h_6,-(h_1+h_4),-(h_2+h_3+h_5+h_6))$\\ \hline
 3D Tensor Product&$(h_1,h_2,h_3,h_4,h_5,h_6,-(h_1+h_2+h_3+h_4+h_5+h_6))$\\ \hline
 Tensor product frusta with $(a',a)=\lambda(b',b)$ or $\mu(a',a)=(b',b)$ ($\lambda,\mu\geq1$) &$(h_1,h_2,h_3,h_4,h_5,h_6,-(h_1+h_2+h_4+h_5),-(h_3+h_6))$\\ \hline
 Tensor product wedges ($a'=0$) with $a\neq1$&$(h_1,h_2,h_3,h_4,h_5,-(h_1+h_4-h_6),-(h_2+h_5),-(h_3+h_6))$\\ \hline
 Tensor product wedges ($a'=0$) with $a=1$ &$(h_1,h_2,h_3,h_4,h_5,-(h_2+h_5),-(h_1+h_3+h_4))$\\ \hline 
Tensor product wedges ($b'=0$) with $b\neq1$&$(h_1,h_2,h_3,h_4,h_5,-(h_1+h_4),-(h_2+h_5-h_6),-(h_3+h_6))$\\ \hline
Tensor product wedges ($b'=0$) with $b=1$ &$(h_1,h_2,h_3,h_4,h_5,-(h_1+h_4),-(h_2+h_3+h_5))$\\ \hline
Tensor product pyramids&$(h_1,h_2,h_3,h_4,h_5,-(h_1+h_2+h_4+h_5-h_6),-(h_3+h_6))$\\ \hline
 \multicolumn{2}{|c|}{\textbf{(C) Prismatoids with triangular base $a=0$, $b>0$, $d>0$, $l>0$} }\\ \hline
Triangular frusta ($b'\neq b$) with $d\neq 1$&$(h_1,h_2,h_3,h_4,h_6,-(h_1+h_4-h_5),-(h_2+h_5),-(h_3+h_6))$\\ \hline
Triangular prism ($b'=b$) with $d\neq 1$&$(h_1,h_2,h_3,h_4,h_6,-(h_1+h_4-h_5),-(h_2+h_3+h_5+h_6))$\\ \hline
Triangular frusta ($b'\neq b$) with $d=1$&$(h_1,h_2,h_3,h_4,h_6,-(h_1+h_2+h_4),-(h_3+h_6))$\\ \hline 
Triangular prism ($b'=b$) with $d=1$&$(h_1,h_2,h_3,h_4,h_6,-(h_1+h_2+h_3+h_4+h_6))$\\ \hline 
Triangular based pyramid with $b\neq1$ and $d\neq1$&$(h_1,h_2,h_3,h_4,-(h_1+h_4+-h_5),-(h_2+h_5-h_6),-(h_3+h_6))$\\ \hline
Triangular based pyramid with $b=1$ and $d\neq1$&$(h_1,h_2,h_3,h_4,-(h_1+h_4+-h_5),-(h_2+h_3+h_5))$\\ \hline
Triangular based pyramid with $b\neq1$ and $d=1$&$(h_1,h_2,h_3,h_4,-(h_1+h_2+h_4-h_6),-(h_3+h_6))$\\ \hline
3D simplex&$(h_1,h_2,h_3,h_4,-(h_1+h_2+h_3+h_4))$\\ \hline 
\end{tabular}
\caption{Minimal Horn matrices for pairs in $\mathcal{P}$. The right column in the table contains the columns of the Horn matrix in terms of the lattice distance functions of each polytope.} \label{table:minimal}
\end{center} 
\end{table} 
\section{Appendix B} \label{ap:technical}
The next lemma gives several equations that hold between a point in $\Theta_{\T}$  and its image under $\phi_{\T}$.
We use Lemma~\ref{lem: identifiability} parts $(2)$ and $(4)$, to define the ideal of model invariants for a multinomial staged tree model.
\begin{lemma} \label{lem: identifiability}
Let $\M_{(\T,\LL)}$ be a multinomial staged tree model where $\T=(V,E)$. Fix $v\in \widetilde{V}$ and suppose $\mathrm{im}(\LL_{v})=f_{\ell,a}$ for 
$\ell\in [m]$, $a \in \mathbb{Z}_{\geq 1}$. Set $i_1, \ldots, i_{|I_{\ell}|}$ to be a fixed ordering of the elements in $I_{\ell}$. Let $(p_1,\ldots, p_n)\in \M_{(\T,\LL)}$ and $\theta=(\theta_i)_{i\in I}\in \Theta_{\T}$ be such that
    $\phi_{\T}(\theta)=(p_1,\ldots, p_n)$.
\begin{itemize}
    \item[(1)]For each $K\in \mathbb{N}^{|I_{\ell}|}$ with $|K|= a$,
    \[
     \frac{p_{[v(K)]}}{p_{[v]}} ={a \choose k_{i_{1}},\ldots, k_{i_{|I_\ell|}}}\prod_{q=1}^{|I_{\ell}|}\theta_{i_q}^{k_{i_q}}. \]
     \item[(2)]   Let $K^1,K^2,K^3,K^4\in \mathbb{N}^{|I_\ell|}$ with $|K^1|=|K^2|=|K^3|=|K^4|=a$, be such that
     $K^1+K^2=K^3+K^4$.  Define $C_{(K^1,K^2)}:= {a \choose K^1}{a \choose K^2}$ and similarly for $C_{(K^3,K^4)}$. Then
     \[C_{(K^3,K^4)}p_{[v(K^1)]}p_{[v(K^2)]}- C_{(K^1,K^2)}p_{[v(K^3)]}p_{[v(K^4)]}=0.\]
     \item[(3)]  For each $i_{q}\in I_{\ell}$, $1\leq q \leq |I_{\ell}|$            
     \begin{align*}
         \theta_{i_q} = \frac{\sum_{|K|=a, k_{i_q}\geq 1}k_{i_q}p_{[v(K)]}}{ap_{[v]}}. 
     \end{align*}
     \item[(4)] Let $w \in V$ and $\mathrm{im}(\LL_{w})= f_{\ell,b}$. For all $i_{q}\in I_{\ell}$:
     \[
     bp_{[w]}\left(\sum_{|K|=a, k_{i_q}\geq 1} k_{i_q}p_{[v(K)]}\right)-ap_{[v]}\left(\sum_{|K'|=b,k_{i_q}'\geq 1}k_{i_q}'p_{[w(K')]}\right)=0.
     \]
\end{itemize}
\end{lemma}
\begin{proof}
$(1)$ Since $\M_{(\T,\LL)}$
is a probability tree, the transition probability from $v$ to $v(K)$ is the probability of arriving at $v(K)$ divided by the
probability of arriving at $v$, namely $p_{[v(K)]}/p_{[v]}$. By definition of $\M_{(\T,\LL)}$, and since $\LL_{v}(v\to v(K))={a \choose k_{i_1},\ldots, k_{i_{|I_{\ell}|}}}\prod_{\alpha=1}^{|I_\ell|}s_{i_\alpha}^{k_{i_\alpha}}$, this probability is exactly 
${a \choose k_{i_1},\ldots, k_{i_{|I_\ell|}}}\prod_{\alpha=1}^{|I_\ell|}\theta_{i_{\alpha}}^{k_{i_\alpha}}$. \newline
$(2)$ This  equality follows by direct substitution for the values from $(1)$ and by noting that the coefficients $C_{(K^1,K^2)},C_{(K^3,K^4)}$ are needed to achieve cancellation.

\noindent $(3)$ We start from the right-hand side, use $(1)$, the fact that $\sum_{i\in I_{\ell}}\theta_i = 1$, and simplification with multinomial
coefficients to arrive at $\theta_{i_{q,}}$:
\begin{align*}
  \frac{\sum_{|K|=a, k_{i_{q}}\geq 1}k_{i_{q}}p_{[v(K)]}}{ap_{[v]}} 
  &=
  \sum_{\substack{|K|=a \\
  k_{i_{q}}\geq 1}} \frac{k_{i_{q}}}{a}\frac{p_{[v(K)]}}{p_{[v]}} 
  = 
  \sum_{\substack{|K|=a \\
  k_{i_{q}}\geq 1}} \frac{k_{i_{q}}}{a}{a \choose k_{i_1},\ldots, k_{i_{|I_\ell|}}}\prod_{\alpha=1}^{|I_{\ell}|}\theta_{i_{\alpha}}^{k_{i_\alpha}}\\
  &= \sum_{\substack{|K|=a \\
  k_{i_{q}}\geq 1}} \frac{k_{i_{q}}a}{k_{i_{q}}a}{a-1 \choose k_{i_1},\ldots, k_{i_q}-1,\ldots, k_{i_{|I_{l}|}}}\theta_{i_{q}}\theta_{i_{q}}^{k_{i_q}-1} \prod_{\substack{\alpha=1\\ \alpha \neq q}}^{|I_{\ell}|}\theta_{i_\alpha}^{k_{i_\alpha}} \\
  &= \theta_{i_{q}}\sum_{|K|=a-1}{a-1 \choose k_{i_1},\ldots, k_{i_q},\ldots, k_{i_{|I_{l}|}}} \prod_{\alpha=1}^{|I_\ell|}\theta_{i_{\alpha}}^{k_{i_\alpha}} \\
  &= \theta_{i_q} (\sum_{\alpha=1}^{|I_{\ell}|}\theta_{i_{\alpha}})^{a-1}= \theta_{i_q}.
\end{align*}
$(4)$ Applying part $(3)$ to $i_{q}$ for $v$ and $w$ separately yields
\[ 
\theta_{i_q} = \frac{\sum_{|K|=a, k_{i_{q}}\geq 1}k_{i_{q}}p_{[v(K)]}}{ap_{[v]}} = \frac{\sum_{|K'|=b,k_{i_q}'\geq 1}k_{i_q }'p_{[w(K')]}}{bp_{[w]}}. 
\] After cross multiplication we get the desired equation in $(4)$.
\end{proof}

\subsection{Binary multinomial staged trees}
From this point on we assume  in $(\T,\LL)$ is a binary multinomial staged tree, and modify our notation according to this assumption.

\begin{lemma}\label{conj:toBeProven}
Let $(\T,\LL)$ be a binary multinomial staged tree where $\T=(V,E)$ and let  $v\in \widetilde{V}$ be such that $\mathrm{im}(\LL_v)=f_{\ell,a}$. 

Fix $k_0\in \{0,\ldots, a\}$.
The following equalities hold in $\mathbb{R}[P_j: j\in J]/I_{\M(\T,\LL)}$:
\begin{align}
    &{a \choose k_0+k}P_{[v(k_0,a-k_0)]} l_1^{k_0+k}l_2^{a-(k_0+k)}={a \choose k_0}P_{[v(k_0+k,a-(k_0+k))]} l_1^{k_0}l_2^{a-k_0} \;\; \label{eq1},
\end{align}
for $1\leq k\leq a-k_0$.
\begin{align}
    &{a \choose k_0-k}P_{[v(k_0,a-k_0)]} l_1^{k_0-k}l_2^{a-(k_0-k)}={a \choose k_0}P_{[v(k_0-k,a-(k_0-k))]} l_1^{k_0}l_2^{a-k_0} \;\; \label{eq2},
\end{align}
for $1\leq k\leq k_0$.
where $l_1:=\sum_{k=1}^a k P_{[v(k,a-k)]}$ and   $l_2:=\sum_{k=1}^a k P_{[v(a-k,k)]}$.
\end{lemma}
\noindent For the proof of Lemma \ref{conj:toBeProven}, we use Lemma \ref{lem:auxiliaryLemma}.
\begin{lemma}\label{lem:auxiliaryLemma}
Under the hypotheses from Lemma \ref{conj:toBeProven}, the following equality holds in $\mathbb{R}[P_j: j\in J]/I_{\M(\T,\LL)}$:
\begin{align}\label{eq:auxiliary1}
    (a-k_0) P_{[v(k_0,a-k_0)]} l_1=(k_0+1)P_{[v(k_0+1,a-(k_0+1))]} l_2.
\end{align}
\end{lemma}

\begin{proof}
By the part of the Definition~\ref{def:idealModelInvariants} 
involving $I_{\mathrm{vertices}}$, we see that the equality
$$
    {a \choose k_0\!+\!1}{a \choose k\!-\!1}P_{[v(k_0,a-k_0)]}P_{[v(k,a-k)]}=
    {a \choose k_0}{a \choose k} P_{[v(k_0+1,a-(k_0+1))]}P_{[v(k-1,a-(k-1))]}
    \label{eq:multibinomials}
$$
holds $\mathbb{R}[P_j: j\in J]/I_{\M(\T,\LL)}$.
Note that
\begin{align*}
    {a \choose k_0+1}{a \choose k-1} &= \frac{a!a!}{(k_0+1)!(a-k_0)!k!(a-(k-1))!} (a-k_0)k\; , \text{ and }\\
    {a \choose k_0}{a \choose k}&= \frac{a!a!}{(k_0+1)!(a-k_0)!k!(a-(k-1))!} (k_0+1)(a-(k-1)).
\end{align*}
Thus, we may cancel the constant $a!a! \;/\; (k_0+1)!(a-k_0)!k!(a-(k-1))! $ from the equality we started with in this proof, to obtain the simplified expression
\begin{align*}
    P_{[v(k_0,a-k_0)]}P_{[v(k,a-k)]} =\frac{(k_0+1)(a-(k-1))}{(a-k_0)k}P_{[v(k_0+1,a-(k_0+1))]}P_{[v(k-1,a-(k-1))]}.
\end{align*}
Using this identity and the definition of $l_1$ in Lemma~\ref{conj:toBeProven}, it follows that
\begin{align*}
    (a-k_0) P_{[v(k_0,a-k_0)]} l_1&=(a-k_0) P_{[v(k_0,a-k_0)]}\sum_{k=1}^a k P_{[v(k,a-k)]}\\
    &=\sum_{k=1}^a (a-k_0)k P_{[v(k_0,a-k_0)]} P_{[v(k,a-k)]}\\
    &=\sum_{k=1}^a (k_0+1)(a-(k-1))P_{[v(k_0+1,a-(k_0+1))]}P_{[v(k-1,a-(k-1))]}\\
    &=(k_0+1)P_{[v(k_0+1,a-(k_0+1))]} l_2.
\end{align*}
\end{proof}
\noindent We are now ready to prove Lemma \ref{conj:toBeProven}.
\begin{proof}
Let us first prove equality (\ref{eq1}).
We will do this by mathematical induction on $k$. 
First we show that (\ref{eq1}) holds for $k=1$:
\begin{align*}
    {a \choose k_0+1}P_{[v(k_0,a-k_0)]} l_1^{k_0+1}l_2^{a-(k_0+1)}
    &\stackrel{(\ref{eq:auxiliary1})}{=}{a \choose k_0+1}\frac{k_0+1}{a-k_0} P_{[v(k_0+1,a-(k_0+1))]} l_1^{k_0}l_2^{a-k_0} \\
    &= {a \choose k_0}P_{[v(k_0+1,a-(k_0+1))]} l_1^{k_0}l_2^{a-k_0}.
\end{align*}
Let us now suppose that (\ref{eq1}) holds for $k$, and prove it for $k+1$.
\begin{align*}
    &{a \choose k_0+k+1}P_{[v(k_0,a-k_0)]} l_1^{k_0+k+1}l_2^{a-(k_0+k+1)}\\
    &\qquad=\frac{a!(a-(k_0+k))}{(k_0+k+1)(k_0+k)!(a-(k_0+k))!}P_{[v(k_0,a-k_0)]} l_1^{k_0+k+1}l_2^{a-(k_0+k+1)}\\
    &\qquad={a \choose k_0+k}\frac{a-(k_0+k)}{k_0+k+1}P_{[v(k_0,a-k_0)]}l_1l_1^{k_0+k} l_2^{a-(k_0+k+1)} l_2^{-1}
\end{align*}
Using (\ref{eq1}) for $k$, we further simplify to
\begin{align*}
   & \textoverset{{hyp.(\ref{eq1})}}{=}& {a \choose k_0}P_{[v(k_0+k,a-(k_0+k))]} l_1^{k_0}l_2^{a-k_0}\frac{a-(k_0+k)}{k_0+k+1}l_1  l_2^{-1}\\
   & \textoverset{{(\ref{eq:auxiliary1})}}{=} & {a \choose k_0} \frac{a-(k_0+k)}{k_0+k+1} \frac{k_0+k+1}{a-(k_0+k)} P_{[v(k_0+k+1,a-(k_0+k+1))]}l_1^{k_0}l_2 l_2^{a-k_0} l_2^{-1}\\
    & \textoverset{}{=}& {a \choose k_0} P_{[v(k_0+k+1,a-(k_0+k+1))]} l_1^{k_0} l_2^{a-k_0}.
\end{align*}
Let us now prove equality (\ref{eq2}).
By (\ref{eq:auxiliary1}), we have
    \begin{align*}
        (a-k_0-1) P_{[v(k_0-1,a-(k_0-1))]} l_1=k_0 P_{[v(k_0,a-k_0)]} l_2.
        \end{align*}
    We will again use mathematical induction on $k$.
First we show that (\ref{eq2}) holds for $k=1$:
\begin{align*}
&{a \choose k_0-1}P_{[v(k_0,a-k_0)]} l_1^{k_0-1}l_2^{a-(k_0-1)}\\
&\qquad=
{a \choose k_0-1}\frac{a-(k_0-1)}{k_0} P_{[v(k_0-1,a-(k_0-1))]} l_1 l_2^{-1} l_1^{k_0-1}l_2^{a-(k_0-1)}\\
&\qquad={a \choose k_0} P_{[v(k_0-1,a-(k_0-1))]}  l_1^{k_0} l_2^{a-k_0}.
\end{align*}
Let us now suppose that (\ref{eq2}) holds for $k$, and prove it for $k+1$.
\begin{align*}
    & \textoverset{}{}&{a \choose k_0-(k+1)}P_{[v(k_0,a-k_0)]} l_1^{k_0-(k+1)}l_2^{a-(k_0-(k+1))}\\
    &\textoverset{}{=}&\!\!\!\!{a \choose k_0-k}\frac{k_0-k}{a-(k_0-(k+1))}P_{[v(k_0,a-k_0)]}l_1^{k_0-k} l_1^{-1} l_2^{a-(k_0-k)} l_2\\
    &\textoverset{hyp.(\ref{eq2})}{=} &{a \choose k_0}P_{[v(k_0-k,a-(k_0-k))]} l_1^{k_0}l_2^{a-k_0} \frac{k_0-k}{a-(k_0-(k+1))} l_1^{-1} l_2\\
    &\textoverset{(\ref{eq:auxiliary1})}{=}&{a \choose k_0} \frac{k_0-k}{a-(k_0-(k+1))}\frac{a-(k_0-(k+1))}{k_0-k} l_1 P_{[v(k_0-k-1,a-(k_0-k-1))]} l_1^{k_0}l_2^{a-k_0}l_1^{-1}\\
    &\textoverset{}{=}&{a \choose k_0}P_{[v(k_0-(k+1),a-(k_0-(k+1)))]}l_1^{k_0}l_2^{a-k_0}.
\end{align*}
\end{proof}

\begin{lemma} \label{lem:finally}
Under the same hypotheses as in Lemma~\ref{conj:toBeProven}, the following equalities hold in $\mathbb{R}[P_j: j\in J]/I_{\M(\T,\LL)}$:
\begin{equation}
  {a \choose k_0}\bigg (\frac{l_1}{a P[v]}\bigg )^{k_0}\bigg (\frac{l_2}{a P[v]}\bigg )^{a-k_0}=\frac{P_{[v(k_0,a-k_0)]}}{P_{[v]}}.
\end{equation}
\end{lemma}
\begin{proof}
First, note that $a P_{[v]}=(l_1+l_2)$. Indeed
\begin{align*}
l_1+l_2&=(P_{[v(1,a-1)]}+2P_{[v(2,a-2)]}+\dotsm+(a-1) P_{[v(a-1,1)]}+aP_{[v(a,0)]})\\
&\qquad+(aP_{[v(0,a)]}+(a-1)P_{[v(1,a-1)]}+
\dotsm +2P_{[v(a-2,2)]}+P_{[v(a-1,1)]})\\
&=a(P_{[v(a,0)]}+P_{[v(1,a-1)]}+P_{[v(2,a-2)]}+\dotsm+P_{[v(1,a-1)]}+P_{[v(0,a)]})=aP_{[v]}.
\end{align*}
Therefore we have $P_{[v(k_0,a-k_0)]} (aP_{[v]})^a=P_{[v(k_0,a-k_0)]} (l_1+l_2)^a$. 
Now, by Lemma \ref{conj:toBeProven}, we obtain:
\begin{align*}
P_{[v(k_0,a-k_0)]} (l_1+l_2)^a &=P_{[v(k_0,a-k_0)]} \sum_{k=0}^a {a \choose k}l_1^k l_2^{a-k}\\
&=\sum_{k=0}^a {a \choose k}P_{[v(k_0,a-k_0)]} l_1^k l_2^{a-k}\\
&=\sum_{k=1}^{k_0} {a \choose k_0-k}P_{[v(k_0,a-k_0)]} l_1^{k_0-k} l_2^{a-(k_0-k)}\\
&\qquad +{a \choose k_0}P_{[v(k_0,a-k_0)]} l_1^{k_0} l_2^{a-k_0}\\ 
&\qquad+\sum_{k=1}^{a-k_0}{a \choose k_0+k}P_{[v(k_0,a-k_0)]} l_1^{k_0+k} l_2^{a-(k_0+k)}\\
&={a \choose k_0}l_1^{k_0} l_2^{a-k_0}\left( \sum_{k=1}^{k_0} P_{[v(k_0-k,a-(k_0-k))]}+P_{[v(k_0,a-k_0)]}\right.\\
&\qquad \left. + \sum_{k=1}^{a-k_0} P_{[v(k_0+k,a-(k_0+k))]} \right)\\
&={a \choose k_0} l_1^{k_0} l_2^{a-k_0} P_{[v]}.
\end{align*}
\end{proof}

\end{document}